\documentclass[11pt,a4paper]{article}
\usepackage[margin=28mm,headheight=14pt]{geometry}
\usepackage[T1]{fontenc}
\usepackage[utf8]{inputenc}
\usepackage{lmodern}
\usepackage{tgpagella}
\usepackage{amsmath,amssymb,amsthm,mathtools,enumitem}
\usepackage{tikz}
\usetikzlibrary{arrows.meta,backgrounds,fit}
\usepackage[all]{xy}
\usepackage{xcolor,hyperref,fancyhdr,needspace,longtable,array}
\definecolor{linkblue}{HTML}{234F67}
\hypersetup{colorlinks=true,linkcolor=linkblue,citecolor=linkblue,urlcolor=linkblue,
linktoc=all,
pdftitle={Operadic Expression Trees, I. Noetherian forms of non-symmetric free operads},
pdfsubject={Operadic expression tree categories, reconstruction, and exact join decomposition},
pdfauthor={Zurab Janelidze},
pdfkeywords={operadic expression trees, prefix-suffix matching, free non-symmetric operads, factorization systems, noetherian forms, exact join decomposition}}
\setlist{itemsep=3pt,topsep=5pt}
\setlist[enumerate,1]{label=(\alph*),ref=\alph*,font=\normalfont}
\theoremstyle{plain}
\newtheorem{theorem}{Theorem}[section]
\newtheorem{lemma}[theorem]{Lemma}
\theoremstyle{definition}
\newtheorem{definition}[theorem]{Definition}
\newtheorem{axiom}{Axiom}
\theoremstyle{remark}
\newtheorem{remark}[theorem]{Remark}
\newtheorem{example}[theorem]{Example}
\newcommand{\C}{\mathcal C}
\newcommand{\Ocat}{\mathcal M}
\newcommand{\Dcat}{\mathcal E}
\newcommand{\Icat}{\mathcal E^\circ}

\newcommand{\one}{1}
\newcommand{\id}{\operatorname{id}}
\newcommand{\dom}{\operatorname{dom}}
\newcommand{\cod}{\operatorname{cod}}
\newcommand{\Occ}{\mathrm{S}}
\newcommand{\Comp}{\operatorname{Comp}}
\newcommand{\Pos}{\operatorname{Pos}}
\newcommand{\Out}{\mathrm{E}}

\newcommand{\Fill}{\mathcal F}

\newcommand{\inside}{\mathrel{\preceq}}
\newcommand{\emptyword}{()}
\title{\textbf{Noetherian forms of free non-symmetric operads}}
\author{Zurab Janelidze\thanks{%
The author used Chat GPT-6 Astra as a research assistant and
for help with typesetting this document. The substantial ideas
of the paper are due to the author. AI was used as a
computational accessory, including for searching the literature.
The author takes full responsibility for the final draft.%
}\\[4pt]
\small Mathematics Division, Department of Mathematical Sciences\\
\small Stellenbosch University\\
\small National Institute for Theoretical and Computational Sciences (NITheCS)\\
\small Private Bag X1, Matieland 7602, South Africa\\
\small\texttt{zurab@sun.ac.za}}
\date{}
\begin{document}
\maketitle
\begin{abstract}
In this paper, we study certain categories of labeled finite rooted ordered trees over a fixed set of labels where each label is equipped with an arity: a fixed number of children that the vertex with the given label must have. Equivalently, these are expression trees for operations in a free non-symmetric operad. A morphism between these trees matches a pruning of one tree (a prefix) with an entire subtree of another (a suffix). We characterize such categories, up to isomorphism, in terms of suitable exactness properties. It turns out that these categories exhibit strong algebraic behavior, in the sense that every such category, when appended with a strict initial object, has a particularly nice noetherian form.
\end{abstract}

\begin{flushleft}
\end{flushleft}

\section*{Introduction}
\addcontentsline{toc}{section}{Introduction}

An expression built from operation symbols has a rooted ordered tree in
which the operation symbols label vertices and the children of each
labeled vertex represent placeholders for its arguments. Empty leaves
represent inputs that are available for further insertion. A symbol
therefore prescribes the number of children of every vertex carrying
it. Such trees occur in universal algebra, rewriting, and algorithms
which recognize classes of expressions
\cite{BurrisSankappanavar,BaaderNipkow,TATA}. These trees actually represent the operations of free
non-symmetric operads \cite[Section~2.3]{Leinster}. An operad records operations with finitely many inputs and one
output, together with a rule for composing them by inserting
operations into the inputs of other operations. For the free
operad, this composition is represented by grafting trees into
empty leaves. More generally, operads allow equations between
such composites to be imposed. For example, associativity
identifies the two binary trees representing the two ways of
multiplying three arguments. An algebra for an operad interprets
its abstract operations as actual operations on a set, a vector
space, or another suitable object, respecting composition and
the prescribed equations. Operads thus allow algebraic structures
to be studied through the operations and laws that define them.
Their original applications were in topology, where they describe
composition on spaces of loops and organize the continuous
deformations through which algebraic laws such as associativity
hold. Their extensions also provide a language for higher
categories, where there are morphisms between morphisms, and
further levels of morphisms between these. In this setting,
operads describe how the different compositions fit together
\cite[Introduction and Section~2.2]{Leinster}.

In this paper we associate to every free (non-symmetric) operad a category, which for the sake of brevity we call an \emph{(operadic) expression tree category}. Objects in this category are operations of the free operad, viewed as the trees described above. A morphism between such trees is obtained by combining two processes. A
\emph{pruning} replaces selected subtrees by empty leaves; its
result is called a \emph{prefix} of the original tree. An \emph{occurrence}
selects an entire subtree at a specified position; we call that
subtree a \emph{suffix}. A \emph{prefix--suffix match} from $X$ to $Y$
identifies a prefix of $X$ with a suffix of $Y$, and thus has the
form
\[
\xymatrix@C=5em{
X\ar[r]^{\text{pruning}}&P\ar[r]^{\text{occurrence}}&Y
}
\]
The position of the suffix is part of the morphism; equal subtrees
at different positions give different occurrences. Labels and
corresponding argument places are preserved by the morphisms. To compose two
matchings, one restricts the pruning in the second matching to
the subtree selected by the first, as illustrated in Figure~\ref{fig:matching-composition}.

\begin{figure}[htbp]
\centering
\begingroup
\definecolor{matchingblue}{RGB}{48,87,123}
\definecolor{matchinggreen}{RGB}{35,109,83}
\definecolor{matchingcut}{RGB}{166,63,53}
\begin{tikzpicture}[
  x=1cm,y=1cm,
  every node/.style={font=\small,inner sep=2pt},
  vertex/.style={font=\normalsize,inner sep=2pt},
  empty/.style={vertex,circle,draw=black!55,fill=white,
    line width=.4pt,inner sep=0pt,minimum size=3.5mm},
  branch/.style={draw=black!75,line width=.55pt},
  matching/.style={-{Stealth[length=2mm]},line width=.65pt},
  selected/.style={draw=matchingblue,fill=matchingblue!5,
    rounded corners=3pt,line width=.6pt,inner sep=5pt},
  retained/.style={draw=matchinggreen,fill=matchinggreen!6,
    rounded corners=3pt,line width=.8pt,inner sep=5pt},
  cut/.style={draw=matchingcut,line width=1.1pt},
  explanation/.style={font=\footnotesize,align=center}
]
\node[font=\normalsize] at (0,.75) {$X$};
\node[font=\normalsize] at (5,.75) {$Y$};
\node[font=\normalsize] at (10,.75) {$Z$};

\node[vertex] (xp) at (0,0) {$p$};
\node[vertex] (xa) at (-.65,-.85) {$a$};
\node[vertex] (xb) at (.65,-.85) {$b$};
\node[vertex] (xbone) at (-.15,-1.70) {$b_1$};
\node[vertex] (xbtwo) at (.65,-1.70) {$b_2$};
\node[vertex] (xbthree) at (1.45,-1.70) {$b_3$};
\draw[branch] (xp)--(xa) (xp)--(xb);
\draw[branch] (xb)--(xbone) (xb)--(xbtwo) (xb)--(xbthree);
\draw[cut] (.21,-.49)--(.45,-.35);

\node[vertex] (ym) at (5,0) {$m$};
\node[vertex] (yp) at (4.25,-.85) {$p$};
\node[vertex] (yc) at (5.75,-.85) {$c$};
\node[vertex] (ya) at (3.70,-1.70) {$a$};
\node[empty] (yone) at (4.80,-1.70) {};
\draw[branch] (ym)--(yp) (ym)--(yc)
  (yp)--(ya) (yp)--(yone);
\draw[cut] (3.86,-1.22)--(4.10,-1.36);

\node[vertex] (zr) at (10,0) {$r$};
\node[vertex] (zm) at (9.25,-.85) {$m$};
\node[vertex] (zd) at (11.20,-.85) {$d$};
\node[vertex] (zp) at (8.55,-1.70) {$p$};
\node[vertex] (zc) at (9.95,-1.70) {$c$};
\node[empty] (zone) at (8.00,-2.55) {};
\node[empty] (ztwo) at (9.10,-2.55) {};
\draw[branch] (zr)--(zm) (zr)--(zd)
  (zm)--(zp) (zm)--(zc) (zp)--(zone) (zp)--(ztwo);

\begin{scope}[on background layer]
  \node[selected,fit=(yp)(ya)(yone)] (Pframe) {};
  \node[selected,fit=(zm)(zp)(zc)(zone)(ztwo),inner sep=9pt]
    (Qframe) {};
  \node[retained,fit=(zp)(zone)(ztwo)] (Zframe) {};
\end{scope}
\node[text=matchingblue,anchor=east] at (Pframe.west) {$P$};
\node[text=matchingblue,anchor=west] at (Qframe.east) {$Q$};
\node[text=matchinggreen,anchor=north east,xshift=-2pt,yshift=-2pt]
  at (Zframe.north east) {$R$};

\draw[matching] (1.15,.10)--node[above=3pt] {$f$} (3.45,.10);
\node[explanation] at (2.30,.88)
  {prune $b$; select $P$};
\draw[matching] (6.15,.10)--node[above=3pt] {$g$} (8.45,.10);
\node[explanation] at (7.30,.88)
  {prune $a$; select $Q$};

\node[font=\normalsize] at (5,-3.35) {The composite $g\circ f$};
\node[font=\normalsize] (source) at (0,-4.20) {$X$};
\node[font=\normalsize,text=matchinggreen] at (5,-4.20) {$R$};
\node[vertex] (rp) at (5,-4.90) {$p$};
\node[empty] (rone) at (4.45,-5.75) {};
\node[empty] (rtwo) at (5.55,-5.75) {};
\draw[branch] (rp)--(rone) (rp)--(rtwo);
\begin{scope}[on background layer]
  \node[retained,fit=(rp)(rone)(rtwo)] (Rframe) {};
\end{scope}
\node[font=\normalsize] (target) at (10,-4.20) {$Z$};
\draw[matching] (.45,-4.20)--(3.80,-4.20);
\node[explanation] at (2.12,-3.89) {prune $a$ and $b$};
\draw[matching] (6.20,-4.20)--(9.55,-4.20);
\node[explanation] at (7.87,-3.89) {select $R$};
\node[explanation,text=matchinggreen] at (7.87,-4.58)
  {the subtree at the\\left-left position of $Z$};
\end{tikzpicture}
\endgroup
\caption{Composition of prefix--suffix matchings. The first matching
prunes the entire subtree rooted at $b$, including its three children,
and selects the left subtree $P$ of $Y$. The second prunes $a$ and
selects the left subtree $Q$ of $Z$. Restricting the second pruning
to $P$ gives $R$, whose root has two empty leaves. Thus the composite
prunes $a$ and the subtree rooted at $b$, and selects $R$ at the
left-left position of $Z$. Short marks across edges indicate the
cuts; open circles denote empty leaves.}
\label{fig:matching-composition}
\end{figure}
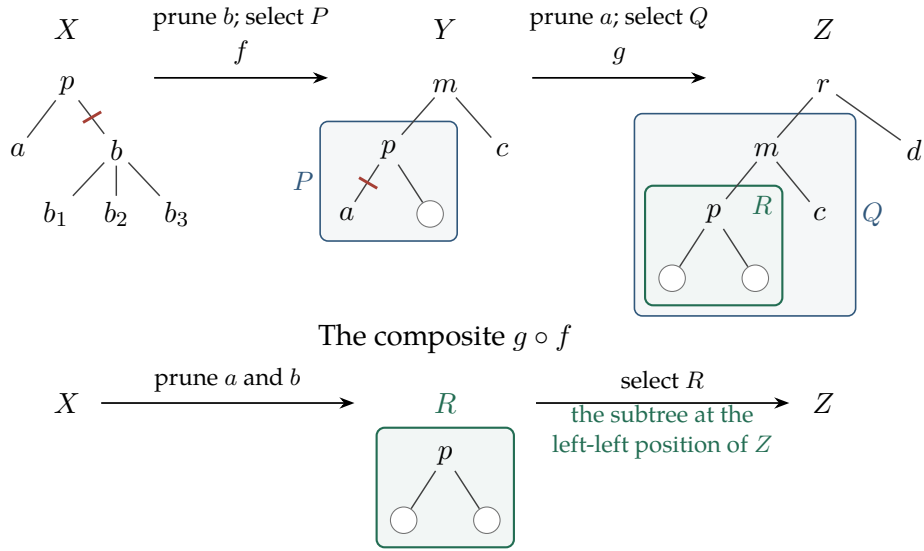

One of the main results of this paper is an intrinsic characterization of these categories of operadic expression trees. The characterization is in terms of exactness properties: behavior of certain pullbacks, pushouts and special morphisms. The category is initially given without a set of labels,
a set of vertices at each object, or designated pruning and
occurrence maps. The axioms recover all of this data, up to isomorphism of categories (Theorem~\ref{thm:representation}).

We then look into the recovery problem: how much of the starting collection of trees does the expression tree category remember? More precisely, every free non-symmetric operad gives rise to the prefix--suffix matching category of its operations. We first ask whether we can recover the operad up to isomorphism. We then ask whether, given a category satisfying our axioms and an object in it, we can recover an operation in the corresponding operad that describes that object. The first question has a positive answer (Theorem~\ref{thm:operad-recovery}): the basic generating operations of the operad can be taken to be the objects $L$ other than the terminal object $1$ for which every non-identity strong monomorphism into $L$ has the form $1\to L$. The number of these morphisms is the arity of $L$. Here $1$ represents an empty input, and the morphisms $1\to L$ represent the empty children of $L$. Now, the way these basic operations reconstruct all operations depends on the ordering of children in each tree. Without specifying this ordering it is, in general, impossible to determine which object corresponds to which operation in the recovered operad. We do show, however, that once an ordering is specified for each set $\mathsf{hom}(1,L)$, a consistent recovery of individual trees becomes possible (Theorem~\ref{thm:uniform-ordering}). Thus, our expression tree categories can be viewed as ``unoriented'' free non-symmetric operads.

We then move on to an algebraic analysis of our axioms for expression tree categories. The first few axioms guarantee the existence of a proper factorization system given by epimorphisms and strong monomorphisms. We want to know how well behaved this factorization system is from an algebraic perspective. Proper factorization systems that underlie a so-called noetherian form permit standard algebraic theorems dealing with homomorphisms, subobjects and quotients, such as the Noether isomorphism theorems. We show that in our case as well, there is a noetherian form as soon as we freely attach a strict initial object to the category. In fact, this form has an exact join decomposition in the sense of \cite{JanelidzeVanNiekerk} (Theorem~\ref{thm:tree-cosub}). The results of \cite{JanelidzeVanNiekerk} relate this property to the categorical algebra of groups: when coupled with finite limits and colimits and pointedness of the category, it forces the opposite category to be semi-abelian. We then characterize expression tree categories with a strict initial object adjoined among categories admitting a noetherian form with exact join decomposition (Theorem~\ref{thm:ejd-tree-characterization}). Pointed expression tree categories correspond precisely to the free operads whose operations are all unary (Theorem~\ref{thm:wordcharacterization}). We can equivalently view these as categories of words, where morphisms match a prefix of the domain with a suffix of the codomain. In this case too, we have a noetherian form with exact join decomposition, without the need to attach a separate initial object (Theorem~\ref{thm:word-noetherian}).

In the concluding Section~\ref{concl:section} we slightly zoom out from the focus of the rest of the paper. We formulate versions of the axioms
which recognize expression tree categories up to equivalence and
which allow matchings to permute children, showing that a similar
noetherian form also exists in the latter case. We extend the
construction used for insertion at an empty leaf to general
prefix--suffix matchings, and explain its connection with universal
constructions studied in the literature. We also briefly compare our
categories with other categorical descriptions of operadic trees
and give familiar examples satisfying the first six axioms.



\clearpage
\setcounter{tocdepth}{2}
{\small\setlength{\parskip}{0pt}\tableofcontents}

\section{Axioms and the recovery of expression trees}\label{sec:axioms}\label{sec:preparation}

Our first aim is to recognize expression tree categories from the
behavior of their morphisms. We therefore begin with a small
category $\C$, without specifying which of its objects should be
basic operation symbols or which of its morphisms should be prunings
and occurrences. Recall that a monomorphism $m:A\to B$
in $\C$ is said to be \emph{strong} if, for every commutative square
\[
\xymatrix@C=4em@R=2.5em{
P\ar[r]^a\ar[d]_e&A\ar[d]^m\\
Q\ar[r]_b\ar@{-->}[ur]^{\ell}&B
}
\]
where $e$ is an epimorphism, there exists a morphism $\ell:Q\to A$
such that $\ell e=a$ and $m\ell=b$. Such a morphism is necessarily
unique, since $m$ is a monomorphism. See, for instance,
\cite[Section 0.5]{AR}. 

\subsection{The axioms}\label{sec:axiom-list}

We consider the following axioms on $\C$.

\begin{axiom}[terminal object]\label{ax:N}
There is a terminal object $\one$ in $\C$.
\end{axiom}

For an object $X$ in $\C$, write $q_X:X\to\one$ for the unique
morphism to $\one$. A morphism $j:\one\to Y$ is called a
\emph{point} of $Y$. A morphism $f:X\to Y$ is said to be
\emph{constant} if there exists a point $j:\one\to Y$ such that
$f=jq_X$. In other words, $f$ has a factorization of the form
\[
\xymatrix@C=4em{
X\ar[r]^{q_X}&\one\ar[r]^j&Y
}
\]

\begin{axiom}[factorization]\label{ax:F}
Every morphism $f:X\to Y$ admits a factorization
\[
\xymatrix@C=4em{
X\ar[r]^e& I\ar[r]^m&Y
}
\]
where $e$ is an epimorphism and $m$ is a strong monomorphism.
Consequently, the classes of epimorphisms and strong monomorphisms
form a proper factorization system on $\C$.
\end{axiom}

Here a factorization system consists of two classes of morphisms,
closed under composition with isomorphisms, such that every
morphism factors through these classes in the
specified order, and every square with a left-class morphism on
the left and a right-class morphism on the right has a unique
diagonal. The word \emph{proper} means that the left class consists
of epimorphisms and the right class consists of monomorphisms.
In Axiom~\ref{ax:F}, the diagonal property is already part of
the definition of a strong monomorphism. Thus the existence of
the factorizations supplies the remaining requirement. In
particular, two such factorizations of the same morphism are
related by a unique isomorphism between their middle objects.

\begin{axiom}[pullbacks]\label{ax:Dpull}
For every epimorphism $d:X\to Y$ and strong monomorphism
$v:B\to Y$, there is a pullback square
\[
\xymatrix@C=4em@R=2.5em{
A\ar[r]^u\ar[d]_a&X\ar[d]^d\\
B\ar[r]_v&Y
}
\]
and the morphism $a:A\to B$ in this square is an epimorphism.
\end{axiom}

The morphism $u$ in this square is a strong monomorphism, since
strong monomorphisms are preserved by pullback. Thus both
classes of the factorization system are preserved in the
indicated pullback construction.

\begin{axiom}[bicartesian pushouts]\label{ax:Dpush}
For every strong monomorphism $u:A\to X$ and epimorphism
$a:A\to B$, there is a pushout square
\[
\xymatrix@C=4em@R=2.5em{
A\ar[r]^u\ar[d]_a&X\ar[d]^c\\
B\ar[r]_v&Y
}
\]
where $v$ is a strong monomorphism. This square is also a pullback.
\end{axiom}

The morphism $c$ is automatically an epimorphism, since
epimorphisms are preserved by pushout. A square which is both
a pullback and a pushout is called \emph{bicartesian}.

\begin{axiom}[pushout stability]\label{ax:pointpushout}
The class of epimorphisms whose pullbacks along all but at most one
point are isomorphisms is stable under pushout along strong
monomorphisms.
\end{axiom}

For example, for the pushout square in Axiom~\ref{ax:Dpush},
this axiom says the following. If the pullback of $a$ along
$j:\one\to B$ is an isomorphism for all but at most one choice
of $j$, then the same assertion holds for the pullbacks of $c$
along the points of $Y$.

\begin{axiom}[terminal pullback factorizations]\label{ax:U}
For every object $A$ and point $j:\one\to B$, the following
category, denoted by $\Fill(A,j)$, has a terminal object. An
object of $\Fill(A,j)$ is a triple $(C,d,u)$, where
$u:A\to C$ is a strong monomorphism, $d:C\to B$ is an
epimorphism, and the square
\begin{equation}\label{eq:fillsquare}
\vcenter{\xymatrix@C=4em@R=2.5em{
A\ar[r]^u\ar[d]_{q_A}&C\ar[d]^d\\
 \one\ar[r]_j&B
}}
\end{equation}
is a pullback. A morphism from $(C',d',u')$ to $(C,d,u)$ is
a morphism $h:C'\to C$ making the following diagram commute
\[
\xymatrix@C=3.5em@R=2.5em{
&&C'\ar[dl]_h\ar@/^1pc/[ddl]^{d'}\\
A\ar[r]^u\ar[d]_{q_A}\ar@/^1pc/[urr]^{u'}&C\ar[d]^d&\\
\one\ar[r]_j&B&
}
\]
Thus $hu'=u$ and $dh=d'$.
Composition and identities are those of $\C$.
\end{axiom}

Thus a terminal object of $\Fill(A,j)$ is a pullback
factorization~\eqref{eq:fillsquare} to which every other such
factorization has a unique comparison morphism fixing $A$ and
$B$.

\begin{axiom}[finiteness]\label{ax:finite}
For every object $X$, there are finitely many strong monomorphisms
with codomain $X$ and finitely many epimorphisms with domain $X$,
counting morphisms with all possible other endpoints.
\end{axiom}

\begin{axiom}[diagonals]\label{ax:O}
Every commutative square of strong monomorphisms
\[
\xymatrix@C=4em@R=2.5em{
W \ar[r]^{a} \ar[d]_{b} & X \ar[d]^{c} \\
Y \ar[r]_{d} & Z
}
\]
admits either a diagonal $h:X\to Y$ satisfying $ha=b$ and $dh=c$,
or a diagonal $k:Y\to X$ satisfying $kb=a$ and $ck=d$.
\end{axiom}

\begin{axiom}[detection of identities]\label{ax:Ddetect}
If an epimorphism $d:X\to Y$ has the property that its pullback along
every point $j:\one\to Y$ is an isomorphism, then $d$ is an identity.
\end{axiom}

In Axiom~\ref{ax:finite}, morphisms with different domains or
codomains are counted separately. For an expression tree, the two
finite collections will be its subtree occurrences and its possible
prunings. Counting occurrences retains the distinction between
equal subtrees at different positions.

\subsection{The category associated to a free non-symmetric operad}\label{sec:model}
We describe the objects and morphisms of an expression tree category
explicitly before verifying the axioms. The labels in the following
definition will serve as the generating operation symbols of the
free operad.

\begin{definition}\label{def:ranked-set}
Let $L$ be a set of labels, and let
$\operatorname{ar}:L\to\mathbb N$ assign a nonnegative integer to
each label. This integer prescribes the number of children at every
vertex carrying that label. For indexing purposes write
$P_\lambda=\{1,\ldots,\operatorname{ar}(\lambda)\}$.
We call $(L,\operatorname{ar})$ a \emph{finite-arity signature},
or a \emph{ranked set}.
The finite rooted trees over this signature are defined recursively as follows:
\begin{enumerate}
\item The tree $\one$ consists of a single empty leaf.
\item For $\lambda\in L$ and a family of trees $(T_p)_{p\in P_\lambda}$,
there is a tree $\lambda((T_p)_{p\in P_\lambda})$.
\end{enumerate}
Two trees are equal when their root labels agree and their children
at corresponding input places are recursively equal. A label of
arity zero gives a vertex with no children, and hence an operation
with no inputs. The tree $\one$ represents one empty input and will
be the identity operation for grafting. The indices above specify
the argument places in this concrete presentation. Replacing them
by other finite sets, with specified bijections, gives isomorphic
expression tree categories.
\end{definition}

A tree with several labeled vertices represents a composite of the
generating operations. Its inputs are its empty leaves, so its
arity as an operadic operation is the number of those leaves.
The fact that these trees describe a free operad means the following.
If each generating symbol is interpreted as an operation of the same
arity in any other non-symmetric operad, then every tree has a unique
interpretation obtained by composing the operations at its vertices.
This interpretation respects insertion of trees into empty leaves.
Theorem~\ref{thm:operad-recovery} proves this assertion and shows how
the expression tree category recovers the operad.

A position of a tree is either its root or a position within one
of its children, with that child specified. Thus the tree $\one$
has one position, although it has no labeled vertex. For trees $A$
and $X$, a \emph{subtree embedding} $A\to X$ specifies
a position of $X$ whose entire subtree is $A$. Let $\Ocat_0$ be
the category with these arrows; composition concatenates the paths
to the two specified positions. A \emph{pruning map} either replaces a
nonempty tree by $\one$, or retains its root label and independently
prunes each child at its given input place. At $\one$ only the
identity is allowed. Write $\Dcat_0$ for the category having exactly one
arrow $X\to T$ when $T$ is obtainable from $X$ in this way, and
no arrow otherwise. The recursive rule for pruning is transitive,
as follows by composing the pruning maps of the children.

A \emph{prefix} of $X$ is a tree $T$ obtainable by pruning $X$,
including $X$ itself and $1$. A \emph{suffix occurrence} in $Y$ is
a specified subtree embedding $u:T\to Y$. A \emph{prefix--suffix
matching} from $X$ to $Y$ is a pair $(d,u)$ with
$d:X\to T$ in $\Dcat_0$ and $u:T\to Y$ in $\Ocat_0$.
The common tree $T$ is its \emph{overlap}. Its labels and input places
agree exactly on both sides. Distinct suffix occurrences give distinct
matchings. In particular, a matching with overlap $\one$ selects
a specified empty leaf of $Y$. Such matchings are precisely the
constant morphisms in this category.

To describe composition, write $\Pos(X)$ for the finite set of
addresses of positions in $X$, including its empty leaves. The root
has the empty word as its address, and descending through child $p$
prefixes $p$ to an address in that child. Write $X|_a$ for the
subtree at address $a$.
An \emph{effective cut} is a set of addresses $a$ with
$X|_a\ne\one$ such that no selected address is a proper prefix of
another. Such a set is called a \emph{prefix antichain}. The condition
ensures that the selected subtrees are pairwise disjoint.
Pruning replaces the selected subtrees by $\one$. Successive cuts
compose by retaining the prefix-minimal addresses in their union. This
discards any selected address
having a shorter selected prefix; addresses of the second cut are
identified with retained addresses of the first tree. The resulting operation is called the \emph{outermost union}.
Induction on the source tree identifies this description with the
recursive definition of pruning.
A pruning $d$ induces the function $p_d$ that collapses each cut
subtree to its boundary; an embedding $u$ at address $a$ induces
the prefix function $p_u(b)=ab$.

Define the \emph{(operadic) expression tree category} $\C_L$ of the
signature $L$ by taking these trees as objects and setting
\begin{equation}\label{eq:modelhom}
\C_L(X,Y)=\coprod_T\Dcat_0(X,T)\times\Ocat_0(T,Y).
\end{equation}
To compose matchings $(d,u):X\to Y$ and $(e,v):Y\to Z$, let $a$
be the address of $u$. Restrict $e$ to this subtree, obtaining
$e_a:T\to S|_{p_e(a)}$, where $S$ is the codomain of $e$.
If a cut of $e$ lies at or above $a$, this restriction is total
pruning. In the other case it has exactly the cuts of $e$ inside
the selected subtree. Let $u_{p_e(a)}$ be the corresponding
embedding in $S$. The restriction gives the commutative square
\[
\xymatrix@C=4.5em@R=2.5em{
T\ar[r]^u\ar[d]_{e_a}&Y\ar[d]^e\\
S|_{p_e(a)}\ar[r]_{u_{p_e(a)}}&S.
}
\]
We use this square to define the composite by
\[
(e,v)\circ(d,u)=(e_a d,\,v u_{p_e(a)}).
\]
The positional function of this composite is $p_vp_ep_up_d$.
To see why this composition is associative, observe first that a
matching with given source and target is determined by its action
on positions. The image of the source root determines the selected
position of the target, hence the overlap $T$ and its embedding.
The pruning to $T$ is then unique by the definition of $\Dcat_0$.
The two ways of composing three matchings have the same positional
function, since composition of functions is associative. They are
therefore equal. Similarly, the pair consisting of the identity
pruning and the occurrence of the whole tree acts identically on
positions and satisfies both identity laws. These observations also
give a faithful functor $\Pos:\C_L\to\mathbf{Set}$.

The empty leaf $\one$ is terminal: it has only itself as a subtree,
and every tree has exactly one total pruning to it. Its points in a
tree are exactly the embeddings selecting that tree's empty leaves.

\begin{lemma}[Intrinsic arrow classes in the model]\label{prop:modelclasses}
In $\C_L$, the epimorphisms are precisely $\Dcat_0$ and the strong
monomorphisms are precisely $\Ocat_0$.
\end{lemma}
\begin{proof}
The positional functor sends pruning maps to surjections and
subtree embeddings to injections. Cancellation of functions, followed
by faithfulness, therefore shows that pruning maps are epimorphisms
and subtree embeddings are monomorphisms.

We next show that a proper subtree embedding is never an
epimorphism. Let $u:A\to Y$ be such an embedding, and let $p$
be the first input place on its path from the root of $Y$.
If $Y=\lambda((Y_p)_p)$, put $r_0(Y)=\lambda((\one)_p)$.
Let $\rho_Y:Y\to r_0(Y)$ clear every immediate child, and let
$j_p:\one\to r_0(Y)$ be that root place. The two maps
\[
 \xymatrix@C=5em{
Y \ar@<0.6ex>[r]^{\rho_Y} \ar@<-0.6ex>[r]_{j_pq_Y} & r_0(Y)
 }
\]
send the root of $Y$ to different positions, and are therefore
distinct. Both send the subtree selected by $u$ to the empty leaf
at $p$, so they agree after precomposition with $u$. This proves
the assertion. If a matching $f=ud$ is an epimorphism, its final
factor $u$ must also be an epimorphism and hence must select the
whole target tree. Thus $u$ is an identity, which identifies all
epimorphisms as pruning maps.

To prove that an embedding $m:A\to B$ is strong monic, consider a
commutative square
\[
 \xymatrix@C=4em@R=2.5em{
P\ar[r]^g\ar[d]_d&A\ar[d]^m\\
Q\ar[r]_h\ar@{-->}[ur]^{\ell}&B
 }
\]
where $d:P\to Q$ is a pruning map. Write $g=u_gd_g$ and $h=u_hd_h$.
The pair presentation~\eqref{eq:modelhom} gives
$mu_g=u_h$ and $d_g=d_hd$. It follows that $\ell=u_gd_h:Q\to A$ satisfies $\ell d=g$ and
$m\ell=h$. Monicity of $m$ makes this diagonal unique. Since the
epimorphisms have already been identified as prunings, this proves
that every subtree embedding is strong monic.

Any isomorphism has a bijective positional map. Its subtree-embedding
factor must be an identity, since a proper subtree embedding misses
the target root. Applying the same argument to its inverse shows
that both maps are prunings. A nonidentity pruning strictly
decreases the number of labeled vertices, so these two prunings
must be identities.

Finally, suppose $f=ud:X\to Y$ is a strong monomorphism, where
$d:X\to T$ is epic and $u:T\to Y$ is a subtree embedding.
Apply the lifting property of $f$ to the square
\[
\xymatrix@C=4em@R=2.5em{
X\ar[r]^{\id_X}\ar[d]_d&X\ar[d]^f\\
T\ar[r]_u\ar@{-->}[ur]^{\ell}&Y.
}
\]
It supplies
$\ell:T\to X$ with $\ell d=\id_X$ and $f\ell=u$.
The equality $d\ell d=d$ and epicity of $d$ give
$d\ell=\id_T$. Thus $d$ is an isomorphism, hence an identity,
and $f$ is a subtree embedding.
\end{proof}

\begin{theorem}[Expression tree categories satisfy the axioms]\label{thm:realization}
For every finite-arity signature $L$, the category $\C_L$ satisfies
Axioms~\mbox{\ref{ax:N}--\ref{ax:Ddetect}}.
\end{theorem}
\begin{proof}

Lemma~\ref{prop:modelclasses} identifies the epimorphisms
and strong monomorphisms. The defining pair $(d,u)$ of an arrow
factors that arrow as the pruning $d$ followed by the subtree
embedding $u$, which proves Axiom~\ref{ax:F}.
Consider a commutative square of subtree embeddings, with common
codomain $Z$. The paths in $Z$ to the two intermediate objects are
prefixes of the path to the common source. They are therefore
comparable by extension. Removing the shorter prefix from the longer
one gives an embedding between the intermediate objects, and path
concatenation gives both diagonal equations of Axiom~\ref{ax:O}.
The empty leaf is terminal, giving Axiom~\ref{ax:N}.

For Axiom~\ref{ax:Dpull}, let $d:X\to Y$ be a pruning map
and let $v:B\to Y$ be a subtree embedding with address $b$. The address $b$ is
retained from $X$. Let $w:X|_b\to X$ be its subtree embedding and let
$d_b:X|_b\to B$ be the restricted pruning. We claim that the square
\[
 \xymatrix@C=4em@R=2.5em{
X|_b\ar[r]^w\ar[d]_{d_b}&X\ar[d]^d\\
B\ar[r]_v&Y
 }
\]
is a pullback. Suppose $h:T\to X$ and $k:T\to B$ satisfy
$dh=vk$, and factor $h$ as a pruning followed by a subtree embedding.
The image of the root under $dh$ lies at or below $b$. The same
is then true of the address of the embedding factor of $h$ in $X$:
a pruning can send a position to an ancestor, but cannot move it
into a different branch. Consequently $h$ factors through $w$.
The resulting map $\ell:T\to X|_b$ is unique by monicity of $w$.
The equality $vd_b\ell=dh=vk$ and monicity of $v$ give
$d_b\ell=k$. This proves the pullback property, and its left map
is a pruning as required by Axiom~\ref{ax:Dpull}.

To verify Axiom~\ref{ax:Dpush}, take a subtree embedding $u:A\to X$
and a pruning $a:A\to B$. Replace the subtree at $u$ by $B$,
obtaining a pruning $c:X\to Y$ and an embedding $j:B\to Y$.
These maps form the commutative square
\[
 \xymatrix@C=4em@R=2.5em{
A\ar[r]^u\ar[d]_a&X\ar[d]^c\\
B\ar[r]_j&Y
 }
\]
Its pullback property follows from the preceding description. To prove the pushout property, let $f:X\to T$ and $g:B\to T$
satisfy $fu=ga$, and write $f=ve$, with $e:X\to R$ a pruning
and $v:R\to T$ a subtree embedding. If a cut of $e$ lies at or above the selected
subtree, it absorbs its replacement by $B$. Otherwise the restriction
of $e$ there factors through $a$. To see this, write both $fu$ and
$ga$ as a pruning followed by a subtree embedding and use the
uniqueness of that description. Use that factor at the replaced subtree and
retain all other restrictions of $e$. In both cases this constructs
a mediator $h:Y\to T$ with $hc=f$; the equation $hj=g$ follows by
cancelling the epic $a$. Uniqueness follows from the epicity of $c$.
The pullback of a pruning $d:X\to Y$ along a point of $Y$
identifies the subtree of $X$ sent to that empty leaf. We call this
subtree the \emph{fibre} of $d$ over the point. It is a nonempty
subtree when that leaf was created by a cut; otherwise it is already
an empty leaf of $X$. Consequently, the fibres other than $\one$
are exactly the subtrees removed by the effective cuts. Thus a
pruning satisfies the condition in Axiom~\ref{ax:pointpushout}
precisely when it is an identity or has a single effective cut. Pushing a single-cut map out along a
subtree embedding concatenates the embedding address with the cut
address, giving another single-cut map. Pushouts of identities are
identities. This proves the required stability.
Any nonidentity pruning has a
nonterminal cut subtree; its boundary is a target point whose pullback
fibre is that subtree. This proves Axiom~\ref{ax:Ddetect}.

There are finitely many subtree embeddings into a tree and finitely
many pruning maps out of it, indexed by positions and effective cuts.
This verifies Axiom~\ref{ax:finite}.

For Axiom~\ref{ax:U}, replace the marked empty leaf of $B$ by $A$.
Let $j:\one\to B$ be the marked point. The resulting tree $C$
has an embedding $u:A\to C$ and a pruning $d:C\to B$ fitting into
\[
 \xymatrix@C=4em@R=2.5em{
A\ar[r]^u\ar[d]_{q_A}&C\ar[d]^d\\
 \one\ar[r]_j&B
 }
\]
The description of point fibres above shows that this is a pullback.
Let $(C',d',u')$ be another filling of $A$ at $j$. Its
pullback square identifies the complete subtree erased over $j$
with $A$, by the preceding description of point fibres. Prune all
effective cut subtrees of $d'$ except this marked fibre. The remaining tree is recursively
exactly $C$, giving a comparison. If $B\ne\one$, both $d$ and $d'$ preserve the root. If a comparison
$h:C'\to C$ had a proper subtree embedding as its last factor, then
$dh$ would send the source root to a proper position of $B$, since
$d$ preserves the root label. This contradicts $dh=d'$. Hence $h$ is a pruning. There is at most one pruning with given
source and target by the definition of $\Dcat_0$, so the comparison
is unique.
If $B=\one$, the pullback along $\id_\one$ forces the top map
to be an isomorphism. The concrete category has only identity
isomorphisms, so every filling is $(A,q_A,\id_A)$, and a
comparison preserving $\id_A$ is itself the identity.
\end{proof}

\subsection{Strong monomorphisms and the positions of an object}\label{sec:strong-positions}

We now work towards a representation theorem for a small category
satisfying the nine axioms. The first step is to identify the
positions of each object. We will show that these positions are
described by the strong monomorphisms into the object, whose domains
represent the subtrees at those positions. Factorization of strong
monomorphisms will then describe the relation between a position
and its ancestors.

The following cancellation property is well known; see, e.g.,
\cite[p.~292]{CassidyHebertKelly}.

\begin{lemma}[Cancellation of strong monomorphisms]\label{lem:strong-cancel}
Let $u:A\to B$ and $v:B\to C$ be morphisms. If $vu$ is a
strong monomorphism, then $u$ is a strong monomorphism.
\end{lemma}

\begin{lemma}[Points]\label{lem:points-strong}
Under Axiom~\ref{ax:N}, every point is a strong monomorphism.
\end{lemma}
\begin{proof}
For $j:\one\to Y$, terminality gives $q_Yj=\id_{\one}$.
Thus $j$ is a split monomorphism, and every split monomorphism
is strong; see, e.g., \cite[p.~292]{CassidyHebertKelly}.
\end{proof}

\begin{lemma}[Rigidity]\label{lem:rigidity}
Every isomorphism is an identity, and every point is a strong monomorphism.
\end{lemma}
\begin{proof}
The assertion about points is Lemma~\ref{lem:points-strong}.
Let $f:X\to Y$ be an isomorphism. Its pullback along any point
of $Y$ is again an isomorphism. Since $f$ is epic,
Axiom~\ref{ax:Ddetect} gives $f=\id_X$.
\end{proof}

Write
\[
\Dcat=\operatorname{Epi}(\C),\qquad
\Ocat=\operatorname{StrongMono}(\C),\qquad
\Occ(X)=\coprod_A\Ocat(A,X),\qquad
\Out(X)=\coprod_Y\Dcat(X,Y).
\]
Here $\Ocat$ and $\Dcat$ have all objects of $\C$ and only the
indicated morphisms. They are subcategories because the two classes
contain identities and are closed under composition. The set
$\Occ(X)$ records a strong monomorphism together with its domain;
$\Out(X)$ records an epimorphism together with its codomain.
They are the finite sets in Axiom~\ref{ax:finite}. We call a
strong monomorphism \emph{proper} when it is not an identity.
We will also use the standard fact that a pullback of a strong
monomorphism is strong whenever that pullback exists; see, e.g.,
\cite[pp.~290 and 292]{CassidyHebertKelly}.

\begin{lemma}\label{lem:empty}
Every $q_X$ is epic. The only strong monomorphism into $\one$ and the only
epimorphism out of $\one$ are $\id_\one$.
\end{lemma}
\begin{proof}
Suppose $a,b:\one\to Z$ satisfy $a q_X=b q_X$.
Terminality makes $a,b$ split monomorphisms, hence strong.
By Axiom~\ref{ax:F}, factor $q_X=ve$ with $e:X\to Q$ epic and
$v:Q\to\one$ strong monic. Cancelling $e$ gives $av=bv$, so we have
the commutative square
\[
 \xymatrix@C=4em@R=2.5em{
Q\ar[r]^v\ar[d]_v&\one\ar[d]^a\\
 \one\ar[r]_b&Z
 }
\]
All four maps are strong monomorphisms. Apply Axiom~\ref{ax:O}.
Whichever diagonal the axiom supplies is an endomorphism of $\one$,
and hence is its identity. Its defining equations therefore give
$a=b$, proving that $q_X$ is epic.

A strong monomorphism into $\one$ is thus epic and strong, hence an
isomorphism and an identity. If $d:\one\to X$ is epic,
then $q_Xd=\id_\one$. Cancelling $d$ from $dq_Xd=d$
gives $dq_X=\id_X$, so $d$ too is an identity.
\end{proof}

The children at the root of a tree are its maximal proper subtrees.
Correspondingly, we recover the children of an object from those
proper strong monomorphisms into it which admit no further proper
strong monomorphism between their domain and codomain.
\begin{definition}\label{def:components}
A strong monomorphism $i:A\to X$ is said to be
\emph{indecomposable} if it is nonidentity and, whenever
$i=vu$ with $u:A\to B$ and $v:B\to X$ in $\Ocat$, at least
one of $u,v$ is an identity.
Write $\Comp(X)$ for the set of indecomposable strong monomorphisms into $X$.
The \emph{component decomposition} of $X$ is the family
$(i:X_i\to X)_{i\in\Comp(X)}$.
\end{definition}

Lemma~\ref{lem:strong-cancel} shows that a factor between
strong monomorphisms is again strong. In particular, either
diagonal supplied by Axiom~\ref{ax:O} is a strong monomorphism.

\begin{lemma}\label{prop:paths}
The following assertions hold.
\begin{enumerate}
\item Every endomorphism that is a strong monomorphism is an identity.
\item The relation $u\inside v$ defined by $u=vw$ for some $w\in\Ocat$
is a partial order on the finite set $\Occ(X)$, with greatest
element $\id_X$, and the elements greater than or equal to any fixed
element form a chain. We refer to such a partial order as a
\emph{tree branching order}; greater elements are ancestors.
\item Every strong monomorphism has a unique factorization as a finite composable
list of indecomposable strong monomorphisms. The empty list represents an identity,
and no nonempty list represents an identity.
\item Every morphism has a unique factorization, called its
\emph{normal factorization}, of the form
\begin{equation}\label{eq:normal}
f=ud,\qquad d\in\Dcat,\quad u\in\Ocat.
\end{equation}
\end{enumerate}
\end{lemma}
\begin{proof}
If $s:X\to X$ is a strong monomorphism, its powers belong to the
finite set $\Occ(X)$. Thus $s^m=s^n$ for some $m>n\geq0$.
Cancelling the monomorphism $s^n$ gives $s^{m-n}=\id_X$.
Consequently $s$ is an isomorphism, and Lemma~\ref{lem:rigidity} makes it an identity.

For \textup{(b)}, the factorization relation is reflexive and
transitive. If $u=vw$
and $v=ut$, then monicity gives $tw=\id$ and $wt=\id$.
Rigidity makes $w,t$ identities and $u=v$, proving antisymmetry.
If a strong monomorphism $z$ factors through both $u:A\to X$
and $v:B\to X$, write $z=ua=vb$. All four maps in the square
\[
\xymatrix@C=4em@R=2.5em{
\dom(z)\ar[r]^a\ar[d]_b&A\ar[d]^u\\
B\ar[r]_v&X
}
\]
are strong monomorphisms. Axiom~\ref{ax:O} therefore makes $u$
and $v$ comparable. It follows that the elements above $z$ form
a finite chain ending at $\id_X$.

For \textup{(c)}, list the elements of the finite chain
$[u,\id_X]$ in their order.
Between two successive elements there is no intervening element;
this is the meaning of a cover in this order. The factors between
successive elements therefore give a factorization of $u$ into
indecomposable strong monomorphisms. Each intervening factor is unique
by monicity, and every component factorization gives exactly this
chain of covers. A nonempty component list would give a strict chain,
so cannot represent an identity.

For \textup{(d)}, Axiom~\ref{ax:F} supplies a factorization which
is unique up to a unique isomorphism, as recalled immediately after
that axiom. Lemma~\ref{lem:rigidity} makes this isomorphism an
identity. The intermediate object and both factors are therefore
uniquely determined.
\end{proof}

The preceding lemma gives the positions of an object their tree
structure: factoring a strong monomorphism describes passing from
a selected subtree to a larger subtree containing it. Two strong
monomorphisms with the same codomain are \emph{disjoint}
if there is no object $T$ and no strong monomorphism from $T$ to
that codomain factoring through both of them. By
Lemma~\ref{prop:paths}, this is equivalent to incomparability
in the strong-subobject order.

Let $d:X\to Y$ be an epimorphism and let $u:A\to X$ be a strong
monomorphism. Factor $du$ as
\begin{equation}\label{eq:residual}
 \vcenter{\xymatrix@C=4em@R=2.5em{
A\ar[r]^u\ar[d]_{d|_u}&X\ar[d]^d\\
A'\ar[r]_{u^d}&Y
 }}
\end{equation}
where $d|_u:A\to A'$ is epic and $u^d:A'\to Y$ is strong monic.
We call these the \emph{restriction} of $d$ at $u$ and the
\emph{residual strong monomorphism}, respectively. The first describes
what $d$ does to the selected part $A$; the second describes where
the resulting part lies in $Y$. Lemma~\ref{prop:paths} makes both
maps and the intermediate object $A'$ uniquely determined.

\subsection{Pruning one subtree at a time}\label{sec:single-pruning}
We next describe how the axioms allow us to replace a selected
subtree by an empty leaf. The corresponding operation in the
abstract category is to replace the domain of a strong
monomorphism by the terminal object. For a strong monomorphism
$u:A\to X$, write $\pi_u:X\to X/u$
for the map with domain $X$ in the pushout
\[
\xymatrix@C=4em@R=2.5em{
A \ar[r]^{u} \ar[d]_{q_A} & X \ar[d]^{\pi_u} \\
 \one \ar[r]_{j} & X/u
}
\]
Axiom~\ref{ax:Dpush} makes this square both a pushout and a
pullback, with $j$ a point. We denote its lower-right object by
$X/u$. The square expresses the replacement of the selected part
$A$ by the terminal object, with $j$ recording the resulting
empty position. The arrow $\pi_u$ is epic, since pushouts preserve epimorphisms and
$q_A$ is epic by Lemma~\ref{lem:empty}. Rigidity makes the pushout
unique, including its structure arrows.

\begin{lemma}\label{lem:pushoutidentity}
The arrow $\pi_u$ is an identity if and only if $A=\one$.
\end{lemma}
\begin{proof}
If $A=\one$, the left-hand map of the pushout square is an
identity. Hence its pushout $\pi_u$ is an isomorphism and, by
rigidity, an identity. Conversely, if $\pi_u$ is an identity,
its pullback $q_A$ is an isomorphism. Rigidity then gives
$A=\one$.
\end{proof}

The next theorem shows that every epimorphism is obtained by
repeating this construction. At each step, we replace one part
that the given epimorphism sends to an empty leaf. The finiteness
axiom ensures that the process ends after finitely many steps.

\begin{theorem}[Finite decomposition]\label{thm:finitegeneration}
Every epimorphism with domain $X$ is a composite of at most
$|\Out(X)|-1$ nonidentity morphisms of the form $\pi_u$. The empty composite
represents an identity.
\end{theorem}
\begin{proof}
For a nonidentity epimorphism $d:X\to Y$, Axiom~\ref{ax:Ddetect}
supplies a point $j:\one\to Y$ whose pullback along $d$ is
not an isomorphism. Write $u:A\to X$ for the pullback strong
monomorphism. Its other projection is $q_A:A\to\one$, and
Lemma~\ref{lem:rigidity} implies $A\ne\one$. Form the pushout
$c=\pi_u:X\to X_1$. Since $du=jq_A$, its universal property gives
$e:X_1\to Y$ in the diagram
\[
 \xymatrix@C=3.5em@R=2.5em{
A\ar[r]^u\ar[d]_{q_A}&X\ar[d]^c\ar@/^1pc/[ddr]^d&\\
 \one\ar[r]_s\ar@/_1pc/[drr]_j&X_1\ar@{-->}[dr]^e&\\
 &&Y
 }
\]
Thus $ec=d$ and $es=j$. If two maps agree after $e$, they agree
after $ec=d$, so $e$ is epic. Moreover, $c$ is nonidentity by
Lemma~\ref{lem:pushoutidentity}.

Apply the same construction to $e:X_1\to Y$ whenever it is
nonidentity, and then to each subsequent epimorphism with target $Y$.
The cumulative epimorphisms $p_0=\id_X$ and $p_n=c_n\cdots c_1$
are distinct. Indeed, $p_m=p_n$ for $m>n$ would imply
$c_m\cdots c_{n+1}=\id$ by epic cancellation. Then $c_{n+1}$
would be split monic as well as epic, hence an isomorphism and
an identity, a contradiction. Finiteness of $\Out(X)$ forces
termination after at most $|\Out(X)|-1$ steps. The remaining
epimorphism must be an identity, since otherwise
Axiom~\ref{ax:Ddetect} would allow a further step. Thus the
composite constructed up to that point equals $d$.
\end{proof}

\subsection{Replacing the children at the root}\label{sec:component-pruning}
When we prune a subtree, every subtree in a different branch
remains in place. We next establish this property from the
axioms. It will allow us to prescribe a pruning of each child
at the root and combine these prescriptions into a pruning of
the whole object.

\begin{lemma}[Factorization of constant composites]\label{lem:constantfactor}
Let $u:A\to X$ be a strong monomorphism, and let $j:\one\to X/u$
be the point in the pushout defining $\pi_u$.
For every point $k\ne j$ of $X/u$, the pullback of $\pi_u$ along
$k$ is an isomorphism.
If $w:W\to X$ is strong monic with nonterminal domain and $\pi_u w$
is constant, then $\pi_u w=jq_W$ and there is a unique $t:W\to A$
such that $ut=w$.
\end{lemma}
\begin{proof}
The terminal object has only one point, so all but at most one
of the pullbacks of $q_A$ along points are isomorphisms.
Axiom~\ref{ax:pointpushout} transfers this property to $\pi_u$.
If $A\ne\one$, the defining square for $\pi_u$ identifies its
pullback along $j$ with the nonisomorphism $q_A$. Hence the pullback
along every other point is an isomorphism. If $A=\one$,
Lemma~\ref{lem:pushoutidentity} makes $\pi_u$ an identity, so all
of its point pullbacks are isomorphisms.

Write $\pi_u w=kq_W$. If $k\ne j$, the pullback of $\pi_u$ along $k$ has domain
$\one$. Its universal property applied to $w$ and $q_W$ therefore
gives a point $t:\one\to X$ with $w=tq_W$. Since $w$ is strong
monic, Lemma~\ref{lem:strong-cancel} makes $q_W$ strong monic. It is also epic, so it is an
isomorphism, contradicting that $W$ is nonterminal. Thus $k=j$,
and the pullback square defining $\pi_u$ gives the unique factor $t$.
\end{proof}

\begin{lemma}[Locality of pushouts]\label{lem:pushoutlocality}
In a pushout of a strong monomorphism $u:A\to X$ along an epimorphism
$a:A\to A'$, write $c:X\to Y$ for the map induced from $X$ and $j:A'\to Y$
for the map induced from $A'$.
For every strong monomorphism $v:B\to X$ disjoint from $u$, the
composite $cv:B\to Y$ is a strong monomorphism.
\end{lemma}
\begin{proof}
First consider $e=\pi_u:X\to X/u$, and factor $ev=mp$ with
$p:B\to C$ epic and $m:C\to X/u$ strong monic.
Let $j:\one\to C$ be a point, and form its pullback along $p$:
\[
 \xymatrix@C=4em@R=2.5em{
K\ar[r]^k\ar[d]_{q_K}&B\ar[d]^p\ar[r]^v&X\ar[d]^e\\
 \one\ar[r]_j&C\ar[r]_m&X/u
 }
\]
Axiom~\ref{ax:Dpull} supplies this pullback. Its top map $k$ is
strong monic because $j$ is strong monic.
If $K$ were nonterminal, the strong monomorphism $vk$ would satisfy
$e(vk)=mj q_K$, a constant morphism. Lemma~\ref{lem:constantfactor}
would then factor $vk$ through $u$. That factor is strong monic by
Lemma~\ref{lem:strong-cancel},
contradicting disjointness of $u$ and $v$.
Thus every such $K$ is terminal, so every pullback of $p$ along a
point is an isomorphism. Axiom~\ref{ax:Ddetect} gives $p=\id_B$.
Consequently $ev$ is strong monic.

For the general pushout in the statement, write $eu=sq_A$.
Since $q_{A'}a=q_A$, the maps $e$ and $sq_{A'}$ agree after
composition with $u$ and $a$, respectively. The pushout property
therefore gives $h:Y\to X/u$ in the diagram
\[
\xymatrix@C=3.5em@R=2.5em{
A\ar[r]^u\ar[d]_a&X\ar[d]^c\ar@/^1pc/[ddr]^e&\\
A'\ar[r]_j\ar@/_1pc/[drr]_{s q_{A'}}&Y\ar@{-->}[dr]^h&\\
&&X/u
}
\]
In particular, $h(cv)=ev$ is strong monic by the case already
proved. Lemma~\ref{lem:strong-cancel} now implies that $cv$ is
strong monic.
\end{proof}

\begin{lemma}[Strong monomorphism pullbacks]\label{lem:strongpullbacks}
For an epimorphism $d:X\to Y$, let $d^*v$ denote the strong
monomorphism obtained by pulling back $v\in\Occ(Y)$. For all
$u\in\Occ(X)$ and $v\in\Occ(Y)$, we have
\[
u^d\inside v\quad\Longleftrightarrow\quad u\inside d^*v,
 \qquad (d^*v)^d=v.
\]
Pullback sends proper strong monomorphisms to proper strong monomorphisms.
\end{lemma}
\begin{proof}
Let $v:B\to Y$ be a strong monomorphism. By the pullback universal
property, $u$ factors through $d^*v$ if and only if $du$ factors
through $v$. If $du=vh$, factor $h$ as an epimorphism followed by a
strong monomorphism. Composing its strong factor with $v$ gives a
normal factorization of $du$, so uniqueness identifies it with $u^d$.
Thus $du$ factors through $v$ if and only if $u^d\inside v$.
The factors between strong monomorphisms are strong by
Lemma~\ref{lem:strong-cancel}. Thus this equivalence is precisely
the first assertion in the displayed formula. The projection from
the pullback to $\dom(v)$ is epic by Axiom~\ref{ax:Dpull}, so
the pullback square already gives a normal factorization whose
strong monomorphism is $v$. This proves $(d^*v)^d=v$.
Finally, if $d^*v=\id_X$, that square factors the epimorphism
$d$ through $v$. It follows that $v$ is epic as well as strong
monic, and hence is an identity. This proves the last assertion.
\end{proof}

\begin{lemma}[One-component pushouts]\label{lem:componentpushout}
Push out an indecomposable strong monomorphism $i:A\to X$ along any epimorphism
$a:A\to B$, writing $c:X\to Y$ and $j:B\to Y$ for the new arrows.
Then $Y\ne\one$, $j$ is an indecomposable strong monomorphism, and the component
decomposition of $Y$ consists exactly of $j$ and the strong monomorphisms
$ck$ for $k\in\Comp(X)\setminus\{i\}$. These strong monomorphisms are
distinct. The restriction of $c$ at $i$ is $a$, and its restriction
at every other component is an identity.
\end{lemma}
\begin{proof}
By Axiom~\ref{ax:Dpush}, the square
\[
 \xymatrix@C=4em@R=2.5em{
A\ar[r]^i\ar[d]_a&X\ar[d]^c\\
B\ar[r]_j&Y
 }
\]
is bicartesian, so $c^*j=i$. Consequently $j$ is proper, since
the pullback of an identity is an identity. This also gives
$Y\ne\one$, because the only strong monomorphism into $\one$
is its identity. Suppose that
$j\inside r\ne\id_Y$, then
$i\inside c^*r\ne\id_X$. Indecomposability gives $c^*r=i$, and
Lemma~\ref{lem:strongpullbacks} gives $r=i^c=j$.
Thus there is no proper strong monomorphism strictly between $j$
and $\id_Y$, which says that $j$ is a component of $Y$.

Let $k$ be another component of $X$. Distinct components are disjoint
by Lemma~\ref{prop:paths}, so Lemma~\ref{lem:pushoutlocality}
makes $l=ck$ a strong monomorphism. To see that $l$ is proper,
suppose that $ck=\id_Y$. Then $c(kj)=j$, and the pullback square
just displayed factors $kj$ through $i$. Since $k$ and $j$ are
strong monic, their composite is a common strong subobject of $i$
and $k$. This contradicts their disjointness.
Choose a component $r$ with $l\inside r$. Then
$k\inside c^*r\ne\id_X$, so $c^*r=k$ and $r=k^c=l$.
Hence $l$ is a component and $c^*l=k$.
Together with $c^*j=i$, these equations make all the displayed
components distinct.

Finally, take a component $r\ne j$ of $Y$. Its pullback is proper,
so choose a component $k$ of $X$ above $c^*r$. If $k=i$, apply the order implication in
Lemma~\ref{lem:strongpullbacks} to $c^*r\inside i$. Together
with $(c^*r)^c=r$ and $i^c=j$, it gives $r\inside j$,
contradicting that $r$ and $j$ are distinct components. Otherwise
$r\inside k^c=ck$, and both are components, so $r=ck$.
Normal-factorization uniqueness gives the stated restrictions.
\end{proof}

Replacing several components at once requires a universal property
that takes all the prescribed replacements into account. We use the
following formulation.
\begin{definition}
For the component decomposition of $X$, consider a family of
morphisms $d_i:X_i\to Y_i$, indexed by $i\in\Comp(X)$.
A \emph{simultaneous pushout} of this family along the components
is a commutative family of squares
\[
\xymatrix@C=4em@R=2.5em{
X_i \ar[r]^{i} \ar[d]_{d_i} & X \ar[d]^{d} \\
Y_i \ar[r]_{j_i} & Y
}
\]
such that, for every object $T$, every map $f:X\to T$ and every
family $g_i:Y_i\to T$ satisfying $fi=g_id_i$, there exists a unique
$h:Y\to T$ with $hd=f$ and $hj_i=g_i$ for all $i$.
The universal property is imposed on the entire family of squares.
\end{definition}
\begin{remark}
An individual square in a simultaneous pushout need not itself be a
pushout. Its universal property requires compatibility with the maps
at all components.
\end{remark}

\begin{theorem}[Simultaneous component pushouts]\label{thm:componentpushouts}
For every $X\ne\one$ and every tuple of component epimorphisms
$d_i:X_i\to Y_i$, the simultaneous pushout exists, its object $Y$ is nonterminal,
and its lower family lists every component of $Y$ exactly once.
It can be constructed by successive one-component pushouts in any
order. Conversely, every epimorphism $d:X\to Y$ with nonterminal target is the map
from $X$ induced by such a simultaneous pushout.
\end{theorem}
\begin{proof}
Choose a listing of $\Comp(X)$, which is finite by
Axiom~\ref{ax:finite}, and apply Lemma~\ref{lem:componentpushout}
to the first component and its specified epimorphism. Apply the same
lemma to the images of the remaining components, in the chosen
order. Each step preserves the source objects of the components
still to be treated, and its restriction at each of them is an
identity. It also preserves nonterminality of the object being
constructed.

To verify the simultaneous universal property, take maps $f$ and
$g_i$ as in the definition. The first pushout gives a unique map
from its target to $T$. This map still has the prescribed values
on the components that remain to be treated, so it gives the
compatible maps required for the second pushout. Continuing in
this way gives a unique $h$ with all the required equations.
For an empty component family, the identity of $X$ has the
required universal property. Since the resulting diagram is
characterized by this universal property, rigidity identifies
the results obtained in different orders, including all their
structure maps.

Consider $\pi_u$ for a proper strong monomorphism $u$.
Factor $u$ into components by Lemma~\ref{prop:paths} and push
the terminal map of its source successively along the arrows in
that path, beginning at the source of $u$. At each step use
Axiom~\ref{ax:Dpush}. The composite of the resulting pushout squares
is a pushout of $q_{\dom(u)}$ along $u$, so its universal property
identifies it with the defining square for $\pi_u$.
The last square in this sequence pushes out the component whose
codomain is $X$. It is therefore a one-component pushout. Adding
identity restrictions at the other components makes it a simultaneous
component pushout with the same universal property.
For $u$ an identity, $\pi_u$ is the terminal map.

Simultaneous component pushouts also compose. Consider the following
family of diagrams, indexed by $i\in\Comp(X)$, and suppose that
each stage has the simultaneous universal property:
\[
\xymatrix@C=3.5em@R=2.5em{
X_i\ar[r]^i\ar[d]_{d_i}&X\ar[d]^d\\
Y_i\ar[r]_{j_i}\ar[d]_{e_i}&Y\ar[d]^e\\
Z_i\ar[r]_{k_i}&Z
}
\]
Given compatible maps $f:X\to T$ and $h_i:Z_i\to T$ for the
outer family, first apply the upper universal property to $f$
and $h_i e_i$. Its mediator, together with the maps $h_i$, then
satisfies the equations for the lower universal property. The
resulting unique map $Z\to T$ proves the universal property
for the outer family. Its restrictions are $e_i d_i$, and its
component bijection is the composite of the two component
bijections.

By Theorem~\ref{thm:finitegeneration}, every epimorphism is a finite
composite of maps of the form $\pi_u$. A composite with nonterminal target
cannot pass through $\one$, by Lemma~\ref{lem:empty}. Thus each
nonidentity factor is of the proper strong monomorphism case just described.
Composition proves the converse.
\end{proof}

\subsection{Describing epimorphisms by their restrictions}\label{sec:epi-restrictions}
There are now two ways an epimorphism can act on a nonterminal
object: it can replace the whole object by an empty leaf, or retain
the root and act separately on its children. In the second case,
the pruning chosen for each child determines the entire
epimorphism, by the simultaneous pushout description in
Theorem~\ref{thm:componentpushouts} and uniqueness of pushouts
in our rigid category. We shall use a separate notation for these
epimorphisms.
Let $\Icat$ consist of all objects of $\C$ and the epimorphisms
with nonterminal target together with $\id_\one$. This is a
subcategory: a composite with nonterminal target cannot pass
through $\one$, by Lemma~\ref{lem:empty}, and the indicated
identities and composites are again in the specified class.
Equivalently, it consists of the nonconstant epimorphisms and $\id_\one$.

The induced map $d:X\to Y$ of a simultaneous pushout of
epimorphisms is itself epic. Indeed, suppose that $hd=h'd$.
Then $hj_i d_i=h'j_i d_i$, so cancellation of each epimorphism
$d_i$ gives $hj_i=h'j_i$. The uniqueness clause of the
simultaneous universal property now gives $h=h'$.
For an empty component family, the pushout is an isomorphism,
hence the identity by Lemma~\ref{lem:rigidity}.

\begin{theorem}[Local classification of epimorphisms]\label{thm:epicclassification}
Under the standing assumptions, the following hold for every $X\ne\one$.
\begin{enumerate}
\item If $d:X\to Y$ is an epimorphism with $Y\ne\one$, the assignment
$\sigma_d(i)=i^d$ is a bijection
$\Comp(X)\to\Comp(Y)$.
Here $i^d$ is the strong monomorphism in the normal factorization
$di=i^d(d|_i)$ of~\eqref{eq:residual}: it records the position
in $Y$ of the child obtained by applying the epimorphism $d|_i$
to $X_i$.
\item The map
\begin{equation}\label{eq:epicclassification}
 \Theta_X:\Out(X)\longrightarrow
 \{*\}\sqcup\prod_{i\in\Comp(X)}\Out(X_i),\qquad
d\longmapsto
 \begin{cases}
 *,&\cod(d)=\one,\\
 (d|_i)_{i\in\Comp(X)},&\cod(d)\ne\one.
\end{cases}
\end{equation}
is well defined and bijective. The empty product is a singleton,
distinct from $*$.
\end{enumerate}
\end{theorem}
\begin{proof}
In a simultaneous pushout supplied by
Theorem~\ref{thm:componentpushouts}, the equation $di=j_i d_i$
is already a normal factorization. Hence $i^d=j_i$ and
$d|_i=d_i$. The theorem supplies every epimorphism with
nonterminal target in this way, and its lower family consists
of all the components of the target, each appearing once.
This proves \textup{(a)}.

For \textup{(b)}, every tuple of component epimorphisms admits
a simultaneous pushout. Two pushouts of the same tuple are
uniquely isomorphic, so rigidity identifies their targets and
all their structure maps. Thus the tuple determines exactly
one epimorphism with nonterminal target. Terminality supplies
the unique remaining epimorphism $q_X$, which corresponds to
the element $*$.
\end{proof}

\subsection{Root labels and input places}\label{sec:roots}
We now recover the basic operation symbol at the root of an
object. Replacing each of its children by an empty leaf leaves
precisely that symbol, together with its input places. The
preceding results allow us to carry out these replacements in
the abstract category. To justify arguments which pass from an
object to its children, we first count the proper strong
monomorphisms into an object. Recall that $\Occ(X)$ is the set
of all strong monomorphisms into $X$, including its identity.
\begin{definition}\label{def:complexity-normal}
The \emph{complexity} of $X$ is the number of proper strong
monomorphisms into $X$, namely $c(X)=|\Occ(X)|-1$.
An object is said to be \emph{normal} if its only outgoing morphism
in $\Icat$ is its identity. Thus a nonterminal normal object has
only two outgoing epimorphisms: its identity and its terminal map.
These will be the objects representing basic operation symbols.
\end{definition}

Every proper strong monomorphism into $X$ factors uniquely through
one member $i:X_i\to X$ of $\Comp(X)$, by
Lemma~\ref{prop:paths}. The strong monomorphisms through this
$i$ correspond to all members of $\Occ(X_i)$, including its identity.
Counting this partition gives
\begin{equation}\label{eq:complexity}
c(X)=|\Comp(X)|+\sum_{i\in\Comp(X)}c(X_i).
\end{equation}
In particular, a proper strong monomorphism $A\to X$ has
$c(A)<c(X)$. Passing to a component therefore decreases
complexity, so induction on $c(X)$ justifies the recursive
constructions below.

Recall that $\Icat$ consists of the epimorphisms with
nonterminal target, together with $\id_\one$. For a morphism
$d:X\to Y$ in this class, with $X\ne\one$, the bijection
$\sigma_d$ from Theorem~\ref{thm:epicclassification} sends
a component $i$ of $X$ to its resulting component $i^d$ in $Y$.
The restriction $d|_i$ describes the pruning of its source,
so that $di=\sigma_d(i)(d|_i)$. The next lemma describes how
these maps behave under composition.
\begin{lemma}\label{lem:transport}
Let $d:X\to Y$ and $e:Y\to Z$ be composable morphisms in
$\Icat$, with $X\ne\one$. Then, for each $i\in\Comp(X)$,
\[
 \sigma_{ed}=\sigma_e\sigma_d,\qquad
 (ed)|_i=(e|_{\sigma_d(i)})(d|_i).
\]
\end{lemma}
\begin{proof}
For $i:X_i\to X$, write $di=\sigma_d(i)(d|_i)$ and
$e\sigma_d(i)=\sigma_e(\sigma_d(i))(e|_{\sigma_d(i)})$.
Substitution gives
\[
edi=\sigma_e(\sigma_d(i))\,(e|_{\sigma_d(i)})(d|_i).
\]
The last two factors compose to an epimorphism, while the first
factor is an indecomposable strong monomorphism by
Theorem~\ref{thm:epicclassification}. Thus this is the normal
factorization of $edi$. Its uniqueness proves both equations.
\end{proof}

\begin{theorem}[Root normalization]\label{thm:root}
An object is normal if and only if all its component sources
are $\one$. Every object $X$ has a unique morphism
$\rho_X:X\to r(X)$ belonging to $\Icat$ and having a normal
target. For every morphism $d:X\to Y$ in $\Icat$,
\begin{equation}\label{eq:rootlaws}
r(Y)=r(X),\qquad \rho_Yd=\rho_X.
\end{equation}
If $X\ne\one$, the induced bijections of component sets satisfy
\begin{equation}\label{eq:roottransport}
 \sigma_{\rho_Y}\sigma_d=\sigma_{\rho_X}.
\end{equation}
In particular, $r(r(X))=r(X)$ and $\rho_{r(X)}=\id$.
\end{theorem}
\begin{proof}
For $X=\one$, all the assertions follow from
Lemma~\ref{lem:empty}. Suppose, then, that $X\ne\one$.
By Theorem~\ref{thm:epicclassification}, an outgoing
morphism in $\Icat$ is uniquely specified by a tuple with one
epimorphism of domain $X_i$ for each $i\in\Comp(X)$.
If all component sources are $\one$, each entry of the tuple
must be an identity. Hence the only outgoing morphism in
$\Icat$ is $\id_X$. Conversely, if a component source $X_i$
is nonterminal, its identity and its terminal map give two
different choices for that entry. Taking identities at the
other components gives two distinct outgoing morphisms, so
$X$ is not normal.

The tuple $(q_{X_i})_i$ determines a unique morphism $\rho_X$
in $\Icat$.
Its target has terminal component sources, because its
components are the targets of the chosen maps $q_{X_i}$.
It is therefore normal. Conversely, any morphism in $\Icat$
from $X$ to a normal object restricts at each component to
the unique map $X_i\to\one$. It has exactly the same tuple,
so Theorem~\ref{thm:epicclassification} proves uniqueness.
The composite $\rho_Yd$ is another morphism in $\Icat$ to a normal object,
so it equals $\rho_X$, including its target. Lemma~\ref{lem:transport}
then gives~\eqref{eq:roottransport}. Applying uniqueness to a normal
object proves the last assertions.
\end{proof}

\begin{lemma}\label{prop:thin}
For any objects $X,Y$, there is at most one epimorphism $X\to Y$.
In particular, every epimorphism $X\to X$ is an identity. A category
with at most one arrow between any two objects is said to be
\emph{thin}; thus $\Dcat$ is thin.
\end{lemma}
\begin{proof}
We argue by induction on $c(X)$. Consider epimorphisms
$d,d':X\to Y$. If $Y=\one$, terminality gives $d=d'$.
If $X=\one$, Lemma~\ref{lem:empty} makes both morphisms
identities. In the remaining case,
equation~\eqref{eq:roottransport} determines both component bijections
from their endpoints:
\[
 \sigma_d=\sigma_{d'}=
 \sigma_{\rho_Y}^{-1}\sigma_{\rho_X}.
\]
Consequently, for each $i\in\Comp(X)$, the restrictions of $d,d'$
have the same domain $X_i$ and the same codomain, namely the source
of $\sigma_d(i)=\sigma_{d'}(i)$. These domains have smaller complexity
by~\eqref{eq:complexity}, so the induction hypothesis identifies
the restrictions of $d$ and $d'$ at every component. Injectivity
of $\Theta_X$ in Theorem~\ref{thm:epicclassification} gives
$d=d'$. Taking $Y=X$ shows that every endo-epimorphism equals
the identity.
\end{proof}

The construction $X\mapsto r(X)$ removes all the expressions
inserted into the inputs of the operation at the root. Its
possible nonterminal values therefore describe the basic
operation symbols, each with all its inputs empty. We use
these objects as the labels in the recovered trees.
\begin{definition}\label{def:labels}
The \emph{root labels} are
\[
 \Phi=\{\phi\ne\one:r(\phi)=\phi\}.
\]
The \emph{input-place set} and \emph{arity} of a label are
\[
P_\phi=\Comp(\phi),\qquad \operatorname{ar}(\phi)=|P_\phi|.
\]
Each $p\in P_\phi$ is an arrow $\one\to\phi$.
\end{definition}

This gives the intrinsic description of basic operations stated
in the Introduction. More explicitly, $L$ is a root label if and
only if $L\ne\one$ and every proper strong monomorphism into $L$
has domain $\one$. Indeed, when $L$ is normal, its components
have terminal domains. Every proper strong monomorphism into
$L$ factors through a component, and the only strong
monomorphism into $\one$ is its identity. The converse follows
by applying the stated condition to the components of $L$ and
using Theorem~\ref{thm:root}. Moreover, every point of $L$ is
proper and strong, so the same argument makes it a component.
Consequently,
\[
P_L=\C(\one,L),\qquad \operatorname{ar}(L)=|\C(\one,L)|.
\]
A nullary label is a nonterminal object with no proper strong
monomorphisms. It represents a generating operation with no inputs.
The object $\one$ represents an empty input into which an
expression can be inserted. It will give the identity operation
of the free operad.

For nonterminal $X$, its root label $r(X)$ belongs to $\Phi$.
The epimorphism $\rho_X$ replaces all the components of $X$
by empty leaves. Its component bijection $\sigma_{\rho_X}$
therefore identifies the components of $X$ with the input places
of $r(X)$. We use this bijection to write
\begin{equation}\label{eq:ports}
i_p^X:X_p\longrightarrow X,\qquad
i_p^X=\sigma_{\rho_X}^{-1}(p).
\end{equation}
Every morphism $d:X\to Y$ in $\Icat$ preserves the root label
and these input places. Write $d_p:X_p\to Y_p$ for its
restriction at $i_p^X$, so that $d_p=d|_{i_p^X}$. Then
\begin{equation}\label{eq:local}
d i_p^X=i_p^Yd_p.
\end{equation}
Thus a nonterminal target of an epimorphism is obtained by
retaining the root label and choosing a pruning for the child
at each input place. We have also recovered how to identify
the input places at every occurrence of the same basic
operation. An order chosen on the inputs of that operation
can therefore be used wherever it occurs. We will use these
choices to recover the ordered expressions of the free operad.

\subsection{Cuts and the action of morphisms on positions}\label{sec:cuts}
Starting at the root, we reach any position by choosing a child,
then a child of that child, and so on. We record these choices
as an address. This will let us specify a pruning by the
positions at which it removes subtrees, and then describe
where a matching sends each position of its source.
\begin{definition}\label{def:addresses}
The \emph{addresses} of strong monomorphisms are defined recursively.
The identity of each object has the root address, which is the
empty word $\emptyword$. The object $\one$ has only this address.
For $X\ne\one$, an address below its $p$-component is $pa$, where
$a$ is an address in $X_p$. Let $\Pos(X)$ be the resulting finite
set of addresses, and let $X|_a$ be the domain of the strong
monomorphism at $a$.
Each letter records a specified input place of the corresponding
label.
\end{definition}

Unique component factorization identifies $\Pos(X)$ with $\Occ(X)$.
An address is a prefix of another precisely when its position is an
ancestor of the other position. If neither address is a prefix of
the other, the corresponding strong monomorphisms are disjoint.
A strong monomorphism at $a$ identifies the addresses of its source with those
of $X$ extending $a$.

To specify a pruning, it suffices to record the positions at which
whole subtrees are replaced by empty leaves. Once one such position
has been selected, selecting a descendant would have no further
effect. This motivates the following definition.
\begin{definition}\label{def:effective-cut}
An \emph{effective cut} in $X$ is a set $K\subseteq\Pos(X)$
whose members have nonterminal source and are pairwise incomparable
under the prefix order. Thus the selected strong monomorphisms
have pairwise disjoint positions. The empty set is also an effective cut.
If $K$ contains the root, then $K=\{\emptyword\}$.
\end{definition}

\begin{theorem}[Classification by cuts]\label{thm:cuts}
Epimorphisms out of $X$ are in bijection with effective cuts in $X$.
The empty cut gives $\id_X$; the root cut gives $q_X$ for $X\ne\one$.
For the epimorphism $\delta_K:X\to X_K$, each selected subtree is
replaced by $\one$, and the selected position is retained:
\begin{equation}\label{eq:pruning}
 \Pos(X_K)=\Pos(X)\setminus
 \bigcup_{a\in K}\{ab:b\ne\emptyword,\ ab\in\Pos(X)\}.
\end{equation}
Moreover,
\[
K=\{a\in\Pos(X_K):X|_a\ne\one,\ X_K|_a=\one\}.
\]
\end{theorem}
\begin{proof}
We argue by induction on $c(X)$. The object $\one$ has just
the empty cut and the identity epimorphism. Suppose that
$X\ne\one$. Under the bijection $\Theta_X$, the element $*$
corresponds to the terminal map $q_X$, which removes the entire
tree and hence corresponds to the root cut.

Every other epimorphism has a restriction at each component.
By induction, each restriction determines an effective cut in
that component. Prefix its addresses by the corresponding input
place and take the union over the components. The resulting set
is an effective cut in $X$, because the components are disjoint.
Conversely, a cut which does not contain the root determines one
cut in each component. Induction and
Theorem~\ref{thm:epicclassification} then give a unique
epimorphism in $\Icat$. These two constructions are inverse.

Each selected position becomes an empty leaf, so all its proper
descendants disappear and the position itself remains. The
epimorphisms in $\Icat$ preserve root labels and input places,
so this description applies recursively in each component and
gives~\eqref{eq:pruning}. Among the retained positions, precisely
the selected ones have changed from nonterminal subtrees to
empty leaves. This proves the displayed recovery formula for $K$.
\end{proof}

\begin{lemma}\label{prop:cutcomposition}
For a cut $K$ in $X$ and a cut $L$ in $X_K$, identify $L$ with its
retained addresses in $X$. Then
\begin{equation}\label{eq:cutcomposition}
 \delta_L\delta_K=\delta_{\operatorname{outer}(K\cup L)},
\end{equation}
where $\operatorname{outer}$ retains those addresses which have
no shorter prefix in the selected set.
\end{lemma}
\begin{proof}
If the second cut contains the root, both sides of
\eqref{eq:cutcomposition} are the unique map $X\to\one$. If the
first cut contains the root, the intermediate object is $\one$ and
the second cut is empty, giving the same conclusion. Otherwise both
epimorphisms preserve the root label. We argue by induction on
complexity and apply the formula to their restrictions at each
component. Within a component, a later cut at its root absorbs
all earlier selected descendants, and the other selected positions
remain. Thus the surviving cuts are exactly those members of
$K\cup L$ which have no shorter selected prefix. By
Theorem~\ref{thm:epicclassification}, the component restrictions
determine the global epimorphism, proving the formula.
\end{proof}

In particular, performing disjoint cuts in either order gives the
same morphism. A later cut at an ancestor of an earlier cut absorbs
that earlier cut, as expressed by~\eqref{eq:cutcomposition}.

\phantomsection\label{not:position-maps}
We can now describe the action of these maps on all positions.
An epimorphism $d:X\to Y$ induces a surjection
$p_d:\Pos(X)\to\Pos(Y)$ by sending every position inside a cut
subtree to the root of that subtree, which survives as an empty
leaf, and retaining all other addresses. A strong monomorphism
$u:A\to X$ at $a$ induces the injection $p_u(b)=ab$.
When $d$ belongs to $\Icat$ and $X\ne\one$, this gives
\[
p_d(\emptyword)=\emptyword,\qquad p_d(pb)=p\,p_{d_p}(b).
\]
The composite-cut formula gives $p_{ed}=p_ep_d$ for epimorphisms.
To compose a subtree occurrence with a pruning, we must determine
what remains of the selected subtree and where it lies after
pruning. The next lemma describes this in terms of addresses.

\begin{lemma}\label{lem:restriction}
Let $d:X\to Y$ be an epimorphism, and let $u:A\to X$ be a
strong monomorphism with address $a$. Write
$du=u^d(d|_u)$ for the normal factorization from
\eqref{eq:residual}. Then $u^d$ has address $p_d(a)$, and
\begin{equation}\label{eq:posrestriction}
p_dp_u=p_{u^d}p_{d|_u}.
\end{equation}
\end{lemma}
\begin{proof}
If $u=\id_X$, the restriction is $d$ and the residual map is
$\id_Y$, giving both assertions. If $d=q_X$, the equality
$q_Xu=q_A$ gives residual map $\id_\one$; both positional maps
are constant at the root of $\one$. In the remaining case, $d$
preserves the root and $u$ has a nonempty component path. Apply
\eqref{eq:local} at its first component, and continue with the
remaining component path. At each step, either the selected
subtree is removed, giving the terminal-map case already treated,
or its first input place is retained and the calculation passes
to the next component. The path is finite by
Lemma~\ref{prop:paths}. The resulting factorization has an epic
first map and a strong monic second map, so uniqueness identifies
it with~\eqref{eq:residual}. The same recursive calculation gives
its address as $p_d(a)$ and proves~\eqref{eq:posrestriction}.
\end{proof}

\begin{theorem}[Faithful positions and composition]\label{thm:positions}
The assignment sending $X$ to $\Pos(X)$ and a morphism $ud$
in normal form to $p_up_d$ is a faithful functor
$\Pos:\C\to\mathbf{FinSet}$. There is a canonical bijection
\begin{equation}\label{eq:normalhom}
 \C(X,Y)\cong\coprod_A\Dcat(X,A)\times\Ocat(A,Y).
\end{equation}
For composable morphisms $f=(d,u)$ and $g=(e,v)$ written in
normal form, their composite is
\begin{equation}\label{eq:composition}
gf=\bigl((e|_u)d,\ vu^e\bigr).
\end{equation}
\end{theorem}
\begin{proof}
The bijection~\eqref{eq:normalhom} is Lemma~\ref{prop:paths}(d).
For the composition law, factor the middle composite $eu$ in
the normal form shown in the square
\[
\xymatrix@C=3em@R=2.5em{
X\ar[r]^d&A\ar[r]^u\ar[d]_{e|_u}&Y\ar[d]^e&\\
&B\ar[r]_{u^e}&C\ar[r]^v&Z.
}
\]
Substituting $eu=u^e(e|_u)$ into $gf=veud$ gives
\eqref{eq:composition}.
The map $(e|_u)d$ is epic and $vu^e$ is strong monic, so this is
its normal factorization. Lemma~\ref{lem:restriction} gives
$p_ep_u=p_{u^e}p_{e|_u}$; together with the corresponding laws
for composites of epimorphisms and strong monomorphisms, this proves
that the positional map of $gf$ is the composite of those of $g,f$.
Identity normal forms act as identity functions, so these
assignments define a functor.

To prove faithfulness, observe that a pruning sends the source
root to the root of its target. Hence, for a morphism $ud$, the
image of the source root is the address of $u$. This determines
the strong monomorphism $u$ and its domain. By
Lemma~\ref{prop:thin}, there is at most one epimorphism from
the given source to that domain. Thus the positional function
determines both factors and therefore the morphism.
\end{proof}
\subsection{From singleton filling to assembly}\label{sec:filling}
Axiom~\ref{ax:U} allows an object to be inserted at one prescribed
point. To recover arbitrary expressions, we need to insert several
objects at different points of the same object. We will show that the
one-point requirement already supplies this simultaneous construction.
The pullback squares in the following definition specify the entire
object inserted at each point. We seek a filling in which these
insertions account for every change to the original object. Its
defining requirement is that every other filling admits exactly one
map to it preserving the inserted objects and the map back to the
original object.
\begin{definition}\label{def:finitefilling}
A finite \emph{marked cocone of constant morphisms} consists of a
finite family of objects $(A_i)_{i\in I}$ and distinct points
$(j_i:\one\to B)_{i\in I}$, with legs $j_iq_{A_i}:A_i\to B$.
The point $j_i$ records where the object $A_i$ is to be inserted.
An object of $\Fill((A_i,j_i)_{i\in I})$ consists of an object
$C$, an epimorphism $d:C\to B$, and strong monomorphisms
$u_i:A_i\to C$ such that, for every $i\in I$, the square
\[
 \xymatrix@C=4em@R=2.5em{
A_i\ar[r]^{u_i}\ar[d]_{q_{A_i}}&C\ar[d]^d\\
 \one\ar[r]_{j_i}&B
 }
\]
is a pullback. Equivalently, each $(C,d,u_i)$ is an object of
$\Fill(A_i,j_i)$ from Axiom~\ref{ax:U}.
A morphism
$h:(C',d',(u'_i)_i)\to(C,d,(u_i)_i)$ is a morphism
$h:C'\to C$ in $\C$ satisfying $dh=d'$ and $hu'_i=u_i$ for
every $i$. We also call it a \emph{comparison of fillings}.
Composition and identities are those of $\C$.
A terminal object is called a \emph{universal filling} of the cocone.
\end{definition}

Thus a comparison of fillings makes the following diagram commute
for every marked point:
\[
\xymatrix@C=3.5em@R=2.5em{
A_i\ar[r]^{u'_i}\ar@/_0.5pc/[dr]_{u_i}
  &C'\ar[r]^{d'}\ar[d]^h&B\ar@{=}[d]\\
&C\ar[r]_d&B
}
\]
The comparison preserves both the inserted object and the map
which removes it again.

By Lemma~\ref{lem:rigidity}, every point $j:\one\to B$ is a
strong monomorphism. Lemma~\ref{lem:empty} also shows that $j$ is
determined by the composite $jq_A$, since $q_A$ is epic. Thus the
point in each marked leg is recovered from that leg and its domain.

\begin{lemma}[Exact point fibres]\label{lem:fibres}
Let $d:C\to B$ be an epimorphism and let $j:\one\to B$ have address $a$.
There is a unique strong monomorphism $w:W\to C$ whose address is $a$ and for which
\[
 \xymatrix@C=4em@R=2.5em{
W \ar[r]^{w} \ar[d]_{q_W} & C \ar[d]^{d} \\
  \one \ar[r]_{j} & B
 }
\]
is a pullback. If $a$ is one of the cut addresses of $d$,
then $W=C|_a$ is the whole subtree erased there. Otherwise
$W=\one$ is the unchanged empty leaf.
\end{lemma}
\begin{proof}
By the description of prunings, every address of $B$ is retained
from $C$. The positions sent to $a$ by $p_d$ form the entire subtree
at $a$: this is the subtree removed by a cut at $a$, or the unchanged
empty leaf if there is no such cut. Let $w:W\to C$ be its occurrence.
Then $dw=jq_W$.

To verify the pullback property, suppose that $h:T\to C$ satisfies
$dh=jq_T$, and write $h=ve$ in normal form. The root of $T$ is
sent by $h$ to the address of $v$, and the displayed equality sends
this address to $a$. Thus the address of $v$ belongs to the subtree
selected by $w$. Its component path consequently extends the path
of $w$, so $v$, and hence $h$, factors through $w$. Monicity gives
uniqueness of this factorization. Its composite with $q_W$ is $q_T$
by terminality, which verifies the remaining pullback equation.
Finally, any other pullback is isomorphic to this one over $C$;
rigidity makes that isomorphism an identity.
\end{proof}

\begin{lemma}\label{lem:comparisons}
If $B\ne\one$, every comparison in a filling category is an epimorphism.
If $B=\one$ and one point is marked, every filling whose square
is a pullback has
$C=A$, $u=\id_A$, and $d=q_A$.
\end{lemma}
\begin{proof}
Suppose first that $B\ne\one$. A comparison $h:C'\to C$
between fillings with epimorphisms $d':C'\to B$ and
$d:C\to B$ satisfies $dh=d'$. Write $h=ve$ in normal form.
Both $d,d'$ belong to $\Icat$ because their target $B$ is
nonterminal. If $v$ were proper, its address would have a first
component place, and $d$ would preserve that place. Hence $dh$
would send the root of $C'$ to a proper address of $B$. But $d'$
preserves the root, contradicting $dh=d'$. Thus $v$ is an
identity and $h=e$ is an epimorphism.
When $B=\one$, its unique point is $\id_\one$. The pullback of
this identity along $d$ makes $u$ an isomorphism. By rigidity it
is an identity, and terminality then gives $d=q_A$.
\end{proof}

\begin{lemma}[Structure of a single-point filling]\label{lem:singlecut}
Let $(C,d,u)$ be a universal filling of $A$ at a point
$j:\one\to B$. Then the only possible effective cut of $d:C\to B$ is at the
address of the marked point. If $A\ne\one$, that cut occurs and
its source subtree is $A$. If $A=\one$, the universal filling is
$(B,\id_B,j)$. Every other point of $B$ lifts uniquely to an unchanged
point of the filling object.
\end{lemma}
\begin{proof}
The case $B=\one$ follows from Lemma~\ref{lem:comparisons}.
Suppose therefore that $B\ne\one$. Lemma~\ref{lem:fibres}
identifies the complete marked fibre with $A$. Starting with a
universal filling $(C,d,u)$, erase all the cut subtrees
of $d$ except this fibre. This gives an epimorphism $e:C\to C_0$ and a
factorization $d=d_0e$, with a retained strong monomorphism $u_0:A\to C_0$
satisfying $eu=u_0$. If $A=\one$, so that the marked fibre is already terminal,
erase all cuts.
The cuts erased by $e$ are incomparable with the marked position, so
Lemma~\ref{lem:fibres} makes $(C_0,d_0,u_0)$ another filling with the required pullback square.

Terminality supplies a comparison $h:C_0\to C$ with
$dh=d_0$ and $hu_0=u$. Thus $he$ is a comparison from the
terminal filling to itself, and consequently $he=\id_C$.
Composing this equation with $e$ gives $(eh)e=e$; epicity of $e$
then gives $eh=\id_{C_0}$. By Lemma~\ref{lem:rigidity}, $e$ is
an identity.
Every effective cut performed by $e$ would make it nonidentity.
Consequently all effective cuts of $d$ are accounted for by the
marked fibre. If this fibre is terminal, $d$ is an identity and
the filling is $(B,\id_B,j)$. At every other point of $B$, the
empty leaf is unchanged, so it has exactly one point above it.
\end{proof}

\begin{theorem}[Finite filling from singleton filling]\label{thm:finite}
Every finite marked cocone of constant morphisms has a unique universal filling. It replaces exactly
the marked empty leaves by their assigned objects. It may be constructed
by successive singleton fillings in any order, and the result, including
its structure arrows, is independent of that order.
\end{theorem}
\begin{proof}
For the empty family, $(B,\id_B)$ is terminal: the comparison from
any epimorphism $d':C'\to B$ has to be $d'$ itself. If $B=\one$,
there is at most one marked point, and the assertion follows from
Lemma~\ref{lem:comparisons}.

Assume $B\ne\one$ and induct on the size of the family. Fill
the first point, obtaining a universal filling $(C_1,d_1,u_1)$
with $d_1:C_1\to B$. By Lemma~\ref{lem:singlecut}, each
remaining point has a unique empty lift in $C_1$. Fill those points
successively. At each stage Lemma~\ref{lem:singlecut} leaves every previously
inserted subtree unchanged. Lemma~\ref{lem:fibres} therefore
identifies the fibre over each original marked point with its
assigned object. Consequently the resulting strong monomorphisms
and composite epimorphism give pullback squares for all the original
points.

Consider another filling $(D,d',(u'_i)_i)$ of the entire marked
family. The universal property for the first point gives a unique
$h_1:D\to C_1$ satisfying $d_1h_1=d'$ and preserving the
first prescribed strong monomorphism. It is an epimorphism by
Lemma~\ref{lem:comparisons}. For each remaining point $j_i$,
write $\widetilde j_i:\one\to C_1$ for its unchanged lift.
Thus $d_1\widetilde j_i=j_i$, and we have the diagram
\[
 \xymatrix@C=4em@R=2.5em{
A_i\ar[r]^{u'_i}\ar[d]_{q_{A_i}}&D\ar[d]^{h_1}\\
 \one\ar[r]_{\widetilde j_i}\ar[d]_{\id_\one}
     &C_1\ar[d]^{d_1}\\
 \one\ar[r]_{j_i}&B
 }
\]
The lower square is a pullback by Lemma~\ref{lem:singlecut},
and the outer rectangle is a pullback by the definition of the
given filling of the marked family.
The upper square is consequently a pullback: a compatible pair
for that square is a compatible pair for the rectangle, whose
unique mediator satisfies the middle equation by the lower
pullback property.
Thus $(D,h_1,u_i')$ satisfies the pullback requirement for each
remaining point $\widetilde j_i$. Its comparison to the next
singleton filling is therefore supplied by Axiom~\ref{ax:U}.
Repeating this argument constructs a comparison to the final object.

We must also check that a comparison obtained at a later step
preserves the previously prescribed maps $u'_j:A_j\to D$. Each later
cut is disjoint from the already inserted subtree with source $A_j$.
Its inverse image under that pruning is therefore the same entire
subtree, with identity restriction. The corresponding square, with
left map $\id_{A_j}$, is a pullback: a map whose composite lands
in that retained subtree has embedding factor in the same subtree,
as in the proof of Lemma~\ref{lem:fibres}. The comparison equation
at the preceding stage and this pullback property force the later
comparison to preserve the specified map from $A_j$. Conversely, any comparison to the final
filling composes to a comparison at each preceding stage. The
uniqueness at each singleton filling therefore proves uniqueness
of the final comparison. Thus the resulting filling is terminal,
and its construction changes precisely the marked points.

Any two terminal fillings are uniquely isomorphic as fillings.
Their underlying isomorphism in $\C$ is an identity, which identifies
the objects and all structure arrows. This proves uniqueness and
independence of the chosen order of construction.
\end{proof}

We now apply the filling construction to a root label $\phi$.
Recall that its input places $P_\phi$ are its points
$\one\to\phi$, and that $X_p$ denotes the child of $X$ at
the input place $p$ of its root label $r(X)$, as in
\eqref{eq:ports}.
\begin{theorem}[Assembly]\label{thm:assembly}
For every root label $\phi\in\Phi$ and every family
$(A_p)_{p\in P_\phi}$ of objects, there exists exactly one
object $X$ with $r(X)=\phi$ and $X_p=A_p$ for every
$p\in P_\phi$.
\end{theorem}
\begin{proof}
Mark the component points $p:\one\to\phi$
and assign $A_p$ to $p$. Theorem~\ref{thm:finite} supplies an
epimorphism $d:X\to\phi$ and, for every $p\in P_\phi$, a
pullback square
\[
 \xymatrix@C=4em@R=2.5em{
A_p\ar[r]^{u_p}\ar[d]_{q_{A_p}}&X\ar[d]^d\\
 \one\ar[r]_p&\phi
 }
\]
Since $\phi$ is nonterminal, $d\in\Icat$, and
Theorem~\ref{thm:root} gives $r(X)=\phi$. By the component bijection
of Theorem~\ref{thm:epicclassification}, the pullback of $p$ is
the component indexed by $p$. The displayed squares therefore give
$X_p=A_p$.
For uniqueness, let $X'$ be another object with this root and
these components. The epimorphism $\rho_{X'}:X'\to\phi$
replaces its children by empty leaves. Together with the
component maps, it gives another object of the same filling
category, since its squares at the points of $\phi$ are
pullbacks. By Lemma~\ref{lem:comparisons}, the terminal
comparison $h:X'\to X$ is an epimorphism. Each of its
restrictions $h_p:A_p\to A_p$ is therefore an identity by
Lemma~\ref{prop:thin}. Thus $h$ has the same tuple of
restrictions as $\id_{X'}$. Theorem~\ref{thm:epicclassification}
gives $h=\id_{X'}$, and hence $X'=X$. This also covers the
empty family for a nullary label.

\end{proof}

\phantomsection\label{not:assembly-operation}
The theorem allows us to use the generating label $\phi$ as an
operation on objects. Write $\phi((A_p)_{p\in P_\phi})$ for the
unique object assembled from its prescribed components.
In particular,
\[
 \phi((\one)_{p\in P_\phi})=\phi,\qquad
X=r(X)((X_p)_{p\in P_{r(X)}})\quad(X\ne\one).
\]
The second equality expresses every nonterminal object in terms of
its root label and immediate components. Since the components have
smaller complexity, repeated application recovers a finite expression
built from the generating labels.

\subsection{Representation}\label{sec:representation}
We can now identify both the objects and the morphisms of $\C$.
Assembly recovers the expressions, while the factorization of
morphisms recovers the pruning and occurrence in each match.
Recall that $\Phi$ is the set of root labels from
Definition~\ref{def:labels}. For a label $\phi$, its input-place
set $P_\phi=\Comp(\phi)$ consists of the component maps into
$\phi$, each of which has domain $\one$. These are the data
from which we will construct the expression trees.
\begin{theorem}[Representation]\label{thm:representation}
Every small category satisfying Axioms~\ref{ax:N}--\ref{ax:Ddetect}
is isomorphic to the prefix--suffix matching category of its derived
finite-arity signature
\[
L=\Phi,\qquad \operatorname{ar}(\phi)=|\Comp(\phi)|.
\]
The isomorphism identifies all morphisms with prefix--suffix matchings,
strong monomorphisms with specified occurrences of entire subtrees,
and epimorphisms with pruning maps. Conversely, every such expression
tree category satisfies the nine axioms.
\end{theorem}
\begin{proof}
Let $L=\Phi$ with the arity function in the statement. We first
use the component maps themselves as the input places of each label.
Choosing a bijection with $\{1,\ldots,\operatorname{ar}(\phi)\}$
will then give the standard presentation of $\C_L$.

For a nonterminal object $X$, recall that $r(X)$ is the root label
obtained by clearing its immediate components, and $X_p$ is the
component at the input place $p$ of this label. Define a correspondence
on objects recursively by sending $\one$ to the
empty leaf, and by sending $X\ne\one$ to the tree with root
label $r(X)$ and with the tree already associated to $X_p$ at each place $p$.
The recursion terminates because $c(X_p)<c(X)$ by
\eqref{eq:complexity}. Conversely, given a finite tree over $L$,
reconstruct its children first and apply Theorem~\ref{thm:assembly}
to their root label and the resulting indexed family. This gives a
unique object. Induction on the tree, using the uniqueness part of
that theorem at each root, shows that these two constructions
are inverse functions on objects.

Lemma~\ref{prop:paths} identifies strong monomorphisms with
specified component paths. Under this correspondence, these are
exactly the subtree embeddings. Theorem~\ref{thm:root} shows that
an epimorphism with nonterminal target preserves the root label and
its input places. The local classification of
Theorem~\ref{thm:epicclassification} then identifies such a map,
recursively, with an arbitrary tuple of child prunings; the remaining
case is total pruning to $\one$. Thus we obtain bijections
on the two arrow classes, with their sources and targets preserved.
By the unique normal factorization of Lemma~\ref{prop:paths},
these two bijections combine to give a bijection on every hom-set.

Finally, restrictions of epimorphisms at components are carried to
restrictions of prunings by~\eqref{eq:local}; iterating along a
component path gives the same statement for every subtree embedding.
The composition formula~\eqref{eq:composition} therefore agrees
with the defining composition in $\C_L$. Identities correspond to
identity normal forms. This correspondence is consequently an isomorphism
of categories. The converse assertion is Theorem~\ref{thm:realization}.
\end{proof}

The theorem recovers the generating operation symbols from the
category itself. They are the nonterminal objects whose proper
strong monomorphisms all have terminal domain. Each of these maps
is one of the input places of its label, so its arity is the number
of such maps. Thus the category determines a set equipped with
an arity function, usually called a \emph{ranked set} or
\emph{ranked signature}. It also retains the actual input-place
sets. To identify individual expressions with operations in a
free non-symmetric operad, we next choose an ordering of these
places and show that the choices extend consistently to all objects.

\subsection{Ordering and recovery of the operad}\label{sec:ordering}
The same generating operation may occur at several vertices of a
tree, and it may occur in many different objects. We want an order
of its arguments that is used consistently at all these occurrences.
This is the sense in which the ordering below will be uniform.
We will show that such choices also order all positions of every
tree, so that a matching preserves their order even when it sends
several positions to the same empty leaf. We first describe the
action on positions in categorical terms.

Recall that $\Occ(X)$ is the set of strong monomorphisms with
codomain $X$. Their identification with addresses gives the set
$\Pos(X)$ of positions from Definition~\ref{def:addresses}.
\phantomsection\label{not:occurrence-images}
For a morphism $f:X\to Y$ and a strong monomorphism $u:A\to X$,
let $f_*u$ be the strong monomorphism in the unique factorization
of $fu$ as an epimorphism followed by a strong monomorphism.
Under $\Occ(X)=\Pos(X)$, the function $f_*$ is the positional
map of Theorem~\ref{thm:positions}. In particular, $(gf)_*=g_*f_*$.
If $f$ is a strong monomorphism, then $f_*u=fu$.
The direct image $f_*u$ records the position occupied by the part
of $A$ retained by $f$.
An \emph{order embedding} is an injective map that preserves and
reflects the order. Its image is an \emph{interval} if every element
lying between two elements of the image also belongs to the image.

\begin{theorem}[Uniform ordering]\label{thm:uniform-ordering}
Let $\C$ be a small category satisfying the nine axioms. Choose a
linear order on each finite input-place set $P_\phi$ of its derived
signature. These choices determine linear orders $\le_X$ on
$\Occ(X)$, for all objects $X$, with the following properties.
\begin{enumerate}
\item The order $\le_X$ extends the factorization order:
if $u=vw$ for strong monomorphisms $u,v,w$, then $u\le_X v$.
\item For every morphism $f:X\to Y$, the map
$f_*:(\Occ(X),\le_X)\to(\Occ(Y),\le_Y)$ is monotone.
\item For every strong monomorphism $m:A\to X$, the map $m_*$
is an order embedding whose image is an interval.
\end{enumerate}
At every vertex labeled $\phi$, the induced order of its immediate
children is the chosen order on $P_\phi$. Thus the choices are uniform
across all occurrences of each label. Moreover, the positional functor
lifts to a faithful functor
\[
 \Pos_{\le}:\C\longrightarrow\mathbf{FinOrd},
\]
where $\mathbf{FinOrd}$ is the category of finite linearly ordered
sets and monotone maps.
\end{theorem}
\begin{proof}
Let the positions be the intrinsic addresses of
Definition~\ref{def:addresses}. We define their order recursively,
placing every vertex after all the positions in its children.
More explicitly, $\Pos(\one)$ has its unique
order. For $X\ne\one$, write the places of $r(X)$ as
$p_1<\cdots<p_n$. List the positions in $p_1\Pos(X_{p_1})$ in
their recursively defined order, then those in
$p_2\Pos(X_{p_2})$, and so on, and finally the root position.
For $n=0$ this is again a singleton. The construction terminates
because the complexity of each child is smaller than that of its
parent. The resulting order is usually called \emph{postorder}:
the children are read in order, each with all of its descendants,
before their parent is read.

Every descendant precedes its proper ancestors, so this order extends
factorization, proving \textup{(a)}. At any address $a$, the positions
extending $a$ form a contiguous block ending at $a$. Removing the
prefix $a$ identifies their order with the recursively defined order
on $\Pos(X|_a)$. A strong monomorphism acts by adjoining precisely
such a prefix. This proves \textup{(c)}, as well as uniformity of the
orders on immediate children.

By Theorem~\ref{thm:cuts}, an epimorphism is determined by an
effective cut, that is, a family of pairwise disjoint nonterminal
subtrees to be replaced by empty leaves. Its positional map
collapses each selected subtree to its boundary
position. The cut positions are pairwise incomparable, so their
subtrees occupy disjoint intervals. The target's postorder is exactly
the order obtained by replacing each such interval by its final
position; retained labels have the same ordered input places.
To see that the resulting map is monotone, take two positions in
the source order. If they belong to the same collapsed interval,
their images are equal. Otherwise the disjoint intervals and
retained positions occur in the same order in the target, so their
images remain in that order. Total pruning is the same construction
with the entire ordered set as one interval.

Every morphism is an epimorphism followed by a strong monomorphism,
and its positional map is the corresponding composite. The preceding
two paragraphs prove \textup{(b)}. Functoriality and faithfulness follow
from Theorem~\ref{thm:positions}, since the underlying positional maps
have not changed.
\end{proof}

\begin{theorem}[Ordered-tree representation]\label{cor:ordered-representation}
A small category satisfies the nine axioms if and only if it is
isomorphic to a prefix--suffix matching category of finite ordered
rooted labeled trees over a finite-arity signature, with an empty
leaf allowed at each input place. Here a vertex with label $\phi$
has exactly $\operatorname{ar}(\phi)$ children in a specified linear
order, and matchings preserve labels and corresponding child places.
Every choice of orders on the derived sets $P_\phi$ gives such a
representation, in which the orders of
Theorem~\ref{thm:uniform-ordering} are the postorders of the
represented trees.
\end{theorem}
\begin{proof}
For each label $\phi$, number the elements of its chosen ordered
set $P_\phi$ increasingly from $1$ to $\operatorname{ar}(\phi)$. At every vertex, this replaces its indexed
children by an ordered list, using the same numbering at every
occurrence of $\phi$. The two constructions in the proof of
Theorem~\ref{thm:representation} are then mutually inverse on these
ordered trees. Their description of morphisms preserves the child
places, and hence their order. The recursive definition in the
preceding proof gives the asserted correspondence of postorders.
Conversely, an ordered-tree category is the concrete category
$\C_L$ with its standard ordered index sets, so
Theorem~\ref{thm:realization} supplies the axioms.
\end{proof}

We can now describe what it means to insert one expression into
another using only the category and the chosen argument orders.
Doing this at all the inputs will recover the composition of
operations in the operad discussed in the Introduction. The inputs
of a composite expression are its remaining
empty leaves. These are the points of the corresponding object,
ordered by restricting the order of Theorem~\ref{thm:uniform-ordering}.
In particular, $\one$ has one input, represented by its identity
morphism. A nullary generating label has no points and therefore
has no inputs. The distinction agrees with substitution: filling
the input of $\one$ by $A$ gives $A$, while a nullary operation
accepts no arguments.

We will also show that assigning to each generator an operation of
the same arity in any other operad determines exactly one interpretation of all expressions
that respects insertion and the identity. Such an interpretation,
preserving the arity of each operation, is an \emph{operad map}.
This is the meaning of freeness described earlier.

\begin{theorem}[Recovery of the free operad]\label{thm:operad-recovery}
Let $\C$ satisfy the nine axioms, and choose a linear order on each
derived input-place set $P_\phi$. There is a free non-symmetric
operad $\mathcal O_\C$ on the derived ranked set $\Phi$ with the
following description.
\begin{enumerate}
\item Its operations of arity $n$ are the objects $X$ with exactly
$n$ points. Its identity operation is $\one$.
\item If the points of $X$ are $j_1<\cdots<j_n$, the composite
$\gamma(X;A_1,\ldots,A_n)$ is the universal filling of $A_i$ at
$j_i$ for $1\le i\le n$. Its input list is obtained by concatenating
the ordered input lists of $A_1,\ldots,A_n$.
\item The category $\C$ is isomorphic to the prefix--suffix
matching category of the operations of $\mathcal O_\C$.
\end{enumerate}
The isomorphism class of this operad is independent of the chosen
orders. Moreover, if $\C$ is the expression tree category of a
free non-symmetric operad, that operad is recovered up to isomorphism.
\end{theorem}
\begin{proof}
Use the ordered-tree representation of
Theorem~\ref{cor:ordered-representation}. Points correspond to
empty leaves, and their order is their left-to-right order in the
represented tree. By Theorem~\ref{thm:finite}, the proposed
composite is obtained by grafting the tree of $A_i$ into the
$i$th empty leaf of $X$. Every empty leaf of the result belongs
to exactly one of the inserted trees. The order on these leaves
first follows the order of insertion places and then the order
within each inserted tree. Thus, if $A_i$ has $n_i$ points, the
composite has $n_1+\cdots+n_n$ points, in the order asserted.
When $X$ has no points, this is the empty filling, so
$\gamma(X;())=X$ by Theorem~\ref{thm:finite}.

Filling all the points of $X$ by $\one$ leaves $X$ unchanged,
and filling the unique point of $\one$ by $A$ gives $A$. These
are the two identity laws. For associativity, consider objects
$B_{i,1},\ldots,B_{i,n_i}$ to be inserted into the points of
$A_i$. One can first make these insertions within each $A_i$
and then insert the results into $X$. Alternatively, one can
insert the $A_i$ into $X$ first and then fill their points in
the resulting tree. Both constructions retain exactly the same
labeled vertices, with the same input places, and place each
$B_{i,j}$ at the same composite address. They consequently give
the same ordered tree. The representation is injective on
objects, so they give the same object of $\C$. This proves the
associative law for simultaneous operadic composition, including
the cases in which an inserted object has no inputs.

We next show that every assignment of values to the generators
extends uniquely to the expressions, as described above.
Let $Q$ be any non-symmetric operad, and assign to each
$\phi\in\Phi$ an operation $h(\phi)$ of $Q$
with arity $|P_\phi|$. We seek an operad map $\widehat h$ making
the following diagram commute, where the horizontal map includes
the generators among the operations:
\[
\xymatrix@C=4em@R=2.5em{
\Phi\ar[r]\ar@/_0.8pc/[dr]_h&\mathcal O_\C\ar@{-->}[d]^{\widehat h}\\
&Q
}
\]
Here both operads are viewed as ranked sets, so each arrow preserves
arity; the map $\widehat h$ must also preserve composition and
the identity. Define it recursively as follows. Send $\one$ to the identity
operation of $Q$. For $X\ne\one$, list the input places of
$r(X)$ as $p_1<\cdots<p_k$ and put
\[
 \widehat h(X)=
 h(r(X))\bigl(\widehat h(X_{p_1}),\ldots,
                         \widehat h(X_{p_k})\bigr).
\]
The right-hand side is composition in $Q$, and the recursion
terminates by complexity descent. Induction shows that
$\widehat h(X)$ has arity equal to the number of points of $X$,
since operadic composition adds the arities of the inserted
operations. Induction on the outer tree,
using the associative law in $Q$, shows that evaluating a grafted
tree is the same as composing the evaluations of its constituent
trees. Hence $\widehat h$ preserves operadic composition and the
identity. For a generating label, every child is $\one$, so the
identity laws in $Q$ give $\widehat h(\phi)=h(\phi)$.
Every object has its unique root decomposition from
Theorem~\ref{thm:assembly}; any operad map extending $h$ must
therefore satisfy the displayed recursion. This proves the
required uniqueness, and hence freeness. The identification
of matching categories in \textup{(c)} is now precisely the
ordered-tree representation.

For two choices of orders, the two resulting operads have the
same generating operations $\Phi$ with the same arities. The
extension property just proved supplies a
unique operad map in each direction fixing every generator.
Their composites fix the generators and are therefore identities.
Thus the two operads are isomorphic. Finally, if $\C$ was obtained
from a free operad on a ranked set $L$, the single-generator trees
are exactly the derived labels: a tree with more than one labeled
vertex has a proper nonterminal subtree, whereas a
single-generator tree has only empty proper subtrees. Its points
are its prescribed argument places. Hence the recovered ranked
set is isomorphic to $L$. Extending this bijection and its inverse
to the two operads, as in the preceding argument, recovers the
original operad up to isomorphism.
\end{proof}

\begin{remark}[Existence and choice of order]\label{rem:order-choice}
The axioms determine the isomorphism class of the free operad, while
the chosen input orders specify how each individual object is read
as an operation of that operad. In general this second correspondence
depends on the choices.
For example, take the signature with one binary label $\phi$.
Reflecting every tree and exchanging the two child places in every
address defines an automorphism of its matching category. It fixes
$\one$ and the tree $\phi(\one,\one)$ consisting of one
labeled root and two empty children, but exchanges its two points.
Such a tree, with one labeled vertex and all inputs empty, is
often called a \emph{corolla}. Its two points therefore admit no
linear order invariant under every automorphism of the category.
Rigidity concerns isomorphism arrows inside the category, whereas
the reflection here is an automorphism of the category itself.
The construction in Theorem~\ref{thm:operad-recovery} makes one
choice per label and uses it at every occurrence of that label.
In this sense the category remembers a free non-symmetric operad
before its input places have been oriented.
\end{remark}

\section{Examples}\label{sec:examples}
The representation theorem allows us to pass between an abstract
category satisfying the axioms and a concrete set of generating
operations with prescribed arities. We first consider two simple
choices of generators. We then give a syntax of bracketed expressions
whose associated trees are recovered by the theorem. Finally, we
show that pointedness singles out the operads generated by unary
operations and leads to the prefix--suffix matching of words.

\subsection{Some choices of signature}\label{sec:signature-examples}
\begin{example}[A binary label]
Let the signature consist of a single label $m$ of arity two, with
input places numbered $1$ and $2$. Its objects are given recursively by
$1$ and the expressions $m(A,B)$. Thus every labeled vertex has two
children, while an empty leaf records a place at which another tree
may be inserted. For example, the two components of $m(1,1)$ give
distinct points
\[
\xymatrix@C=3em{
1\ar@/^/[r]^{j_1}\ar@/_/[r]_{j_2}&m(1,1)
}
\]
Filling the first point by $A$ and the second by $B$ gives $m(A,B)$.
When $A=B$, the two occurrences still select different input places
and hence give distinct morphisms. This category satisfies all nine axioms by
Theorem~\ref{thm:realization}. Its terminal object is not initial,
since $m(1,1)$ has two points.
\end{example}

\begin{example}[Nullary labels]
Suppose that every label in a set $L$ has arity zero. The objects are
then $1$ and the single-vertex trees $a$, for $a\in L$. Each such
$a$ has its identity and the pruning $q_a:a\to1$. These, together
with $\id_1$, are all the morphisms: a nonempty overlap would have
to retain the same nullary label, and an empty overlap would have
to select an empty leaf of the target. The object $a$ has no empty
leaf. In particular, there is no morphism $1\to a$, and there is
no morphism $a\to b$ for distinct labels $a,b$. This is a poset
regarded as a category, with greatest element $1$ and otherwise
incomparable elements indexed by $L$.

For the empty signature, the only object is $1$, and its identity
is the only morphism. It is the simplest example satisfying all
nine axioms.
\end{example}

\subsection{Bracketed expressions}\label{sec:bracketed}
The representation theorem also applies when the objects are
initially given as expressions. The following syntax makes each
argument place explicit by a pair of brackets. Its categorical
description then recovers the expression tree of each string,
together with the pruning and occurrence described by each morphism.

\begin{definition}[Bracketed expressions]\label{def:bracketed}
Fix a set $\Sigma$ of symbols, disjoint from the two bracket symbols
\texttt{[} and \texttt{]}. Let $\mathcal B_\Sigma$ be the set of finite
words in these symbols whose brackets are properly balanced, including the empty
word $\varepsilon$. Write $\Sigma^*$ for the set of finite words
formed from symbols of $\Sigma$ alone, also including the empty word.
Reading the outermost matched bracket pairs from left to right gives
every nonempty expression a unique decomposition
\begin{equation}\label{eq:bracket-decomposition}
E=w_0[E_1]w_1\cdots[E_n]w_n,
 \qquad w_i\in\Sigma^*,\quad E_i\in\mathcal B_\Sigma,
\end{equation}
where the displayed brackets are its outermost matched pairs.
Its \emph{root pattern} is
\[
s(E)=w_0[]w_1\cdots[]w_n.
\]
Write $L_\Sigma$ for the set of all nonempty root patterns. A pattern
with $n$ bracket pairs has arity $n$, its places numbered from left
to right. When $n=0$, the expression is a nonempty bracket-free word
and is itself a nullary pattern.
\end{definition}

The complete root patterns serve as the generating operation
symbols. For example, \texttt{[]+[]} is binary, while
\texttt{a+b} is a single nullary pattern, since its symbols lie
outside any bracket pair. The expression
\texttt{[]} is a unary pattern with empty content and is distinct
from the empty expression $\varepsilon$.
The left-to-right orders of the bracket pairs provide the choices
in Theorem~\ref{thm:uniform-ordering}. The resulting ordered trees
are the expression trees of this syntax, and the root patterns
freely generate their operad by Theorem~\ref{thm:operad-recovery}.

An \emph{occurrence} of $A$ in $E$ selects either $E$ itself, with
$A=E$, or the entire content of a specified matched bracket pair.
Occurrences compose by nesting their selected positions. A
\emph{deletion} replaces the contents of selected matched pairs by
$\varepsilon$, retaining the enclosing brackets; total deletion of
the entire expression to $\varepsilon$ is also allowed. Doing nothing
is the identity. Nested deletions have the same result as deleting
their outermost selected pairs, and deleting already empty contents
has no effect. These descriptions are identified, so there is one
deletion arrow for each source and obtainable target.

\phantomsection\label{not:bracket-category}
Let $\mathcal E_\Sigma$ have these expressions as objects. An arrow
$E\to F$ consists of a deletion $d:E\to T$ and a specified
occurrence $u:T\to F$. Equivalently, it chooses a position in $F$ whose
content is a deletion outcome of $E$. Composition restricts the next
deletion to the selected occurrence, and then composes the residual
occurrences. If the selected content is erased by that deletion, its
residual occurrence is the empty content at the surviving cut boundary.
Identities select the whole expression without deleting anything.
The following theorem also verifies the category laws.

\begin{theorem}[Representation of bracketed expressions]\label{prop:bracketed}
The category $\mathcal E_\Sigma$ satisfies the nine axioms, with
terminal object $\varepsilon$, epimorphisms the deletions, and strong
monomorphisms the specified occurrences. Its derived labels are
precisely $L_\Sigma$, with the arities above. Under their natural
identification, the representation of Theorem~\ref{thm:representation}
is the isomorphism
\[
 \tau:\mathcal E_\Sigma\longrightarrow\C_{L_\Sigma},
 \qquad
 \tau(\varepsilon)=1,\qquad
 \tau(E)=s(E)\bigl(\tau(E_1),\ldots,\tau(E_n)\bigr).
\]
Thus $\tau(E)$ is the expression tree of $E$. The points of $E$ are
its occurrences of empty bracket contents, together with the identity
when $E=\varepsilon$.
\end{theorem}
\begin{proof}
The unique outermost decomposition defines $\tau$ recursively. It is
bijective on objects: given a tree, recursively write each child as a
bracketed expression and insert it into the corresponding empty pair
of its root pattern. This reverses~\eqref{eq:bracket-decomposition}.
Each matched pair specifies a child position at some depth, so
occurrences correspond bijectively to specified rooted subtrees.
Deleting its contents replaces exactly that subtree by the empty leaf;
total deletion is the root cut. Consequently deletion outcomes
correspond exactly to tree prefixes.

The bijections on deletions and on occurrences give a bijection on
arrows by their pair descriptions. If a second deletion is restricted
to the selected occurrence, its deleted bracket contents are exactly
the cut subtrees occurring within that selected subtree. If the
second deletion erases an ancestor of the occurrence, both
constructions give total pruning to the same retained empty boundary.
Thus restriction agrees in all cases, so the composition formula in
Subsection~\ref{sec:model} agrees with the stated composition on
expressions. The identity pairs also correspond. The category laws follow,
and $\tau$ is an isomorphism. Theorem~\ref{thm:realization} supplies
the axioms, while Lemma~\ref{prop:modelclasses} identifies the
two arrow classes. Terminality and the description of points follow
as well.

Clearing all outermost bracket contents gives the root normalization
$E\to s(E)$ for $E\ne\varepsilon$. Its components have sources
$E_1,\ldots,E_n$. Thus the intrinsic root label and indexed components
used in Theorem~\ref{thm:representation} are exactly those in the
recursion for $\tau$, proving the final identification.
\end{proof}

\begin{example}[An expression and its matchings]\label{ex:bracketed-matchings}
Take distinct symbols \texttt{a}, \texttt{b}, \texttt{c}, \texttt{+},
and \texttt{x} in $\Sigma$. For
$E=\texttt{[[a]+[b]]x[c]}$, the component expressions are
\texttt{[a]+[b]} and \texttt{c}. Its representation has root label
\texttt{[]x[]}; the first child has label \texttt{[]+[]} and children
\texttt{a}, \texttt{b}, while the second child is \texttt{c}:
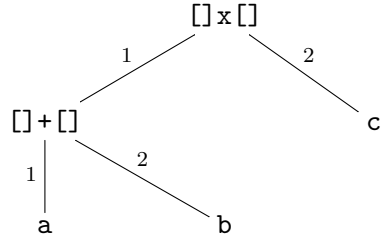
\[
 \xymatrix@C=3em@R=2.3em{
  & \texttt{[]x[]} \ar@{-}[dl]_{1} \ar@{-}[dr]^{2} & \\
  \texttt{[]+[]} \ar@{-}[d]_{1} \ar@{-}[dr]^{2} & & \texttt{c} \\
  \texttt{a} & \texttt{b} &
 }
\]
The edges display the numbered child places. If \texttt{+} and
\texttt{x} are read as addition and multiplication, this is the usual
expression tree of $(a+b)\mathbin{\texttt{x}}c$; these interpretations
impose no equations on the expressions.

Deleting \texttt{b} gives \texttt{[[a]+[]]x[c]}. In its tree the
vertex \texttt{b} is replaced by an empty leaf $1$, with its child
place retained. A prefix--suffix matching is
\[
 \xymatrix@C=2em{
  \texttt{[a]+[b]} \ar[r]^{d} &
  \texttt{[a]+[]} \ar[r]^{u} &
  \texttt{[[a]+[]]x[c]}
 }
\]
where $d$ deletes \texttt{b} and $u$ selects the first outermost
bracket content. Its overlap \texttt{[a]+[]} is a prefix of the source
and a suffix at the specified position of the target.
If $e$ next deletes \texttt{a} in that target, restriction gives
\[
 \xymatrix@C=4em@R=2.5em{
  \texttt{[a]+[]} \ar[r]^{u} \ar[d]_{d'} &
  \texttt{[[a]+[]]x[c]} \ar[d]^{e} \\
  \texttt{[]+[]} \ar[r]_{u'} &
  \texttt{[[]+[]]x[c]}
 }
\]
Hence $e(ud)=u'(d'd)$: the composite overlap is \texttt{[]+[]}.
Finally, \texttt{[a]+[a]} has two distinct occurrences of
\texttt{a}. They give different strong monomorphisms even though
the selected expressions agree, just as equal subtrees at different
positions give different tree morphisms.
\end{example}
\subsection{Pointed categories and words}\label{sec:words}
A category is \emph{pointed} if it has an object that is both initial
and terminal, called a \emph{zero object}. In the presence of
Axiom~\ref{ax:N}, this means precisely that $\one$ is initial.
There is then a unique point $\one\to X$ of every object, and the
constant morphism $X\to\one\to Y$ is the zero morphism $0_{X,Y}$.
We first describe the resulting categories concretely, and then
give their intrinsic characterization.

\begin{definition}[Word categories]\label{def:wordcategory}
Let $L$ be a set, called an \emph{alphabet}, and let $L^*$ be its set
of finite words, including the empty word $\varepsilon$.
For a word $u$, write $|u|$ for its length and
$\operatorname{pre}_r(u)$ and $\operatorname{suf}_r(u)$ for its
prefix and suffix of length $r$. Both are $\varepsilon$ when $r=0$.
The \emph{word category} $\mathcal W_L$ has objects $L^*$ and
\begin{equation}\label{eq:wordhom}
 \mathcal W_L(u,v)=
 \{f^r_{u,v}:0\le r\le\min(|u|,|v|),\quad
          \operatorname{pre}_r(u)=\operatorname{suf}_r(v)\}.
\end{equation}
The arrow $f^r_{u,v}$ matches the first $r$ letters of $u$ with
the last $r$ letters of $v$. Their common word is its
\emph{overlap}, so the superscript $r$ records its length. The identity
of $u$ is $f^{|u|}_{u,u}$, and composition is
\begin{equation}\label{eq:wordcomposition}
f^s_{v,w}f^r_{u,v}
   =f^{\max(0,r+s-|v|)}_{u,w}.
\end{equation}
\end{definition}

\begin{lemma}[Words as unary trees]\label{prop:wordmodel}
These data define a category. Giving every letter of $L$ arity one
identifies $\mathcal W_L$ with its tree category $\C_L$, by
\[
 \varepsilon\longleftrightarrow\one,
 \qquad
a_1\cdots a_n\longleftrightarrow
a_1\bigl(a_2(\cdots a_n(\one)\cdots)\bigr).
\]
In particular, $\mathcal W_L$ satisfies all nine axioms and is
pointed, with zero object $\varepsilon$ and
$0_{u,v}=f^0_{u,v}$.
\end{lemma}
\begin{proof}
Let every letter of $L$ have arity one. A nonterminal tree then
has a labeled root with a single child. Repeating this description
until reaching the empty leaf gives exactly the displayed chain.
By the cut classification in Theorem~\ref{thm:cuts}, a pruning of
this tree retains an initial sequence of its labels, hence a prefix
of the word. A whole subtree at a specified position is a suffix.
Consequently a normal factorization has the form
\[
 \xymatrix@C=4em{
u\ar[r]^{p_{u,t}}&t\ar[r]^{i_{t,v}}&v
 }
\]
where $t$ is a prefix of $u$ and a suffix of $v$, $p_{u,t}$ is the
unique pruning to that prefix, and $i_{t,v}$ is its suffix inclusion.
For $r=|t|$, the existence condition is exactly
$\operatorname{pre}_r(u)=\operatorname{suf}_r(v)$.
Uniqueness of both maps therefore gives~\eqref{eq:wordhom}.

Put $m=|v|$. The suffix selected by the first arrow starts at depth
$m-r$ in $v$. The second arrow retains its prefix of length $s$.
If $s\le m-r$, this pruning collapses the selected subtree to the
empty leaf. Otherwise it retains the portion between depths
$m-r$ and $s$, of length $r+s-m$. This portion is a prefix of $u$
and a suffix of the prefix of $v$ selected by the second arrow,
hence a suffix of $w$. The tree composition is therefore exactly
\eqref{eq:wordcomposition}. Identities also agree, so the category
laws follow from those of $\C_L$.

Theorem~\ref{thm:realization} supplies the nine axioms.
Formula~\eqref{eq:wordhom} gives exactly one arrow from or to
$\varepsilon$. Its factorization through $\varepsilon$ is the
arrow with empty overlap, as claimed.
\end{proof}

\begin{theorem}[Characterization of word categories]\label{thm:wordcharacterization}
For a small category $\C$, the following are equivalent:
\begin{enumerate}
\item $\C$ satisfies Axioms~\ref{ax:N}--\ref{ax:Ddetect} and is pointed;
\item $\C$ satisfies those axioms and every derived label has arity one;
\item $\C$ is isomorphic to $\mathcal W_L$ for some set $L$.
\end{enumerate}
These conditions are also equivalent, among categories satisfying
the nine axioms, to every operation of the recovered operad having
arity one. The alphabet can be recovered as the derived label set
$\Phi$, or equivalently as the set of objects of complexity one.
Under this representation, complexity is word length.
\end{theorem}
\begin{proof}
Suppose \textup{(a)} holds. Since $\one$ is initial, every object has exactly
one point. A label $\phi$ of arity $n$ has exactly $n$ points:
its component sources are terminal, and every point selects one
of these components. These are precisely the input places of
$\phi$, by the definition of the derived signature.
Consequently $n=1$, proving \textup{(b)}.

Suppose now that \textup{(b)} holds.
Theorem~\ref{thm:representation} identifies $\C$ with the
expression tree category of its derived unary signature, and
Lemma~\ref{prop:wordmodel} identifies this category with
$\mathcal W_\Phi$. This proves \textup{(c)}. Conversely, the
same lemma shows that each $\mathcal W_L$ is pointed and
satisfies the nine axioms, so \textup{(c)} implies \textup{(a)}.

A nonempty word of length $n$ has precisely $n$ proper suffix occurrences,
including the empty suffix. Since these are exactly its nonidentity
strong monomorphisms, its complexity is $n$; the empty word has
complexity zero. The words of complexity
one are therefore exactly the single letters. This recovers the
alphabet intrinsically and proves the remaining claims. By
Theorem~\ref{thm:operad-recovery}, every operation of the recovered
operad has arity one: each word has exactly one point, including
the empty word that represents the identity operation. Conversely,
if every operation has arity one, so does every generating label,
and condition \textup{(b)} holds.
\end{proof}

\begin{remark}[Finite pointed sets]
The category of finite pointed sets illustrates the role of the
axiom on intersecting strong subobjects. Its inclusions in the square
\[
\xymatrix@C=3.5em@R=2.5em{
\{*\}\ar[r]\ar[d]&\{*,a\}\ar[d]\\
\{*,b\}\ar[r]&\{*,a,b\}
}
\]
are strong monomorphisms, since they are injective functions preserving
the distinguished element. Neither of the two diagonals required by
Axiom~\ref{ax:O} exists: its composite with the indicated inclusion
would have to retain an element absent from that inclusion's image.
Thus Axiom~\ref{ax:O} distinguishes these pointed sets from the
word categories characterized by the theorem.
\end{remark}

\begin{lemma}[Arrow classes and filling]\label{prop:wordarrows}
In $\mathcal W_L$, an arrow $f^r_{u,v}$ with overlap of length
$r$ is epic precisely when
$r=|v|$, and monic precisely when $r=|u|$.
Every monomorphism is strong. Thus epimorphisms retain prefixes,
and monomorphisms include suffixes.
Filling the unique point of a word $b$ by a word $a$ gives the
concatenated word $ba$.
\end{lemma}
\begin{proof}
The epimorphism and strong-monomorphism statements follow from
Lemma~\ref{prop:modelclasses} and the identification of
prefix prunings and suffix occurrences in
Lemma~\ref{prop:wordmodel}.
To prove the assertion about all monomorphisms, suppose that $r<|u|$.
Let $z$ be the suffix of $u$ of positive length $|u|-r$, and let
$h:z\to u$ be its occurrence. Formula~\eqref{eq:wordcomposition}
gives $f^r_{u,v}h=0_{z,v}=f^r_{u,v}0_{z,u}$, although
$h\ne0_{z,u}$. Hence $f^r_{u,v}$ is not monic.

For filling, let $i:a\to ba$ be the suffix occurrence and
$p:ba\to b$ the prefix pruning. They form the square
\[
\xymatrix@C=4em@R=2.5em{
a \ar[r]^{i} \ar[d]_{q_a} & ba \ar[d]^{p} \\
 \varepsilon \ar[r]_{j_b} & b
}
\]
where $q_a:a\to\varepsilon$ is the total pruning and
$j_b:\varepsilon\to b$ selects the empty suffix of $b$.
The fibre of $p$ over this point
is exactly its removed suffix $a$, so the square is a pullback.
Let $(c,d,u)$ be another pullback factorization of $j_bq_a$ with
$u:a\to c$ strong monic and $d:c\to b$ epic, as in
Axiom~\ref{ax:U}. By the arrow classes already proved, $d$ retains
the prefix $b$ of $c$, and its fibre at $j_b$ is the suffix following
that prefix. The given pullback identifies this suffix with $a$.
Hence $c=ba$, $d=p$, and $u=i$.
A comparison endomorphism $h:ba\to ba$ must satisfy $hi=i$ and
$ph=p$.
If $a\ne\varepsilon$, the first equality and
\eqref{eq:wordcomposition} force its overlap to have length $|ba|$,
so $h=\id_{ba}$. If $a=\varepsilon$, then $p=\id_b$, and the
second equality gives the same conclusion. Thus this is the terminal filling.
The order $ba$ records reading the surrounding chain $b$ first,
then the chain $a$ inserted at its empty leaf.
\end{proof}

\begin{example}[One letter and two-letter overlaps]\label{ex:wordexamples}
For a one-letter alphabet, there is one word of each length
$n\ge0$. Every permitted length is an overlap, so
\[
 |\mathcal W_L(n,m)|=\min(n,m)+1.
\]
For an alphabet with distinct letters $a,b$, the arrows
$ab\to ba$ have overlap $\varepsilon$ or $a$, while those
$ba\to ab$ have overlap $\varepsilon$ or $b$.
The composites of the two nonzero arrows are zero, because the
composition formula gives $\max(0,1+1-2)=0$.
An empty alphabet gives the one-object category.
\end{example}

\begin{remark}[Comparison with colored chains]
Our word categories occur within the construction of Szczesny
\cite[Section~3.2.1 and Example~7]{SzczesnyPreLie}, which uses
colored partially ordered sets and morphisms identifying suitable
parts of them. The agreement includes both objects and morphisms,
as can be seen by restricting that construction to chains.
Associate to $a_1\cdots a_n$ the chain with $n$ elements, colored
in this order from greatest to least, and associate the empty word
with the empty poset. A morphism in that construction identifies
the complement of an order ideal in the source with an order ideal
in the target, preserving order and colors. Here an order ideal
means a subset containing every element below each of its elements.
For chains, the two selected parts are exactly a prefix of the
source word and a suffix of the target word. Their order-preserving
identification is unique when the colors agree. Composition retains
the intersection of the two selected parts of the middle chain,
whose length is $\max(0,r+s-|v|)$. Consequently the full subcategory
on colored chains and the empty poset is equivalent to
$\mathcal W_L$.

Szczesny's colored-ladder example uses a finite set of colors; the
same construction applies to an arbitrary alphabet. Our word
$a_1\cdots a_n$ corresponds to his ladder $L(a_n,\ldots,a_1)$,
since that notation reads the colors from leaf to root.
\end{remark}

\section{A noetherian form with exact join decomposition}\label{sec:cosubquotients}
We now turn to the algebraic behavior of expression tree categories.
A matching has a quotient, given by its pruning, and a subobject,
given by the subtree that it selects. We want to know how closely
these two constructions behave like quotients and subgroups in the
category of groups. For groups, direct and inverse images of
subgroups describe the kernel and image of a homomorphism and
provide the basis for the isomorphism theorems. Noetherian forms
give a way of expressing these properties for other categories.
The data that take the place of subgroups are called clusters.

In our construction, a cluster over a tree consists of a pruning
and a choice of an empty leaf in the resulting tree; we also allow
the choice to be empty. This marking remembers which part of the
original tree has been collapsed to the chosen leaf. We shall show
that these data admit direct and inverse images with the properties
needed for a noetherian form. The construction uses the empty tree
as an additional object, and it simplifies for word categories,
where the empty word is already a zero object.

The general results of Janelidze and van Niekerk
\cite[Theorems 159 and 163]{JanelidzeVanNiekerk} give conditions
under which subobjects and quotients determine a noetherian form.
We apply these results to trees and identify the resulting form
with the marked prunings just described. We also prove a converse,
which allows expression tree categories with an empty tree adjoined
to be recognized among categories carrying such a form.

\subsection{The empty tree and the clusters}\label{sec:tree-form-construction}
Two subtrees in different branches of a tree have no position in
common. To represent their intersection by an object, we introduce
the tree with no root. The same object will also allow us to
describe an empty marking.

\phantomsection\label{not:augmented-tree-category}
Let $\mathcal T$ be an operadic expression tree category, as in
Theorem~\ref{thm:representation}. Write $\widehat{\mathcal T}$
for the category obtained by
adjoining a new object $0$, with a unique morphism $0\to X$ for
every object $X$ and no morphism from an object of $\mathcal T$
to $0$.
Thus $0$ is strict initial: every morphism with codomain $0$ is
an isomorphism. It represents the tree with no root. The object
$1$ continues to represent the tree consisting of one empty position.

\phantomsection\label{not:augmented-tree-classes}
Let $\mathcal E$ consist of the original prunings together with
$\id_0$, and let $\mathcal M$ consist of the original occurrences
together with all morphisms $0\to X$. An occurrence identifies
its domain with the subtree that it selects, so we shall also
call it a subtree inclusion. We refer to an
$\mathcal M$-subobject as a subtree, allowing the empty subtree.
Both classes are closed under composition.

\begin{lemma}\label{lem:augmented-factorization}
The pair $(\mathcal E,\mathcal M)$ is a proper factorization system
on $\widehat{\mathcal T}$.
\end{lemma}
\begin{proof}
Every matching in $\mathcal T$ retains its pruning--occurrence
factorization. The new morphism $0\to X$ factors as $\id_0$
followed by itself. Prunings remain epic, since adjoining $0$
introduces no new morphisms with a nonempty domain. Occurrences
remain monic: the only new possible test morphisms come from
$0$, and any two such morphisms are equal. Finally, $0\to X$
is monic because there is at most one morphism from any object
to $0$.

Consider a commutative square with its left map in $\mathcal E$
and its right map in $\mathcal M$. If the left map is $\id_0$,
the top map gives the unique diagonal.
Otherwise its domain belongs to $\mathcal T$. A right map with
domain $0$ would then require a map from an object of $\mathcal T$
to $0$, contradicting the construction. All four objects and
maps belong to the original category, where the required diagonal
exists uniquely. We have proved the unique diagonal property,
as well as factorization and the epic and monic conditions.
These are precisely the conditions for a proper factorization system.
\end{proof}

\begin{remark}
The distinguished classes in Lemma~\ref{lem:augmented-factorization}
need not consist of all epimorphisms and all strong monomorphisms
of the enlarged category. For example, when the original category
is pointed, the new map $0\to1$ is epic and is not in
$\mathcal E$. The construction below uses the specified classes.
\end{remark}

\phantomsection\label{not:tree-subobjects-quotients}
For an object $X$, let $\mathsf S(X)$ be the partially ordered
set of subobjects represented by morphisms in $\mathcal M$.
Thus a member of $\mathsf S(X)$ is a specified subtree of
$X$, which may be empty. The inequality $[m]\le[n]$ means
that $m$ factors through $n$. Similarly, let $\mathsf Q(X)$
be the partially ordered set of quotients represented by morphisms
in $\mathcal E$. The inequality $[q]\le[r]$ means that $r$
factors through $q$, so that further pruning increases a quotient
in this order. Subobjects and quotients are taken up to
isomorphism fixing $X$.

\begin{definition}\label{def:tree-clusters}
A \emph{cluster over $X$} is a cospan
\[
 \xymatrix@C=3.5em{
 X\ar[r]^q&Q&U\ar[l]_u
 }
\]
where $q\in\mathcal E$, $u\in\mathcal M$, and $U$ is $0$ or
$1$. Cospans are identified by isomorphisms that fix $X$.
For clusters $[q,u]$ over $X$ and $[r,v]$ over $Y$, and a
morphism $f:X\to Y$, write $[q,u]\le_f[r,v]$ when there are
morphisms $h$ and $k$ making the following diagram commute
\[
 \xymatrix@C=3.5em@R=2.4em{
 X\ar[r]^q\ar[d]_f&Q\ar[d]^h&U\ar[l]_u\ar[d]^k\\
 Y\ar[r]_r&R&V\ar[l]^v
 }
\]
Denote this collection of clusters and comparisons by $\mathsf G$.
\end{definition}

Thus a cluster is a pruned tree with either an empty marking or
a marked empty leaf. It remembers both how the quotient was
obtained and which leaf was selected. In particular, selecting
different leaves of the same quotient gives different clusters.
Whenever a comparison exists, its maps are unique. Indeed, $h$
is determined by the epic map $q$, and $k$ is determined by
the monic map $v$. Comparisons compose by pasting their diagrams.
Comparisons over $\id_X$ define a partial order,
since mutual comparisons give inverse isomorphisms of cospans.
Consequently the clusters and comparisons form a category with
a faithful functor to $\widehat{\mathcal T}$, obtained by
forgetting the cluster. A functor of this kind is called a
\emph{form}. Its partially ordered set of clusters over $X$
is its \emph{fibre}, denoted $\mathsf G(X)$.

\subsection{Subtrees, prunings, and image operations}\label{sec:tree-form-images}
To compare clusters over different objects, we need to transport
subobjects and quotients along matchings. Pullbacks provide inverse
images of subobjects, while pushouts provide direct images of
quotients. Both constructions have a direct interpretation in
terms of the trees, which we describe before applying the general
noetherian-form theorem.

\phantomsection\label{not:labeled-vertices}
Write $V_+(X)$ for the labeled vertices of $X$, excluding the
empty leaves. A subset of $V_+(X)$ is \emph{descendant-closed}
if, together with any vertex, it contains every labeled descendant
of that vertex.

\begin{lemma}\label{lem:tree-lattice-calculus}
The following properties hold in $\widehat{\mathcal T}$.
\begin{enumerate}
\item The posets $\mathsf S(X)$ and $\mathsf Q(X)$ are finite
bounded lattices. Quotients of a nonempty tree $X$ correspond to
descendant-closed subsets of $V_+(X)$, ordered by inclusion.
\item Pullbacks of maps in $\mathcal M$ along arbitrary
morphisms exist. Pulling a map in $\mathcal M$ back along a
map in $\mathcal E$ gives a map in $\mathcal E$ on the other
side of the square.
\item Pushouts of maps in $\mathcal E$ along arbitrary
morphisms exist.
\end{enumerate}
\end{lemma}
\begin{proof}
Two subtrees at comparable positions meet in the smaller subtree.
Subtrees at incomparable positions meet in $0$. Their join is the
subtree rooted at their least common ancestor. The empty subtree
is the bottom element and the whole tree is the top element.
These descriptions include all subobjects. There are finitely many
positions, so $\mathsf S(X)$ is finite.

A pruning is determined by its set of erased labeled vertices.
This set is descendant-closed. Conversely, a descendant-closed
set $D$ determines a pruning: replace each of its vertices nearest
the root by an empty leaf, thereby also erasing all descendants
of those vertices. Further pruning enlarges $D$. Hence
$\mathsf Q(X)$ is the lattice of these sets, with intersection
as meet and union as join. Its bottom is the identity quotient;
its top erases all labeled vertices and has codomain $1$.
Both lattices at $0$ have one element.

The intersection construction just given is the pullback of two
subtree inclusions: a matching that factors through both must
have its image in their intersection. In the disjoint case, a
nonempty image cannot lie in both, so the domain of such a
matching must be $0$.

Consider next a pruning $e:X\to Y$ and a subtree $B\to Y$.
The root position of $B$ is a position retained by the pruning,
possibly as an empty leaf created by a cut. Take the entire
subtree $A$ of $X$ at that position. The restriction $A\to B$
is a pruning, and gives a commutative square
\[
\xymatrix@C=3.5em@R=2.4em{
A\ar[r]\ar[d]&X\ar[d]^e\\
B\ar[r]&Y.
}
\]
A matching into $X$ whose composite with $e$
factors through $B$ has its image in $A$: pruning changes the
subtree at a cut to a leaf, and changes no ancestor of that cut.
It therefore factors uniquely through $A\to X$. This proves
the pullback property. The inverse image of the empty subtree
is empty. To pull back along a general matching, factor it as
a pruning followed by an inclusion and paste the two pullback
squares. This proves (b).

The pushout of two prunings of $X$ erases the union of their
erased sets. To check its universal property, observe that a
matching out of $X$ factors through the pruning with erased set
$D$ exactly when its own erased set contains $D$. Indeed, in
that case the residual pruning and the same final subtree
inclusion give the factorization; necessity follows from the
composition rule. Thus a common composite factors uniquely
through the union pruning. Uniqueness of the induced maps from
the two quotients follows because pruning maps are epic.

The pushout of a pruning along a subtree inclusion performs that
pruning at the specified position in the ambient tree. This is
the bicartesian square already used in the tree construction.
To push out along a general matching, first factor that matching
as a pruning followed by an inclusion, and then paste the two
pushouts. For a pushout involving $0$, the map in $\mathcal E$
must have domain $0$ and hence must be $\id_0$. Its pushout
is an identity, so this case is covered as well. This proves (c).
\end{proof}

\phantomsection\label{not:subobject-quotient-images}
Let $f:X\to Y$ be a morphism. For a subtree represented by
$m:A\to X$, the inclusion part of $fm$ represents its
\emph{direct image}, denoted $f_*^{\mathsf S}[m]$. For a
subtree of $Y$, its pullback along $f$ represents its
\emph{inverse image}, denoted $f_{\mathsf S}^*$ applied to
that subtree. Thus the superscript or subscript $\mathsf S$
indicates that the image operation acts on subobjects.

We use $\mathsf Q$ in the corresponding notation for quotients.
The direct image $f_*^{\mathsf Q}[q]$ of a quotient
$q:X\to Q$ is obtained by pushing $q$ out along $f$.
The inverse image $f_{\mathsf Q}^*[r]$ of a quotient
$r:Y\to R$ is the quotient part of the composite $rf$.
These two constructions have the following diagrams:
\[
 \xymatrix@C=3em@R=2.3em{
 X\ar[r]^q\ar[d]_f&Q\ar[d]\\
 Y\ar[r]_{q'}&P
 }
 \qquad
 \xymatrix@C=3em@R=2.3em{
 X\ar[r]^e\ar[d]_f&I\ar[d]^m\\
 Y\ar[r]_r&R
 }
\]
The first is a pushout and gives $f_*^{\mathsf Q}[q]=[q']$;
the second has $rf=me$ and gives $f_{\mathsf Q}^*[r]=[e]$.
Each direct image is left adjoint to the corresponding inverse
image. For example, a subobject $A$ of $X$ satisfies
$f_*^{\mathsf S}A\le B$ exactly when the composite of its
inclusion with $f$ factors through $B\to Y$. By the pullback
property, this is equivalent to $A\le f_{\mathsf S}^*B$.
The analogous assertion for quotients follows from the pushout
property. The uniqueness in these universal constructions also
makes direct and inverse images functorial. The subobject and
quotient forms, with these operations and the given factorization
system, constitute
what \cite[Definition 95 and Theorem 100]{JanelidzeVanNiekerk} calls an
\emph{orean factorization}.

\phantomsection\label{not:matching-factor-parts}
Write $I_f$ for the subobject of $Y$ and $K_f$ for the quotient
of $X$ given by the two parts of the factorization of $f$.
Thus $I_f\in\mathsf S(Y)$ and $K_f\in\mathsf Q(X)$.
For a matching between nonempty trees, $K_f$ records what is
pruned from its source, while $I_f$ records the entire subtree
selected in its target.

\phantomsection\label{not:subtree-collapse}
Let
$\alpha_X:\mathsf S(X)\to\mathsf Q(X)$ send a subtree to
the quotient that collapses it to an empty leaf; set
$\alpha_X(0)=\bot$. For a subtree $A\ne0$, this is the
pushout displayed below, which is also a pullback
\[
 \xymatrix@C=3.5em@R=2.3em{
 A\ar[r]^{m_A}\ar[d]&X\ar[d]^{c_A}\\
 1\ar[r]&Q_A
 }
\]
The map $\alpha_X$ records how much of $X$ is removed when
we collapse a selected subtree. Equivalently, it is the quotient
direct image along $m_A$ of the largest quotient of $A$. This
largest quotient is $A\to1$ for a nonempty tree and $\id_0$ for
the new empty tree. This describes $\alpha$ entirely in terms
of the factorization system and its image operations:
\[
 \alpha_X(A)=(m_A)_*^{\mathsf Q}
                 (\top_{\mathsf Q(A)}).
\]

\phantomsection\label{not:subtree-saturation}
For a quotient $R$ represented by $q:X\to Q$, and a subtree
$A$ of $X$, put
\[
 \sigma_R(A)=q_{\mathsf S}^*q_*^{\mathsf S}(A).
\]
This is the least subtree containing $A$ that is an inverse
image under $q$. We call $A$ \emph{saturated by $R$} if
$\sigma_R(A)=A$. In tree terms, saturation enlarges $A$ to
the entire subtree at a cut above its root, if there is such a
cut. Otherwise it leaves $A$ unchanged. The empty subtree
remains empty.

\begin{lemma}\label{lem:tree-form-identities}
For every morphism $f:X\to Y$, quotient $D$ of $X$, quotient
$E$ of $Y$, and subtree $B$ of $Y$, one has
\[
 \begin{split}
 f_{\mathsf Q}^*f_*^{\mathsf Q}D&=D\vee K_f,\\
 f_*^{\mathsf Q}f_{\mathsf Q}^*E&=E\wedge\alpha_Y(I_f),\\
 f_*^{\mathsf S}f_{\mathsf S}^*B&=B\wedge I_f.
 \end{split}
\]
Moreover, $\sigma_{\alpha_X(A)}(A)=A$ for every subtree $A$,
and $\alpha_X$ preserves bottom elements and binary meets.
\end{lemma}
\begin{proof}
Recall that $V_+(X)$ is the set of labeled vertices, and that
quotients are represented by descendant-closed sets of erased
vertices. Let $D_f\subseteq V_+(X)$ be the vertices erased
by the pruning part of $f$.
The surviving labeled vertices are identified with those of its
matched subtree by a bijection
\[
 \phi_f:V_+(X)\setminus D_f\longrightarrow V_+(I_f)
 \subseteq V_+(Y).
\]
In erased-set coordinates, the two quotient operations are
\[
 f_*^{\mathsf Q}D=\phi_f(D\setminus D_f),\qquad
 f_{\mathsf Q}^*E=D_f\cup\phi_f^{-1}(E).
\]
Here the inverse image on the right uses only the vertices of
$E$ lying in $I_f$. For direct image, erased vertices of $D$
that were already erased by $f$ impose no further cut. Each
remaining erased vertex specifies the corresponding cut inside
$I_f$. Performing these cuts in the matched subtree is exactly
the pushout construction described above. For inverse
image, the composite first erases $D_f$ and then those surviving
vertices whose images are erased by $E$. These observations
prove the two formulas. Applying them successively gives
$D\cup D_f$ and $E\cap V_+(I_f)$, respectively, as required.
For a map from $0$, the quotient source lattice is a singleton
and its direct image is the identity quotient; the same
identities follow immediately.

Pullback stability in Lemma~\ref{lem:tree-lattice-calculus}
shows that direct image after inverse image along a pruning is
the identity. Along an inclusion it is intersection with the
included subtree. Factor $f$ to obtain the third displayed
identity.

The collapse square has fibre $A$ over its marked empty leaf,
so $A$ is saturated by its own collapse. The erased set of
$\alpha_X(A)$ is precisely the set $V_+(A)$ of its labeled
vertices. The intersection description of subtrees gives
\[
 V_+(A\wedge B)=V_+(A)\cap V_+(B),\qquad V_+(0)=\varnothing.
\]
These equalities prove preservation of binary meets and bottom.
\end{proof}

\subsection{The noetherian property}\label{sec:noetherian-property}
We have described how subtrees and prunings behave under matchings.
We now state the properties required of the clusters that combine
these two kinds of data. For a form $\mathsf F$, write
$\mathsf F(X)$ for its fibre of clusters over $X$.
\phantomsection\label{not:cluster-images}
The form is \emph{orean} when each fibre is a bounded lattice
and every morphism $f:X\to Y$ has functorial direct and inverse
images characterized by
\[
 f_*L\le N\quad\Longleftrightarrow\quad L\le_f N
 \quad\Longleftrightarrow\quad L\le f^*N.
\]
Here $L$ is a cluster over $X$ and $N$ is a cluster over
$Y$, so $f_*$ takes clusters over $X$ to clusters over $Y$,
while $f^*$ acts in the opposite direction.
The direct image is thus the least target cluster receiving a
comparison from $L$ over $f$. The inverse image is the greatest
source cluster admitting a comparison to $N$ over $f$.
\phantomsection\label{not:form-kernel-image}
Define the \emph{kernel} and \emph{image} of $f$ by
\[
 \operatorname{Ker}f=f^*(\bot),\qquad
 \operatorname{Im}f=f_*(\top).
\]
The kernel is the largest source cluster whose direct image is
the bottom cluster. The image is the least target cluster whose
inverse image is the whole source cluster. For groups, these are
the usual kernel and image of a homomorphism.
A cluster over $X$ is \emph{normal} if it is the kernel of
a morphism with domain $X$, and \emph{conormal} if it is the
image of a morphism with codomain $X$. These names recall the
subgroup example: kernels are normal subgroups, and every
subgroup is the image of its inclusion.

A \emph{universal quotient} of a cluster $L$ over $X$ is a
morphism $q:X\to Q$ through which every $f:X\to Y$ with
$L\le\operatorname{Ker}f$ factors uniquely, and which itself
satisfies $L\le\operatorname{Ker}q$. Dually, a
\emph{universal embedding} of a cluster $N$ over $Y$ is a
morphism $m:A\to Y$ through which every $f:X\to Y$ with
$\operatorname{Im}f\le N$ factors uniquely, and which itself
satisfies $\operatorname{Im}m\le N$.
In the group example, these maps are the quotient homomorphism
by a normal subgroup and the inclusion of a subgroup, respectively.
Thus the universal maps express, within the category, the
subobject or quotient information carried by a cluster.
An orean form is \emph{noetherian} when the following conditions
hold \cite[Definition 48]{JanelidzeVanNiekerk}.
\begin{enumerate}
\item For every $f:X\to Y$ and clusters $L$ over $X$ and
$N$ over $Y$, one has
\[
 f^*f_*L=L\vee\operatorname{Ker}f,\qquad
 f_*f^*N=N\wedge\operatorname{Im}f.
\]
\item Every morphism factors as a universal quotient of its
kernel followed by a universal embedding of its image.
\item Binary joins of normal clusters are normal, and binary
meets of conormal clusters are conormal.
\end{enumerate}
The first condition describes what is retained when a cluster is
transported in one direction and then back again. On the source,
one must add the kernel; on the target, one must intersect with
the image. The second condition realizes these two parts of a
morphism by universal maps. These properties are the basis for
the isomorphism theorems in this setting, just as the corresponding
properties of subgroups are the basis for the isomorphism
theorems for groups.

Kernels and images give two ways of obtaining clusters from
morphisms. The additional property that we shall establish says
that an arbitrary cluster is determined by the largest cluster
of each of these two kinds that it contains. It also requires
this description to be compatible with direct and inverse images.
\phantomsection\label{not:normal-conormal-parts}
Precisely, the form has \emph{exact join decomposition} if every cluster
$L$ has a largest normal cluster $n(L)\le L$ and a largest
conormal cluster $c(L)\le L$, with
\[
 L=c(L)\vee n(L),\qquad
 f^*n(N)=n(f^*N),\qquad
 f_*n(\top)=n(\operatorname{Im}f).
\]
The first equality recovers a cluster from its two parts. The
second says that taking the largest normal part commutes with
inverse image. The third describes the normal part of an image
in terms of the largest normal cluster of the source. These are
the additional compatibility properties required by exact join
decomposition; the formulation is equivalent to the one in
\cite[Theorem 140]{JanelidzeVanNiekerk}.

Semi-abelian categories provide another axiomatic setting for
homomorphism theorems, with the category of groups as a principal
example. The relation with noetherian forms becomes especially
close when a zero object, finite products and finite coproducts
are present. The following theorem is the dual formulation of
\cite[Theorem 177, equivalence of \textup{(ii)} and \textup{(vi)}]{JanelidzeVanNiekerk}.

\begin{theorem}[Janelidze--van Niekerk]\label{lem:ejd-semiabelian-opposite}
Let $\mathcal C$ be a pointed category with finite products and
finite coproducts. Then $\mathcal C$ admits a noetherian form
with exact join decomposition if and only if its opposite category
$\mathcal C^{\mathrm{op}}$ is semi-abelian.
\end{theorem}

For expression tree categories, we shall use the following
consequence of the characterization in
\cite[Theorems 159 and 163]{JanelidzeVanNiekerk}. For a proper
factorization system with the lattice and image operations above,
let $\alpha_X(A)$ be the direct image along a representative
$m_A:A\to X$ of the largest quotient of $A$.
Recall that $\sigma_R(A)$ first takes the image of $A$ under
the quotient $R$ and then takes its inverse image.
Suppose that the quotient image identities of
Lemma~\ref{lem:tree-form-identities} hold. Suppose also that
\[
 e_*^{\mathsf S}e_{\mathsf S}^*A=A
 \quad(e\in\mathcal E),\qquad
 \sigma_{\alpha(A)}(A)=A,
\]
and that $\alpha$ preserves binary meets and bottom. Then the
pairs
\begin{equation}\label{eq:tree-pairs}
 (A,R)\in\mathsf S(X)\times\mathsf Q(X),\qquad
 \sigma_R(A)=A,\quad\alpha_X(A)\le R
\end{equation}
form a noetherian form with exact join decomposition. Their order
is componentwise; comparison over $f$ means
$f_*^{\mathsf S}A\le B$ and $f_*^{\mathsf Q}R\le S$.
For trees, the condition $\alpha_X(A)\le R$ says that the
pruning $R$ erases all labeled vertices of the subtree $A$.
If $A$ is nonempty, its image is therefore an empty leaf.
The condition $\sigma_R(A)=A$ says that $A$ is the entire
subtree collapsed to that leaf. Thus these pairs describe the
same information that we recorded by marking a leaf of a pruned
tree. The possibility $A=0$ corresponds to the empty marking.

For completeness, the notation used in the cited characterization
can be matched with ours as follows. Its expression $A\mathsf wR$
takes the image of the subobject $A$ in the quotient $R$ and
then takes its inverse image. This is exactly $\sigma_R(A)$.
The map from subobjects to quotients in that characterization is
our map $\alpha$, which collapses a subtree. With these
identifications, the quotient conditions of Theorem 163 are
the two quotient image identities above. Its condition involving
the subobject image operations reduces to
$\sigma_{\alpha(A)}(A)=A$, and its remaining condition is
preservation of binary meets and bottom by $\alpha$.

\begin{theorem}\label{thm:tree-cosub}
The form $\mathsf G$ of Definition~\ref{def:tree-clusters} on
$\widehat{\mathcal T}$ is noetherian and has exact join
decomposition. Its universal quotients are the maps of
$\mathcal E$, and its universal embeddings are the maps of
$\mathcal M$.
\end{theorem}
\begin{proof}
Lemma~\ref{lem:tree-lattice-calculus} supplies the bounded
subobject and quotient lattices and the image operations from
pullbacks, pushouts and factorizations. They give the two forms required by
\cite[Theorem 100]{JanelidzeVanNiekerk}. In these two forms,
every subobject is
the image of its inclusion, and every quotient is the kernel
of its quotient map. Their universal embeddings and universal
quotients are precisely the maps of $\mathcal M$ and
$\mathcal E$, respectively: the defining image or kernel
inequality says exactly that the given morphism factors through
that inclusion or quotient, and uniqueness follows from
monicity or epicity. The construction in Theorem 159 retains
these two classes of universal maps. Its additional image
identities were proved in
Lemma~\ref{lem:tree-form-identities}. In particular,
$I_e=Y$ for $e:X\to Y$ in $\mathcal E$, so the third
identity there gives $e_*^{\mathsf S}e_{\mathsf S}^*A=A$.
For $R=[q:X\to Q]$, strictness of $0$ also gives
\[
 \alpha_X\bigl(q_{\mathsf S}^*\bot_{\mathsf S(Q)}\bigr)
 =\alpha_X(0)=\bot_{\mathsf Q(X)}\le R.
\]
In Theorem 159, the inverse image of the bottom subobject under
$q$ is denoted by $\beta(R)$. The displayed calculation verifies
the compatibility $\alpha\beta(R)\le R$ required there.
It remains to identify the pairs in \eqref{eq:tree-pairs} with the cospans used
to define $\mathsf G$, including their comparisons.

Given $[q,u]$, let $A=q_{\mathsf S}^*[u]$ and $R=[q]$.
The pullback square has the form
\[
 \xymatrix@C=3.5em@R=2.3em{
 A\ar[r]^{m_A}\ar[d]_p&X\ar[d]^q\\
 U\ar[r]_u&Q
 }
\]
The left map belongs to $\mathcal E$. If $U=0$, then
$A=0$. If $U=1$, the pruning $q$ collapses the entire
subtree $A$ to its marked leaf. Thus $\alpha_X(A)\le R$.
In both cases pullback stability gives
$q_*^{\mathsf S}A=[u]$, whence $\sigma_R(A)=A$.

Conversely, suppose $(A,R)$ satisfies \eqref{eq:tree-pairs}, and choose a
representative $q:X\to Q$ of $R$. If $A=0$, mark the
empty subtree of $Q$. Otherwise $\alpha_X(A)\le R$ says
that $q$ erases every labeled vertex of $A$. Its image in
$Q$ is therefore an empty leaf $u:1\to Q$. The saturation
condition says that the inverse image of this leaf is exactly
$A$. These two constructions are inverse, so they identify
the clusters.

Finally, consider clusters $[q,u]$ over $X$ and $[r,v]$ over
$Y$, with corresponding pairs $(A,R)$ and $(B,S)$, and let
$f:X\to Y$ be a morphism. A cospan comparison over $f$ gives
a quotient comparison
$f_*^{\mathsf Q}R\le S$. Its marked square, pulled back to
$X$ and $Y$, gives $f_*^{\mathsf S}A\le B$. Conversely,
the quotient comparison determines $h:Q\to R'$ with
$hq=rf$, where $R'$ is the codomain of $r$. The subobject
comparison gives a morphism $t:A\to V$ such that $rfm_A=vt$.
Since $qm_A=up$, the following square commutes:
\[
\xymatrix@C=3.5em@R=2.4em{
A\ar[r]^t\ar[d]_p&V\ar[d]^v\\
U\ar[r]_{hu}\ar@{-->}[ur]^k&R'.
}
\]
The map $p$ belongs to $\mathcal E$ and the marking $v$
belongs to $\mathcal M$, so the unique diagonal $k$ exists.
It is exactly the remaining map needed for a comparison of
the marked prunings. Hence the pair form
of the characterization is the cospan form $\mathsf G$.
The cited theorem proves its noetherian property and exact join
decomposition, and identifies the universal quotients and
embeddings with the two distinguished classes.
\end{proof}

The images in this form can be computed directly on marked
prunings. Let $f:X\to Y$ be a morphism. To take the direct
image of a cluster $[q,u]$ over $X$, push $q$ out along $f$
and factor the composite $hu$ of the induced map with the
marking. To take the inverse image of a cluster $[r,v]$ over
$Y$, factor $rf$ and pull its marking $v$ back along the
inclusion part. The diagrams are
\[
 \xymatrix@C=3em@R=2.3em{
 X\ar[r]^q\ar[d]_f&Q\ar[d]^h&U\ar[l]_u\ar[d]^e\\
 Y\ar[r]_{q'}&P&W\ar[l]^w
 }
 \qquad
 \xymatrix@C=3em@R=2.3em{
 X\ar[r]^{e'}\ar[d]_f&I\ar[d]^m&W'\ar[l]_{w'}\ar[d]\\
 Y\ar[r]_r&R&V\ar[l]^v
 }
\]
They give $f_*[q,u]=[q',w]$ and $f^*[r,v]=[e',w']$.
In the first diagram, a map out of $0$ or $1$ is an inclusion,
so $W$ is again $0$ or $1$. In the second, the intersection
of a subtree with an empty leaf is that leaf or the empty
subtree, so $W'$ is also permitted. The universal properties
of the pushout and pullback, together with the diagonal
property of the factorization system, give respectively the
least and greatest comparisons defining these image operations.

\begin{lemma}[The two parts of a cluster]\label{prop:cosub-join-parts}
Recall that a cluster is represented in \eqref{eq:tree-pairs}
by its quotient $R$ and the inverse-image subtree $A$ of its
marking. In these coordinates, the normal and conormal clusters
are, respectively,
\[
 \mathfrak n_X(R)=(0,R),\qquad
 \mathfrak c_X(A)=(A,\alpha_X(A)).
\]
For a cluster $(A,R)$, its largest normal and conormal parts
are $\mathfrak n_X(R)$ and $\mathfrak c_X(A)$, and
\[
 (A,R)=\mathfrak c_X(A)\vee\mathfrak n_X(R).
\]
For $f:X\to Y$, its kernel is $(0,K_f)$ and its image is
$(I_f,\alpha_Y(I_f))$.
\end{lemma}
\begin{proof}
Let $f:X\to Y$. The bottom cluster over $Y$ is
$[\id_Y,0\to Y]$. Its inverse image along $f$ gives the
pruning $K_f$ with empty
marking, because inverse images of $0$ are empty. Thus its
kernel is $(0,K_f)$. Every quotient occurs as $K_q$ for
its own representative $q$, so these are all normal clusters.
The top cluster over $X$ is the whole collapse with its marked point;
for $X=0$ it is its unique cluster. Pushing it forward along
$f$ collapses the matched subtree $I_f$ and marks the
resulting leaf. Hence its image is
$(I_f,\alpha_Y(I_f))$. Every subtree occurs as the image
subtree of its inclusion, so these are all conormal clusters.

The inequalities $0\le A$ and $\alpha_X(A)\le R$ show that
both proposed parts lie below $(A,R)$. Any common upper bound
$(B,S)$ has $A\le B$ and $R\le S$, so it lies above
$(A,R)$. This proves the join formula. A normal cluster below
$(A,R)$ has its quotient at most $R$; a conormal cluster
below it has its subtree at most $A$. Monotonicity of $\alpha$
therefore gives the two maximality assertions.
\end{proof}

\begin{example}
Let $X=\lambda(1,1)$ be a binary corolla. Marking either leaf
in the identity quotient gives a cluster. Their join collapses
the entire corolla and marks the resulting empty leaf. Indeed,
a common upper bound must have an inverse-image subtree
containing both leaves, hence the whole corolla; its quotient
must collapse that subtree. This example shows concretely how
a join can require additional pruning.
\end{example}

\subsection{Recovering expression tree categories from the form}
\label{sec:tree-form-characterization}
We now consider the converse question. Suppose we start with a
category carrying a noetherian form with exact join decomposition,
without knowing that its objects represent trees. Which additional
properties allow the trees and their matchings to be recovered?
The form supplies quotient and embedding maps, namely its
universal quotients and universal embeddings. These give a proper
factorization system $(\mathcal E,\mathcal M)$, on which we
will impose the additional conditions. A useful part of the
argument is that these conditions recover the morphism classes
used in our nine axioms: on the full subcategory of nonempty
objects, the quotient maps become all epimorphisms, and the
embedding maps become all strong monomorphisms.

The following consequences of exact join decomposition will be
used in the characterization. We retain the notation
$\mathsf S(X)$ for the lattice of subobjects represented by
maps in $\mathcal M$, and $\mathsf Q(X)$ for the lattice of
quotients represented by maps in $\mathcal E$. The map $\alpha_X$
sends a subobject to the quotient induced by its largest quotient,
as in the collapse construction above. These constructions now
use the factorization system of the given form.

\begin{lemma}\label{lem:ejd-background}
Let $\mathsf F$ be a noetherian form with exact join decomposition
over a category $\mathcal C$, with underlying factorization system
$(\mathcal E,\mathcal M)$.
\begin{enumerate}
\item The posets of $\mathcal M$-subobjects and
$\mathcal E$-quotients are bounded lattices. Pullbacks along
morphisms in $\mathcal M$ and pushouts along morphisms in
$\mathcal E$ exist.
\item The pushout of $u\in\mathcal M$ along $a\in\mathcal E$
has the form
\[
\xymatrix@C=3.5em@R=2.4em{
A\ar[r]^u\ar[d]_a&X\ar[d]^c\\
B\ar[r]_v&Y,
}
\]
where $v\in\mathcal M$ and $c\in\mathcal E$. This square is
also a pullback.
\item The map
\[
\alpha_X:\mathsf S(X)\longrightarrow\mathsf Q(X),\qquad
\alpha_X(A)=(m_A)_*^{\mathsf Q}\top_{\mathsf Q(A)},
\]
preserves binary meets and bottom elements.
\item If $0$ is an initial object, every morphism $0\to X$
belongs to $\mathcal M$, and every morphism in $\mathcal E$
with domain $0$ is an isomorphism.
\end{enumerate}
\end{lemma}
\begin{proof}
The lattice, pullback and pushout assertions in (a) are supplied
by \cite[Theorem 100]{JanelidzeVanNiekerk}. For (b), we use
the dual of the characterization in
\cite[Theorem 167]{JanelidzeVanNiekerk}. The dual of the first
assertion of (B1) gives $v\in\mathcal M$, and the dual of
(B2$'$) makes the pushout a pullback. Membership of $c$ in
$\mathcal E$ follows from stability of that class under pushout.
The preservation assertions in (c) are
\cite[Theorem 163(ii)]{JanelidzeVanNiekerk}.

To prove (d), every monomorphism into $0$ is an isomorphism:
its right inverse exists by initiality and is also a left inverse
by monicity. Thus $\mathsf S(0)$ has one element. Its image under
$\alpha_0$ is the largest $\mathcal E$-quotient of $0$, since
the identity represents this unique subobject. By (c), this image
is also the smallest quotient, represented by $\id_0$.
Consequently every $\mathcal E$-map out of $0$ is an isomorphism.
Factoring $0\to X$ then shows that it belongs to $\mathcal M$.
\end{proof}

\phantomsection\label{not:nonempty-subcategory}
For the rest of this subsection, suppose that $\mathcal C$ has a
strict initial object $0$ and a terminal object $1$, with
$0\not\cong1$. Call an object \emph{nonempty} if it is not
isomorphic to $0$, and let $\mathcal C^+$ be the full subcategory
of nonempty objects. Strictness means that no morphism from a
nonempty object has an initial object as codomain. In particular,
the factorization system restricts to $\mathcal C^+$.
We write $q_A:A\to1$ for the unique morphism to the terminal
object, as in the original axioms.

\begin{lemma}\label{lem:ejd-nonempty-classes}
Suppose that $q_A:A\to1$ belongs to $\mathcal E$ for every
nonempty object $A$. Then
\[
\mathcal E|_{\mathcal C^+}=\operatorname{Epi}(\mathcal C^+),
\qquad
\mathcal M|_{\mathcal C^+}
=\operatorname{StrongMono}(\mathcal C^+).
\]
\end{lemma}
\begin{proof}
Let $m:A\to X$ belong to $\mathcal M$, with $A,X$ nonempty,
and suppose that $m$ is epic in $\mathcal C^+$. Form the pushout
\[
\xymatrix@C=3.5em@R=2.4em{
A\ar[r]^m\ar[d]_{q_A}&X\ar[d]^c\\
1\ar[r]_j&Y.
}
\]
By Lemma~\ref{lem:ejd-background}, this square is bicartesian.
The object $Y$ is nonempty, since it receives a morphism from
$1$. Hence the square is also a pushout and a pullback in
$\mathcal C^+$. The map $j$ is epic there, as a pushout of
$m$. It is split monic, because its domain is terminal.
It is therefore an isomorphism. Its pullback $m$ is an
isomorphism as well.

Now factor an epimorphism $f$ in $\mathcal C^+$ as $f=me$,
where $e\in\mathcal E$ and $m\in\mathcal M$. The map $m$
is epic, so the preceding argument makes it an isomorphism.
Thus $f\in\mathcal E$. Conversely, every map of $\mathcal E$
is epic in $\mathcal C$, and remains epic in the full subcategory
$\mathcal C^+$. We have therefore identified the left class.
The right class consists of the monomorphisms with the unique
diagonal property against the left class. Since the latter is
now the class of all epimorphisms of $\mathcal C^+$, this is
exactly the definition of a strong monomorphism.
\end{proof}

We next show that the stability condition in
Axiom~\ref{ax:pointpushout} follows from the form once
pullback stability and detection by points are imposed. This
is the step which shortens the additional list of axioms.
Recall that the fibre of a morphism over a point is its pullback
object along that point. The fibre is terminal exactly when its
projection to $1$ is an isomorphism.

\begin{lemma}\label{lem:ejd-single-fibres}
Suppose that $q_A:A\to1$ belongs to $\mathcal E$ for every
nonempty object $A$,
that pullback of an $\mathcal E$-map along an $\mathcal M$-map
has its other projection in $\mathcal E$, and that an
$\mathcal E$-map between nonempty objects is an identity whenever
its pullbacks along all points are isomorphisms. Then the class
of $\mathcal E$-maps between nonempty objects whose pullbacks
along all but at most one point are isomorphisms is stable under
pushout along $\mathcal M$-maps between nonempty objects.
\end{lemma}
\begin{proof}
Every point belongs to $\mathcal M$, since it is a split
monomorphism. An $\mathcal E$-map with domain $1$ is an
isomorphism: its composite with the terminal map is $\id_1$,
and epicity makes its left inverse also a right inverse.
Pullback stability and Lemma~\ref{lem:ejd-background}(d) imply
that every point fibre of an $\mathcal E$-map between nonempty
objects is nonempty.

Let $m_A:A\to X$ belong to $\mathcal M$, with $A$ nonempty,
and form its collapse square
\[
\xymatrix@C=3.5em@R=2.4em{
A\ar[r]^{m_A}\ar[d]_{q_A}&X\ar[d]^{c_A}\\
1\ar[r]_j&Q_A.
}
\]
The square is bicartesian, and $[c_A]=\alpha_X(A)$ because
$q_A$ is the largest $\mathcal E$-quotient of $A$.
For a point $k:1\to Q_A$ different from $j$, let
$m_K:K\to X$ be its pullback along $c_A$.

The intersection of two distinct point subobjects is $0$.
Indeed, a nonempty intersection $H$ would give
$j q_H=k q_H$, and $q_H$ is epic. Pulling back this intersection
along $c_A$ therefore gives $A\wedge K=0$; the inverse image
of $0$ is $0$ by strictness. The equation
$c_A m_K=kq_K$ implies $\alpha_X(K)\le\alpha_X(A)$, by the
pushout property defining the collapse of $K$. Hence
\[
\alpha_X(K)
=\alpha_X(K)\wedge\alpha_X(A)
=\alpha_X(K\wedge A)
=\alpha_X(0)
=\bot_{\mathsf Q(X)}.
\]
The collapse of $K$ is thus an isomorphism. Its bicartesian
square makes $q_K$ an isomorphism, so $K\cong1$. We have
proved that $c_A$ has at most one nonterminal point fibre.

Conversely, let $d:X\to Y$ belong to $\mathcal E$ and have
at most one nonterminal point fibre. If every point fibre is
terminal, the detection assumption makes $d$ an identity.
Otherwise let $A$ be its fibre over the exceptional point
$p:1\to Y$. Collapse $A$ as above. The pushout property gives
a factorization $d=r c_A$ with $rj=p$, as displayed in
\[
\xymatrix@C=3.5em@R=2.4em{
A\ar[r]^{m_A}\ar[d]_{q_A}&X\ar[d]^{c_A}
  \ar@/^1.5pc/[dd]^d\\
1\ar[r]_j\ar@/_1.5pc/[dr]_p&Q_A\ar[d]^r\\
&Y.
}
\]
Factor $r=me$, with $e\in\mathcal E$ and $m\in\mathcal M$.
Then $d=m(ec_A)$ and $ec_A\in\mathcal E$. Since
$d\in\mathcal E$, uniqueness of factorization makes $m$ an
isomorphism. Hence $r\in\mathcal E$.
For any point $k:1\to Y$, let $L$ be its
fibre along $r$, and let $K$ be its fibre along $d$. Pullback
pasting gives
\[
\xymatrix@C=3.5em@R=2.4em{
K\ar[r]\ar[d]_g&X\ar[d]^{c_A}\\
L\ar[r]\ar[d]&Q_A\ar[d]^r\\
1\ar[r]_k&Y,
}
\]
where $g\in\mathcal E$ by pullback stability. If $k\ne p$,
then $K\cong1$, so $g$ is an isomorphism and $L\cong1$.
If $k=p$, identify $K$ with $A$. The equation $rj=p$ gives
a point $t:1\to L$ with $g=tq_A$. Since $g,q_A$ belong to
$\mathcal E$ and $t$ belongs to $\mathcal M$, uniqueness of
factorization makes $t$ an isomorphism. Thus $L\cong1$ in
this case as well. Detection applied to $r$ gives $r=\id$.
Every nonidentity map under consideration is therefore a collapse
$c_A$.

Finally, pushout pasting shows that pushing $c_A$ out along
$X\to Z$ in $\mathcal M$ gives the collapse of the composite
$A\to X\to Z$. The first part of the proof applies to that
collapse. Pushouts of identities are identities, which completes
the proof.
\end{proof}

\begin{theorem}[Characterization with an empty tree]
\label{thm:ejd-tree-characterization}
Let $\mathsf F$ be a noetherian form with exact join decomposition
over a small category $\mathcal C$, and let
$(\mathcal E,\mathcal M)$ be its underlying factorization system.
Suppose that $\mathcal C$ has a strict initial object $0$ and a
terminal object $1$, with $0\not\cong1$.
Then $\mathcal C$, together with its two morphism classes, is
isomorphic to an operadic expression tree category with a strict
initial empty tree adjoined, if and only if the following conditions
hold. Under the isomorphism, $\mathcal E$ consists of prunings
and the identity of the empty tree, and $\mathcal M$ consists
of occurrences and the morphisms from the empty tree.
\begin{enumerate}
\item Pullback of an $\mathcal E$-map along an $\mathcal M$-map
has its other projection in $\mathcal E$.
\item For nonempty $A,B$ and every point $p:1\to B$, the
category of pullback squares
\[
\xymatrix@C=3.5em@R=2.4em{
A\ar[r]^u\ar[d]_{q_A}&C\ar[d]^d\\
1\ar[r]_p&B,
}
\]
with $u\in\mathcal M$ and $d\in\mathcal E$, has a terminal
object. Its comparison maps fix $A$ and $B$, as in
Axiom~\ref{ax:U}.
\item For every object $X$, there are finitely many
$\mathcal M$-maps with codomain $X$ and finitely many
$\mathcal E$-maps with domain $X$, counting maps with all
possible other endpoints.
\item For $A,B\in\mathsf S(X)$, if $A\wedge B\ne0$, then
$A\le B$ or $B\le A$.
\item If the pullback of $d\in\mathcal E$ along every point
of its codomain is an isomorphism, then $d$ is an identity.
\end{enumerate}
Under this isomorphism, $\mathsf F$ is isomorphic to the form of
marked prunings in Definition~\ref{def:tree-clusters}.
\end{theorem}
\begin{proof}
First suppose that the category and its two classes arise from
an expression tree category with an empty tree adjoined.
Condition (a) is
Lemma~\ref{lem:tree-lattice-calculus}(b). The universal fillings
in (b) are those of the nonempty tree category. Their middle
objects cannot be initial, so adjoining $0$ adds no further
fillings or comparison maps. Finiteness in (c) follows from
finiteness of the trees; adjoining $0$ adds one inclusion into
each nonempty object and only the identity quotient at $0$.
Two nonempty subtrees which intersect are nested, giving (d).
Condition (e) is the point-detection axiom on the nonempty category;
the only $\mathcal E$-map with initial codomain is $\id_0$.
The existence of the stated form is Theorem~\ref{thm:tree-cosub}.

Conversely, assume (a)--(e). An isomorphism belongs to
$\mathcal E$ and all its pullbacks are isomorphisms. Condition
(e) therefore makes every isomorphism an identity. In particular,
$0$ is the only initial object. We work in $\mathcal C^+$.

Apply (b) with $B=1$ and $p=\id_1$. The pullback square has
an isomorphism as its upper map, and its right map belongs to
$\mathcal E$. Thus $q_A\in\mathcal E$ for every nonempty
$A$. Lemma~\ref{lem:ejd-nonempty-classes} now identifies the
restricted classes with the epimorphisms and strong monomorphisms
of $\mathcal C^+$.

We verify the nine axioms on $\mathcal C^+$. Its terminal object
is $1$, which gives Axiom~\ref{ax:N}. The restricted factorization
system, together with the preceding identification of its classes,
gives Axiom~\ref{ax:F}. Condition (a) supplies
Axiom~\ref{ax:Dpull}. The pullback
object in that axiom is nonempty: its projection in $\mathcal E$
has nonempty codomain, while an $\mathcal E$-map from an
initial object is an isomorphism by
Lemma~\ref{lem:ejd-background}(d).
The bicartesian pushouts of Lemma~\ref{lem:ejd-background}(b)
remain in $\mathcal C^+$, since their objects receive maps
from nonempty objects. They give Axiom~\ref{ax:Dpush}.
Lemma~\ref{lem:ejd-single-fibres} gives
Axiom~\ref{ax:pointpushout}.

Condition (b) is Axiom~\ref{ax:U}, and (c) implies
Axiom~\ref{ax:finite}. To obtain Axiom~\ref{ax:O}, consider
a commutative square of strong monomorphisms in $\mathcal C^+$.
Its two subobjects in the common codomain have a nonempty
subobject in common. Their intersection is therefore nonempty,
and (d) makes them comparable. The corresponding factor is a
diagonal of the square; its other equation follows by monicity.
Finally, (e) gives Axiom~\ref{ax:Ddetect}.

Theorem~\ref{thm:representation} consequently identifies
$\mathcal C^+$ with the tree category of its derived signature,
and identifies the restricted classes with prunings and subtree
occurrences. Strictness and initiality of $0$ extend this
isomorphism to the categories with the empty tree adjoined.
Lemma~\ref{lem:ejd-background}(d) shows that the additional
$\mathcal M$-maps are exactly the maps $0\to X$, and that
the only additional $\mathcal E$-map is $\id_0$.

After this identification, $\mathsf F$ and the marked-pruning
form have the same underlying proper factorization system.
Theorem~158 of \cite{JanelidzeVanNiekerk} states that a noetherian
form with exact join decomposition is determined, up to isomorphism,
by its conormal and normal subforms. By Theorem~100 of the same
paper, these two subforms are determined by the proper factorization
system. Hence $\mathsf F$ is isomorphic to the marked-pruning form.
\end{proof}

\begin{remark}
The theorem retains the counterparts of
Axioms~\ref{ax:Dpull}, \ref{ax:U}, \ref{ax:finite},
\ref{ax:O} and~\ref{ax:Ddetect}. Exact join decomposition
supplies the bicartesian pushout axiom, and its compatibility
with intersections gives the remaining pushout stability axiom.
The assertion in (d) expresses the geometry of a single rooted
tree: two subtrees with a common nonempty part lie one inside
the other. The empty intersection is permitted, which allows
distinct branches.
\end{remark}

\begin{remark}\label{rem:ejd-intrinsic-classes}
The proof identifies the morphism classes of the form on
$\mathcal C^+$ without separately requiring them to be all
epimorphisms and strong monomorphisms. Universal filling puts the
nonempty terminal maps in $\mathcal E$, and the bicartesian
collapse square then gives the identification in
Lemma~\ref{lem:ejd-nonempty-classes}. This argument already
works before finiteness and the nesting of intersecting subobjects
are used.

On the enlarged category, the specified factorization system
remains part of the characterization. For example, adjoining
$0$ to the expression tree category containing only
$1$ gives $0\to1$, with identities as $\mathcal E$ and all
maps as $\mathcal M$. The nonidentity map is epic, but is
not strong monic. These classes nevertheless form a proper
factorization system, and Theorem~\ref{thm:tree-cosub} applies.
The same distinction accounts for the augmented word examples.
\end{remark}

\subsection{The pointed example}\label{sec:word-noetherian-form}
For a pointed expression tree category, the reconstruction in
Theorem~\ref{thm:wordcharacterization} gives unary generating
operations. Every expression can therefore be read as a word.
The unique empty leaf gives a preferred marking in every pruned
word, and the empty word is already a zero object. We can use
these markings to construct a noetherian form directly on the
word category.

\begin{theorem}\label{thm:word-noetherian}
On a word category, consider the cospans
\[
\xymatrix@C=3.5em{
X\ar[r]^q&Q&1\ar[l]_u,
}
\]
where $q$ is a pruning and $u$ is the unique point of $Q$.
With the comparisons of Definition~\ref{def:tree-clusters},
these cospans form a noetherian form with exact join decomposition.
Its universal quotients are the prunings, and its universal
embeddings are the occurrences.
Its fibre over a word of length $n$ is the chain
$\{0,1,\ldots,n\}$ with its usual order, where a cluster
$d$ records the deletion of the suffix of length $d$.
\end{theorem}
\begin{proof}
There is exactly one marking in each quotient, so a cluster is
specified by the number $d$ of letters deleted from the end of
the word. The marking imposes no additional condition on a
quotient comparison, since every morphism preserves the unique
point. Thus comparisons over an identity give the usual order
on these numbers. In particular, each fibre is a bounded lattice,
whose meet and join are minimum and maximum.

Let $f:X\to Y$ retain a prefix of $X$ of length
$r$ and identify it with a suffix of $Y$. Write $n=|X|$ and
$k=n-r$. Thus $k$ is the number of letters deleted by $f$.
On the chains whose elements count deleted suffix lengths, the
direct and inverse images are
\[
 f_*d=\max(d-k,0),\qquad f^*e=k+\min(e,r).
\]
The first formula counts the letters of the additionally deleted
suffix that survive the pruning of $f$. The second counts
those already deleted by $f$, together with the letters in its
retained prefix that are deleted by the target quotient. They
give the least and greatest comparisons over $f$, respectively.
Indeed, $\max(d-k,0)\le e$ is equivalent to
$d\le k+\min(e,r)$ for $0\le d\le n$ and
$0\le e\le |Y|$. The quotient pushout and factorization
constructions give these same operations, so they are functorial.
We have therefore obtained an orean form.

The kernel of $f$ is $f^*(0)=k$ and its image is
$f_*(n)=r$. Direct calculation gives
\[
 f^*f_*d=\max(d,k),\qquad f_*f^*e=\min(e,r).
\]
Every cluster $d$ is the kernel of the pruning deleting $d$
letters and the image of the inclusion of the suffix of length
$d$. The pruning is its universal quotient because a matching
factors through it precisely when it deletes at least those
$d$ letters. The suffix inclusion is its universal embedding
because a matching factors through it precisely when its matched
suffix has length at most $d$. The first factorization is unique
because the pruning is epic; the second is unique because the
inclusion is monic. Applied to $k$ and $r$, these statements
identify the ordinary pruning--occurrence factorization of $f$
with the universal factorization required by the noetherian axioms.

Every cluster is both normal and conormal. Consequently binary
joins of normal clusters and binary meets of conormal clusters
have the required types. This completes the verification that
the form is noetherian. Its largest normal and conormal parts,
denoted earlier by $n(d)$ and $c(d)$, therefore satisfy
$n(d)=c(d)=d$ for every cluster $d$. The join decomposition and its two compatibility
identities reduce to equalities between identical expressions.
The form therefore has exact join decomposition.
\end{proof}

For the augmented word category, the original quotients acquire
also the option of an empty marking. Thus the fibre has two
clusters for each deleted length: an empty marking and the
unique point marking. Theorem~\ref{thm:tree-cosub} applies to
this augmented category as well. The unaugmented construction
in Theorem~\ref{thm:word-noetherian} explains why the extra
empty tree is unnecessary in the pointed example.

\section{Concluding remarks}\label{concl:section}

\subsection{Recognition up to equivalence}\label{concl:equivalence-axioms}

The representation theorem recognizes expression tree categories up to
isomorphism. We first explain how to obtain a recognition up to
equivalence instead. This permits several isomorphic copies of the
same tree, while preserving the distinction between its
different positions.

For a fixed object $X$, let $\operatorname{Sub}_{\mathrm{s}}(X)$
be the set of strong monomorphisms into $X$, considered up to
isomorphism over $X$. Thus $m:A\to X$ and $m':A'\to X$ represent
the same element when there is an isomorphism $\alpha:A\to A'$
with $m=m'\alpha$. Similarly, let
$\operatorname{Quot}_{\mathrm{e}}(X)$ be the set of epimorphisms
out of $X$, considered up to isomorphism under $X$: the maps
$e:X\to B$ and $e':X\to B'$ represent the same element when
$e'=\beta e$ for an isomorphism $\beta:B\to B'$.
The two identifications are illustrated by
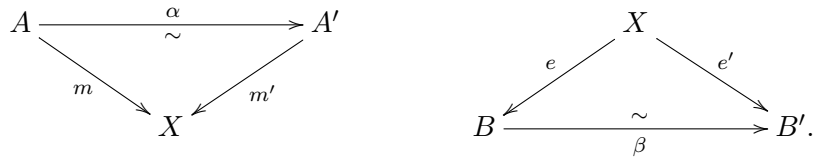
\[
\xymatrix@C=3.8em@R=2.3em{
A\ar[rr]^{\alpha}_{\sim}\ar[dr]_m&&A'\ar[dl]^{m'}\\
&X&
}
\qquad\qquad
\xymatrix@C=3.8em@R=2.3em{
&X\ar[dl]_e\ar[dr]^{e'}&\\
B\ar[rr]_{\beta}^{\sim}&&B'.
}
\]
The object $X$ remains fixed in these diagrams. In particular,
two equal subtrees at different positions still determine
different strong subobjects.

Retain Axioms~\ref{ax:N}--\ref{ax:U} and~\ref{ax:O}, and make
the following two replacements.
\begin{enumerate}
\item Replace Axiom~\ref{ax:finite} by the requirement that
$\operatorname{Sub}_{\mathrm{s}}(X)$ and
$\operatorname{Quot}_{\mathrm{e}}(X)$ are finite for every $X$.
\item Replace Axiom~\ref{ax:Ddetect} by the following requirement.
If an epimorphism $d:X\to Y$ has isomorphic pullbacks along
all points of $Y$, then $d$ is the unique isomorphism from
$X$ to $Y$.
\end{enumerate}
Here an isomorphic pullback means that the projection to the
domain of the point is an isomorphism, as in the original
axiom. The second replacement can equivalently be stated as
two requirements: such an epimorphism is invertible, and every
automorphism in the category is an identity. Indeed, every
isomorphism satisfies the pullback hypothesis, and two
isomorphisms with the same endpoints differ by an automorphism.

\begin{theorem}\label{concl:equivalence}
For an essentially small, locally small category $\C$, the
following conditions are equivalent.
\begin{enumerate}
\item The category satisfies the modified axioms above.
\item It is equivalent to a small category satisfying the
original nine axioms.
\item It is equivalent to the expression tree category
$\mathcal C_L$ associated to a ranked set $L$.
\end{enumerate}
\end{theorem}
\begin{proof}
Suppose first that the modified axioms hold. Choose a skeleton
$\C_0$ of $\C$, that is, a full subcategory containing one
object from each isomorphism class. Essential smallness permits
this subcategory to be small. The modified detection axiom
implies that every isomorphism in $\C_0$ is an identity.
Consequently, two strong monomorphisms into an object of
$\C_0$ represent the same strong subobject only when they
are equal. The corresponding assertion holds for epimorphic
quotients. The modified finiteness axiom therefore gives the
original finiteness axiom in $\C_0$, and the modified detection
axiom gives the original detection of identities.

The remaining axioms pass to $\C_0$. An equivalence preserves
epimorphisms, strong monomorphisms, and terminal objects.
It also preserves the universal properties of pullbacks and
pushouts. The bijections of points preserve the condition
that at most one point is exceptional. For
Axiom~\ref{ax:U}, transporting the prescribed objects and maps
along isomorphisms gives equivalent categories of pullback
factorizations, so their terminal objects are preserved.
The existence of either diagonal in Axiom~\ref{ax:O} is
preserved by fullness and faithfulness. Hence $\C_0$ satisfies
all nine original axioms. Theorem~\ref{thm:representation}
now gives an isomorphism $\C_0\cong\mathcal C_L$.

Conversely, every category satisfying the original axioms has
only identity isomorphisms. It therefore satisfies their
modified versions. The same preservation arguments, together
with the induced bijections of strong subobject and
epimorphic quotient sets, show that the modified versions
pass across an equivalence. This proves all three implications.
\end{proof}

In this formulation, components are maximal proper strong
subobject classes. Thus their definition uses noninvertible
maps rather than merely nonidentity maps. The labels are
recovered as the isomorphism classes of the nonterminal
objects whose components have terminal domains.
The arity of such a label is its number of components.
Theorem~\ref{thm:uniform-ordering} then supplies compatible
orders on positions after an order has been chosen on the
inputs of each recovered label. Equivalence preserves the
existence of these choices; it does not select a preferred
one.

After adjoining a strict initial object, the noetherian form
of Section~\ref{sec:cosubquotients} also passes across an
equivalence. Its clusters are represented by cospans considered
up to isomorphism. An equivalence therefore identifies their
fibre lattices and preserves the universal constructions
defining direct and inverse images. The noetherian identities
and exact join decomposition are preserved with them.

\begin{remark}\label{concl:rigidity-remark}
Detection of isomorphisms alone would permit extra symmetries.
For example, fix a nontrivial finite group $G$ and form a
category with objects $a,1$, with $\operatorname{Aut}(a)=G$,
a unique map $a\to1$, and no map $1\to a$. All morphisms
are epic and the strong monomorphisms are precisely the
isomorphisms. This category satisfies
Axioms~\ref{ax:N}--\ref{ax:U} and~\ref{ax:O}, the modified
finiteness condition, and detection of isomorphisms by point
pullbacks. Its nontrivial automorphisms prevent it from being
equivalent to an expression tree category of the present
paper. The uniqueness clause excludes this extra information.
Conversely, counting actual arrows would exclude even a
category with infinitely many objects and exactly one
morphism between every ordered pair, although that category
is equivalent to the category with one object and one
morphism. These two examples explain the separate roles of
the replacements.
\end{remark}

\subsection{Pullback factorizations of arbitrary matchings}
\label{concl:pullback-complements}

Axiom~\ref{ax:U} describes how to insert an expression at an empty
leaf. The same question can be asked at a place already occupied
by an expression. Suppose that $e:A\to D$ is a pruning and that
$m:D\to B$ selects a subtree of $B$. We would like to replace this
copy of $D$ by $A$, so that applying $e$ there recovers $B$. The
following extension of the axiom gives a universal description of
this replacement.

For an epimorphism $e:A\to D$ and a strong monomorphism
$m:D\to B$, let $\Fill(e,m)$ be the category whose objects are
pullback squares
\begin{equation}\label{concl:eq:mixed-complement}
\vcenter{\xymatrix@C=4em@R=2.5em{
A\ar[r]^u\ar[d]_e&C\ar[d]^c\\
D\ar[r]_m&B
}}
\end{equation}
with $c$ epic. The morphism $u$ is automatically a strong
monomorphism, being a pullback of $m$. A comparison from
$(C',c',u')$ to $(C,c,u)$ is a morphism $h:C'\to C$ satisfying
$hu'=u$ and $ch=c'$. Thus the comparison preserves both the
inserted expression and the morphism which removes it again.

\begin{theorem}\label{concl:thm:mixed-complements}
In a category satisfying the nine axioms, $\Fill(e,m)$ has a
terminal object for every epimorphism $e$ and strong monomorphism
$m$ as above. The square of this terminal object is also a
pushout. In the representation by trees, its object $C$ is
obtained by replacing the subtree selected by $m$ with $A$.
\end{theorem}
\begin{proof}
By Theorem~\ref{thm:representation}, we may work in the represented
tree category. Let $a$ be the address selected by $m$, so that
$B|_a=D$. Form $C$ by replacing $B|_a$ with $A$, and let $u$
select this copy of $A$. Applying the pruning $e$ there gives
$c:C\to B$. The description of pushouts of occurrences along
prunings shows that~\eqref{concl:eq:mixed-complement} is a pushout;
it is also a pullback by Axiom~\ref{ax:Dpush}.

Consider another object $(C',c',u')$ of $\Fill(e,m)$. Suppose
first that $B\ne1$. Every address of $B$ is retained by the
pruning $c'$. Its pullback over the subtree at $a$ is therefore
the subtree $C'|_a$, with the restricted pruning to $D$.
The prescribed pullback identifies this subtree with $A$, its
occurrence with $u'$, and its restricted pruning with $e$.
Apply the cuts of $c'$ outside this copy of $A$, leaving $A$
itself unchanged. The resulting tree is $C$, and the resulting
pruning $h:C'\to C$ satisfies $hu'=u$ and $ch=c'$.

Any comparison with these properties must be a pruning. Indeed,
$c'$ sends the root to the root of the nonterminal tree $B$.
The only position sent to that root by $c$ is the root of $C$.
Consequently the occurrence in the normal factorization of a
comparison selects the whole of $C$. Uniqueness now follows
because there is at most one pruning between two given trees.
If $B=1$, then $m$ is an identity, and the pullback forces
$u$ to be an identity as well. The object $(A,e,\id_A)$ is
then the required terminal object.
\end{proof}

Taking $e=q_A$ and $m=j$ recovers Axiom~\ref{ax:U}, since
$q_A$ is epic by Lemma~\ref{lem:empty}.

This result concerns comparisons which fix $A$. The literature
also considers a stronger property which allows the expression
over $D$ to vary. A square~\eqref{concl:eq:mixed-complement} is
called a \emph{final pullback complement} of the composable pair
$(e,m)$ if it is a pullback and has the following property.
For every pullback with upper-left object $A'$ and lower-left
object $C'$ in the diagram
\[
\xymatrix@C=3.7em@R=3em{
A'\ar[r]^f\ar[d]_{u'}\ar@/^2pc/[rr]^{e'}&
 A\ar[r]^e\ar[d]^u&D\ar[d]^m\\
C'\ar@{-->}[r]_h\ar@/_2pc/[rr]_{c'}&C\ar[r]_c&B,
}
\]
where the outer rectangle is that pullback and $ef=e'$, there
is a unique $h$ making the diagram commute. Here $c'$ and $f$
are arbitrary morphisms. In particular, the definition places
no epic requirement on $c'$. This is the property used in
\cite[Definition~1]{CorradiniEtAl}. It implies terminality in
$\Fill(e,m)$ when the resulting map $c$ is epic.

\begin{theorem}\label{concl:thm:nonconstant-final-complements}
If $D\ne1$, the terminal object constructed in
Theorem~\ref{concl:thm:mixed-complements} is a final pullback
complement. Consequently, the epi--strong-mono factorization
of every nonconstant matching has such a complement.
\end{theorem}
\begin{proof}
Retain the tree description and address $a$ from the preceding
proof. Given the diagram above, write the normal factorization
of $c'$ as a pruning followed by an occurrence at an address
$b$ of $B$. The addresses $a$ and $b$ are comparable. Indeed,
the equation $me'=c'u'$ gives a morphism whose image occurs
in both selected subtrees, whereas subtrees at incomparable
addresses have no common occurrence.

Suppose first that $b$ lies at or below $a$. Then $c'$ factors
through $m$. Since $m$ is monic, its pullback along $c'$ has
$u'$ invertible. There is exactly one possible comparison,
namely $h=uf(u')^{-1}$, and the defining equations show that
it has the required properties.

It remains to consider the case where $b$ is a proper ancestor
of $a$. The pullback fibre $A'$ is now the subtree of $C'$ at
the address corresponding to $a$, and $e':A'\to D$ is the
restriction of the pruning part of $c'$. In particular, $e'$
is epic. The equation $ef=e'$ forces $f$ to be epic as well:
$e'$ sends the root to the root of $D$, and, since $D\ne1$,
the only position sent there by $e$ is the root of $A$.
Thus the occurrence in the normal factorization of $f$
selects all of $A$.

In the pruning part of $c'$, replace its restriction $e'$ at
this subtree by $f$, retaining its other cuts. This gives a
pruning to the subtree of $C$ at $b$. Follow it by that
subtree's occurrence in $C$. The resulting map $h$ satisfies
$ch=c'$ and $hu'=uf$. For uniqueness, $c$ changes positions
only at or below $a$, so its only preimage of the proper
ancestor $b$ is $b$ itself. Every possible comparison must
therefore select the subtree at $b$, and its pruning part is
unique. This proves finality.

Finally, the overlap in the normal factorization of a
matching is $1$ precisely when that matching is constant.
This gives the last assertion.
\end{proof}

\begin{remark}\label{concl:rem:constant-not-final}
The terminality in Axiom~\ref{ax:U} can be strictly weaker
than finality. In the word category with one letter $x$, the
terminal completion of $x\to1\to x$ is
\[
\xymatrix@C=4em@R=2.5em{
x\ar[r]^u\ar[d]_{q_x}&xx\ar[d]^c\\
1\ar[r]_j&x,
}
\]
where $u$ includes the suffix $x$ and $c$ prunes to the
prefix $x$. The pullback of $j$ along $\id_x$ has fibre $1$.
Together with the map $j:1\to x$ of that fibre into the
prescribed fibre, standard finality would supply a section
$x\to xx$ of $c$. There is no such matching: a section would
have to send the root to the root, requiring $xx$ to be a
prefix of $x$. Thus the displayed square is not a final
pullback complement.
\end{remark}

The relation with dependent products explains why the stronger
property occurs in other familiar categories. Suppose that
pullback along a monomorphism $m:D\to B$ has a right adjoint
\[
m^*:\C/B\longrightarrow\C/D,
\qquad \Pi_m:\C/D\longrightarrow\C/B.
\]
Such a morphism is called \emph{exponentiable}. The adjunction
means that a map from a fibre over $D$ to $e:A\to D$ determines
exactly one map, over $B$, to $\Pi_m(e)$. This is already the
comparison property in the definition of a final pullback
complement. Monicity ensures that the prescribed fibre is
recovered exactly. Indeed, composition with $m$ gives a left
adjoint $\Sigma_m$ to $m^*$, and $m^*\Sigma_m\cong\id$.
For every object $z$ of $\C/D$, the two adjunctions give
\[
\begin{aligned}
(\C/D)(z,m^*\Pi_m e)
&\cong(\C/B)(\Sigma_m z,\Pi_m e)\\
&\cong(\C/D)(m^*\Sigma_m z,e)
\cong(\C/D)(z,e).
\end{aligned}
\]
Hence the counit $m^*\Pi_m(e)\to e$ is invertible, and the
object $\Pi_m(e)$ gives a final pullback complement. The
relationship between exponentiable monomorphisms and these
universal completions was established in
\cite{DyckhoffTholen}. For a general composable pair, whenever
the indicated right adjoint exists at $e$, invertibility of
the counit is precisely the additional condition for a final
pullback complement \cite[Lemma~11]{CorradiniEtAl}.
These completions allow a rewriting rule to remove or replace
a chosen part while determining the surrounding object by a
universal property, as in \cite{CorradiniEtAl}.
Hosseini, Tholen, and Yeganeh develop related completions and
their behavior under composition of diagrams, with applications
to computing pullbacks in categories of spans
\cite{HosseiniTholenYeganeh}. Axiom~\ref{ax:U} therefore has
a close connection with an established construction, with
the prescribed epic and strong monic maps specifying the
comparisons appropriate to expression trees.

In an elementary topos every morphism is exponentiable;
explicit constructions of the dependent product are given
in \cite{CaramelloZanfa}. To obtain an object of $\Fill(e,m)$,
its map to $B$ must also be epic. For example, this follows
when $e$ has a section, because applying $\Pi_m$ to the section
gives a section of $\Pi_m(e)\to B$. This explains its use for
pointed objects: the distinguished point of $A$ splits
$q_A$. For a general epimorphism $e$, the epimorphism requirement
on the resulting map to $B$ remains a separate condition.

\subsection{Other categorical descriptions of operadic trees}
\label{concl:operadic-literature}

The trees in this paper belong to a familiar description of
operations by expressions. Their use in universal algebra,
rewriting, and recognition of expressions is developed in
\cite{BurrisSankappanavar,BaaderNipkow,TATA}. In operad theory,
the same ordered trees describe the operations of a free
non-symmetric operad, with composition given by insertion
at empty leaves \cite[Section~2.3]{Leinster}. An operad records
how operations can be composed through their inputs. It can
also impose equations between the resulting expressions;
an algebra for the operad realizes the operations as actual
functions respecting those equations. This gives a common
language for algebraic structures and for compositions in
topology and higher category theory
\cite[Introduction and Section~2.2]{Leinster}.
Our objects are these free operations themselves. Passing
between two such objects by prefix--suffix matching gives
additional morphisms which record how part of one expression
can be found in another.

Kock's description by polynomial functors is particularly
close at the level of trees \cite{KockTrees}. For a ranked
set $L$, choose an input set $P_\lambda$ of size
$\operatorname{ar}(\lambda)$ for each label, and consider
\[
P(Z)=\coprod_{\lambda\in L}Z^{P_\lambda}.
\]
An element of $P(Z)$ specifies a label and one element of
$Z$ for each of its inputs. Kock's $P$-trees for this
one-colour polynomial give our trees: edges record positions,
vertices carry labels, and each vertex has its inputs
identified with $P_\lambda$. The tree $1$ is a single edge
without a vertex. A nullary label is instead a vertex with
an output edge and no input edges. The distinction is thus
preserved. Entire-subtree embeddings in that description
give our occurrences. A prefix also embeds into its original
tree, but our pruning has the reverse direction. Consequently,
the data of a matching can be displayed as
\[
\xymatrix@C=6em{
X&P\ar[l]_{\text{prefix}}
\ar[r]^{\text{subtree}}&Y.
}
\]
The convention which sends an erased position to the boundary
of a cut is an additional part of our composition rule.

Huet's zipper gives a related description of a selected
subtree and its surrounding context \cite{Huet}. The
context remembers the successive ancestors, the chosen
child at each step, and the remaining branches. In the
fixed-arity case, these data describe exactly the position
information carried by an occurrence. Huet uses them to
navigate and edit trees efficiently. A prefix--suffix
matching additionally specifies a pruning of its source.
Thus an occurrence admits a zipper description, whereas
the full matching also carries the information removed
before that occurrence is selected.

The cut constructions of Kremnizer and Szczesny, and
Szczesny's incidence categories, are closer at the level of
morphisms \cite{KremnizerSzczesny,SzczesnyIncidence,SzczesnyPreLie}.
Their morphisms identify a root-containing part of a source
forest with a collection of branches in a target forest.
This permits cuts and extensions to be counted in the
construction of algebras. For connected trees and a
nonempty matched part, the identification has the same
prefix--suffix pattern as ours. Their ordinary rooted
forests omit empty input places, and their tree isomorphisms
can rearrange branches. The coloured version permits
repeated vertex colours which must be preserved, as our
labels are preserved. On forgetting empty leaves and input
places, a matching in our category gives such cut data;
all constant matchings then have empty matched part. For
example, the two points of a binary corolla become the same
empty matching. Retaining empty leaves is therefore
essential to the positional information in our category.
For unary labels, the comparison reduces to coloured chains,
as discussed in Section~\ref{sec:words}.

Moerdijk and Weiss use a different category of trees to
study operads \cite{MoerdijkWeiss}. A tree there determines
a free operad with one colour for each edge and one
generating operation for each vertex. A colour specifies
the type of an input or output, so only matching types
can be composed. Morphisms between the trees are maps
between these operads; in particular, a vertex operation
can be sent to a composite operation. The tree shapes
overlap with ours, but the morphisms express a different
relationship between them. For instance, the one-vertex
unary tree has three endomorphisms in their category,
corresponding to the three order-preserving maps of a
two-element chain. A one-vertex unary tree in our category
has two: its identity and its constant matching.
Weiss also characterizes the dendroidal tree category by
relations which express how a finite collection of inputs
combines into an output \cite[Theorem~5.6]{WeissBroadPosets}.
That result recognizes trees together with the operad maps
just described. Theorem~\ref{thm:representation} instead
recognizes the pruning and occurrence structure intrinsic
to prefix--suffix matching.

\subsection{Permutations of the inputs}\label{concl:symmetry}
We can also change which matchings are allowed by permitting
a matching to permute the children at each vertex that it retains.
The permutation is then part of the morphism. In particular, two
isomorphic branches at different positions can be interchanged by an
automorphism of their parent tree. We give the corresponding modification
of the axioms and explain why the noetherian form survives this change.

Let $\C$ be an essentially small, locally small category. Write
$\operatorname{Sub}_{\mathrm{s}}(X)$ and
$\operatorname{Quot}_{\mathrm{e}}(X)$ for strong subobjects and
epimorphic quotients, with the comparison isomorphisms required to
commute with their maps into or out of $X$. A \emph{component} of $X$
is a maximal proper strong subobject; denote their set by $\Comp(X)$.
A \emph{corolla} is a nonterminal object whose component domains are
terminal. These definitions permit components with isomorphic domains
to remain distinct subobjects.

Retain Axioms~\ref{ax:N}--\ref{ax:U} and~\ref{ax:O}. Replace
Axioms~\ref{ax:finite} and~\ref{ax:Ddetect}, respectively, by the
following two conditions, and impose the third condition as well.
\begin{enumerate}
\item For every $X$, the sets $\operatorname{Sub}_{\mathrm{s}}(X)$
and $\operatorname{Quot}_{\mathrm{e}}(X)$ are finite.
\item An epimorphism is an isomorphism whenever its pullback along
every point of its codomain is an isomorphism.
\item For every corolla $\phi$, the action on components induces
an isomorphism
\begin{equation}\label{concl:corolla-symmetry}
 \operatorname{Aut}_{\C}(\phi)
 \xrightarrow{\ \sim\ }\operatorname{Sym}(\Comp(\phi)),
 \qquad g\longmapsto\bigl([u]\longmapsto[gu]\bigr).
\end{equation}
\end{enumerate}
The last condition says both that every permutation is realized and
that its realization is unique. For a nullary or unary corolla it
requires the automorphism group to be trivial. All these conditions
are invariant under equivalence of categories: an equivalence
identifies the subobject and quotient sets and their universal
constructions, and conjugates the component action in
\eqref{concl:corolla-symmetry} by the induced bijection of components.

For a ranked set $L$, define $\mathcal T_L^{\mathrm{sym}}$ using the
same indexed trees as before. An isomorphism of these trees preserves
labels and empty leaves and may bijectively rearrange the children
at every labeled vertex. A morphism $X\to Y$ is a triple
\begin{equation}\label{concl:symmetric-matching}
 (K,p,\theta),\qquad
 \theta:X/K\xrightarrow{\sim}Y|_p,
\end{equation}
where $K$ is an effective cut and $p$ is a position of $Y$.
Composition restricts the second cut to the subtree selected by the
first matching and transports it through $\theta$. Permutations
compose at retained vertices; permutations inside an erased subtree
disappear with that subtree. If the selected subtree is entirely
erased, the composite selects the empty leaf at the boundary of the
cut. Thus the construction retains the positional information of
our original composition rule.

\begin{theorem}\label{concl:symmetric-representation}
An essentially small, locally small category satisfies the modified
axioms above if and only if it is equivalent to
$\mathcal T_L^{\mathrm{sym}}$ for a ranked set $L$. The labels are
recovered as the isomorphism classes of corollas, and their arities
are the numbers of components. Under the equivalence, epimorphisms
are prunings followed by tree isomorphisms, and strong monomorphisms
are entire-subtree occurrences preceded by tree isomorphisms.
\end{theorem}
\begin{proof}
We first describe the reconstruction, retaining comparison
isomorphisms at the steps where the original proof gave equalities.
Terminal maps are epic: apply Axiom~\ref{ax:U} with
$j=\id_1$, so that the upper map of its pullback square is
invertible and its right map is epic. Points are strong
monic because they are split monic. It follows also that a
strong monomorphism into $1$, or an epimorphism out of $1$,
is an isomorphism.

For each $X$, its finite strong-subobject order is a tree branching
order: it has greatest element $[\id_X]$, and the subobjects above
any fixed subobject form a chain by Axiom~\ref{ax:O}.
Postcomposition by a strong monomorphism $A\to X$ identifies
the subobjects of $A$ with those below its class. Thus each position
has a unique finite path of component inclusions to the root.

Push out $q_A$ along a strong monomorphism $u:A\to X$, obtaining
the elementary collapse $X\to X/u$. Its fibre at the new point is
$A$, and Axiom~\ref{ax:pointpushout} makes every other point fibre
terminal. Consequently, a strong monomorphism with nonterminal
domain which becomes constant under this collapse factors through
$u$. This is the argument of Lemma~\ref{lem:constantfactor}, with
terminal domains understood up to isomorphism. Factoring the image
of a disjoint subobject, and applying point detection to its epic
factor, then shows that this subobject remains strong monic.

Every noninvertible epimorphism has a nonterminal point fibre.
Collapsing that fibre factors the epimorphism through a
noninvertible elementary collapse. Repetition strictly increases
the quotient class of the original domain, so finiteness makes it
terminate. Every epimorphism is therefore a sequence of elementary
collapses followed by an isomorphism.

The component calculation in
Lemma~\ref{lem:componentpushout} applies to subobject classes:
inverse image preserves properness, and direct image after inverse
image is the identity. A pushout at one component consequently
replaces that component and preserves the other components. Applying
this successively shows that an epimorphism with nonterminal target
induces a bijection of components and has epic component
restrictions. Collapsing every component gives a corolla $r(X)$ and
an epimorphism $\rho_X:X\to r(X)$. The pushout property identifies
root corollas along epimorphisms with nonterminal targets.

Induction on the finite strong-subobject tree now classifies
quotients as effective cuts. A quotient with nonterminal target
is obtained by its component restrictions; a quotient with
terminal target collapses the whole tree. Collapses at disjoint
positions commute by their pushout property, and a collapse at
an ancestor absorbs collapses below it. Thus the successive
collapses may be recorded by pairwise incomparable source
positions with nonterminal domains. These positions are
recovered from the quotient as the maximal source positions
with nonterminal domains whose images have terminal domains.
This proves uniqueness as well as existence of the cut.

We next explain the role of filling. A terminal singleton filling
has terminal fibres at its unmarked points: collapsing any additional
nonterminal fibres gives an epic comparison $e:C\to C'$ to
another filling. Its comparison $h:C'\to C$ back to the
terminal filling satisfies $he=\id_C$ by terminality.
Epicity of $e$ then gives $eh=\id_{C'}$. The cut description
excludes a noninvertible collapse with this property. The unmarked
points therefore lift uniquely, so singleton filling can be iterated
at finitely many distinct points. The proof of
Theorem~\ref{thm:finite} gives a terminal simultaneous filling.
Its comparisons are epic: a proper strong image would lie in a
component and would remain proper after the epic map to the
nonterminal base.

In particular, prescribe an object at every component point of a
corolla. Every complete filling is isomorphic to the terminal one.
Indeed, its comparison has isomorphic restrictions at all components,
so every point pullback is an isomorphism and the modified detection
axiom applies. For a nullary corolla this argument uses detection
with an empty set of points. Thus complete fillings are uniquely
isomorphic when the corolla and the component inclusions are fixed.

Suppose $X$ and $Y$ have isomorphic root corollas. Condition
\eqref{concl:corolla-symmetry} assigns a unique root-corolla
isomorphism to each component bijection. The universal property of
complete fillings then gives
\begin{equation}\label{concl:recursive-isomorphisms}
 \operatorname{Iso}_{\C}(X,Y)\cong
 \coprod_{\sigma:\Comp(X)\xrightarrow{\sim}\Comp(Y)}
 \prod_{i\in\Comp(X)}\operatorname{Iso}_{\C}(X_i,Y_{\sigma(i)}).
\end{equation}
For clarity, prescribed restrictions $a_i:X_i\to Y_{\sigma(i)}$
extend to the unique comparison $h$ in
\[
\xymatrix@C=3.2em@R=2.4em{
X_i\ar[r]^{u_i}\ar[d]_{a_i}&X\ar[r]^{\rho_X}\ar[d]^h
 &r(X)\ar[d]^\beta\\
Y_{\sigma(i)}\ar[r]_{v_{\sigma(i)}}&Y\ar[r]_{\rho_Y}&r(Y).
}
\]
Both rows specify complete fillings over the same corolla after
transport along $\beta$, so $h$ is an isomorphism. If the root
corollas differ, there are no isomorphisms. Recursion now recovers
exactly the isomorphisms of labeled trees with arbitrary child
permutations.

Finally, normal factorization, with representatives chosen for cuts
and subobjects, gives
\begin{equation}\label{concl:symmetric-hom}
 \C(X,Y)\cong
 \coprod_{K\text{ a cut of }X}\ \coprod_{p\in
 \operatorname{Sub}_{\mathrm{s}}(Y)}
 \operatorname{Iso}_{\C}(X/K,Y_p).
\end{equation}
These are exactly the triples~\eqref{concl:symmetric-matching}.
Direct image on positions preserves composition by uniqueness of
strong images. A matching's positional function determines its
target position, its cut, and its surviving tree isomorphism, so
\eqref{concl:symmetric-hom} respects composition. Complete fillings
construct every finite tree, which proves essential surjectivity.

Conversely, the triples define a category because their positional
functions compose as just described and distinguish morphisms.
Prunings are epic, being surjective on positions. A proper
occurrence is not epic, since its collapse and the corresponding
constant map agree on that subtree and differ at the ambient root.
Factorization therefore identifies the epimorphisms as prunings
followed by isomorphisms. Occurrences are strong monic: in a lifting
square, surjectivity of the epic map on positions forces the other
map's image into the specified subtree, where it restricts uniquely.
These observations identify the two intrinsic arrow classes.
Restricting a pruning gives the required
pullback, and performing it inside an ambient subtree gives the
bicartesian pushout. The surviving isomorphisms restrict in the
pullback and extend uniquely in the pushout. Filling a marked empty
leaf has the required terminal property: a comparison erases the
other cut fibres, and its surviving isomorphism is fixed on the
inserted tree and on its context. The remaining axioms follow from
finite cuts and positions, nesting of overlapping subtrees, detection
of a nonempty cut by its point fibres, and the full permutation group
of a corolla.
\end{proof}

\begin{theorem}\label{concl:symmetric-form}
Adjoining a strict initial object $0$ to
$\mathcal T_L^{\mathrm{sym}}$ gives a category with a noetherian
form having exact join decomposition. Its clusters are the marked
prunings of Definition~\ref{def:tree-clusters}, allowing an empty
marking. Its universal quotients are prunings followed by
isomorphisms, together with $\id_0$; its universal embeddings are
occurrences preceded by isomorphisms, together with all maps
$0\to X$.
\end{theorem}
\begin{proof}
These two classes form a proper factorization system. For old
objects this is the system in
Theorem~\ref{concl:symmetric-representation}; maps from $0$ factor
through $\id_0$, and strictness reduces every remaining lifting
square to an old square or one with an identity left map.

Over a fixed tree, subobjects for the stated embedding class
are its positions together with $0$, and quotients for the
stated quotient class are descendant-closed sets of erased
labeled vertices. Isomorphisms in these definitions commute
with the map to or from the fixed tree. Thus symmetric branches
still determine distinct subobjects. Intersections and least common ancestors give
the subobject lattice; intersection and union of erased sets give
the quotient lattice. Pullbacks restrict matchings to inverse-image
subtrees. Pushouts of two prunings erase the union of their erased
sets, while pushing a pruning along an occurrence performs it in
that occurrence. Their universal properties retain the isomorphism
on every surviving part, so these constructions supply the
pullbacks and pushouts required in the proof of
Theorem~\ref{thm:tree-cosub}.

For an old matching $f:X\to Y$, let $D_f$ be its erased labeled
vertices and $I_f$ its matched subtree. The surviving isomorphism
gives a bijection
$\phi_f:V_+(X)\setminus D_f\to V_+(I_f)$, including all the
specified permutations. In the quotient coordinates of
Lemma~\ref{lem:tree-form-identities},
\[
 f_*^{\mathsf Q}D=\phi_f(D\setminus D_f),\qquad
 f_{\mathsf Q}^*H=D_f\cup\phi_f^{-1}(H).
\]
Applying them successively gives $D\cup D_f$ and
$H\cap V_+(I_f)$. On subobjects, direct image after inverse image
gives intersection with $I_f$. The collapse of a subtree $A$ erases
exactly $V_+(A)$ and has $A$ as its marked fibre. Hence
$\sigma_{\alpha_X(A)}(A)=A$, and $\alpha$ preserves binary meets
and bottom. For maps from $0$ the source lattices are singletons,
and the same identities hold directly.

These are precisely the identities used in
Theorem~\ref{thm:tree-cosub}. The pairs in
\eqref{eq:tree-pairs} therefore give a noetherian form with exact
join decomposition. The condition $\alpha_X(A)\le R$ says that
$R$ erases $A$, while $\sigma_R(A)=A$ says that $A$ is the entire
fibre over its image leaf. Thus the pairs identify this form with
marked prunings, including $A=0$ for an empty marking, and identify
its universal quotients and embeddings with the stated classes.
\end{proof}

\subsection{Examples of the first six axioms}\label{concl:first-six-examples}

The first six axioms describe features shared by expression tree
categories and several familiar categories. The examples below explain
why the last three axioms have a separate role. In particular, the
first six axioms allow both arbitrary functions between sets and
categories having at most one morphism between any two objects.
For these examples we consider the six structural axioms without a
smallness restriction; a small skeleton is used when the finiteness
axiom is discussed.

We recall the terminology needed for the examples. An
\emph{elementary topos} is a category with finite limits, exponential
objects, and a subobject classifier: its subobjects are described by
characteristic morphisms to an object $\Omega$
\cite{MacLaneMoerdijk}. An object $X$ is \emph{inhabited} if
$X\to1$ is epic. A topos is \emph{two-valued} if the only subobjects
of $1$ are $0$ and $1$, and \emph{nondegenerate} if these are
distinct. In a nondegenerate two-valued topos, the inhabited objects
are precisely the noninitial objects. A \emph{pointed object} is an
object $X$ with a specified morphism $x:1\to X$; morphisms between
pointed objects preserve the specified points.

\begin{theorem}\label{concl:topos-examples}
The following categories satisfy Axioms~\ref{ax:N}--\ref{ax:U}:
\begin{enumerate}
\item the category of pointed objects of an elementary topos;
\item the full subcategory of inhabited objects of a nondegenerate
two-valued elementary topos.
\end{enumerate}
\end{theorem}
\begin{proof}
We use three standard properties of toposes. Every morphism factors
as an epimorphism followed by a monomorphism; epimorphisms are
coequalizers and are stable under pullback; and pushouts along
monomorphisms are pullbacks, preserve monomorphisms, and are stable
under pullback \cite{MacLaneMoerdijk,Lack}. We first explain why
these properties apply to the intrinsic morphism classes in the
two categories in the statement.

Their epimorphisms are precisely the underlying epimorphisms. Indeed,
if a morphism has a proper image in $Y$, the characteristic morphism
of that image and the constant true morphism $Y\to\Omega$ are
distinct and agree after the given morphism. The object $\Omega$ is
inhabited, and in the pointed case both morphisms preserve the point
when $\Omega$ is pointed by true. Monicity is also reflected. In
(b), a witness to failure of monicity has a noninitial, hence
inhabited, domain. In (a), a witness with domain $Z$ can be given a
disjoint basepoint, producing a pointed witness with domain $1+Z$.
Ambient image factorizations consequently restrict to these
categories, and their strong monomorphisms are the underlying
monomorphisms. This proves Axioms~\ref{ax:N} and~\ref{ax:F}.

The pullbacks in Axiom~\ref{ax:Dpull} have the required type.
In (b), their objects remain inhabited because they cover an
inhabited object; in (a), they carry the induced point. The
pushouts in Axiom~\ref{ax:Dpush} likewise remain in the indicated
category and have their ambient universal properties. The standard
properties just recalled prove both axioms.

In (a), every object has exactly one categorical point. Thus every
epimorphism belongs to the class in Axiom~\ref{ax:pointpushout},
and the axiom follows from preservation of epimorphisms under
pushout. For (b), consider a pushout as in
Axiom~\ref{ax:Dpush}, and a point $k:1\to Y$ of its bottom-right
object. Pulling back the bottom monomorphism $B\to Y$ gives a
subobject of $1$. If it is $1$, then $k$ factors uniquely through
$B$, and its fibre is the corresponding fibre of $A\to B$.
If it is $0$, pullback stability of the pushout shows that the
fibre of $X\to Y$ over $k$ is terminal. Two-valuedness exhausts
the possibilities. Hence at most one exceptional point before the
pushout gives at most one exceptional point afterwards.

For Axiom~\ref{ax:U}, let $j:1\to B$ be the given point.
In a topos the pullback functor from objects over $B$ to objects
over $1$ has a right adjoint, denoted by $\Pi_j$. This adjoint
constructs an object over $B$ from a prescribed fibre: morphisms
to $\Pi_j(A)$ over $B$ correspond exactly to morphisms from the
fibre over $j$ to $A$. Write $\Pi_j(A)=(d:C\to B)$.
Since $j$ is monic, the adjunction identifies the fibre of $d$
with $A$, giving the pullback square
\[
\xymatrix@C=3.5em@R=2.3em{
A\ar[r]^u\ar[d]_{q_A}&C\ar[d]^d\\
1\ar[r]_j&B.
}
\]
The morphism $d$ is epic. In (b), the identity of $A$, regarded
as a morphism from the fibre of $A\times B\to B$, corresponds
to a morphism $A\times B\to C$ over $B$. Its composite with
$d$ is the epic projection $A\times B\to B$. In (a), the
chosen point of $A$ gives a section of $d$ by the same adjunction;
give $C$ the point obtained by composing that point with $u$.
For any competing pullback factorization, the comparison to $C$
must restrict to the identity on its specified fibre $A$.
The adjunction gives exactly one such comparison. In (a) it
preserves the point as well. This proves terminality and completes
the proof.
\end{proof}

A similar argument applies when complements replace the right
adjoint used above. Recall that a \emph{regular epimorphism} is a
coequalizer of a pair of morphisms. A \emph{pretopos} has finite
limits, disjoint finite coproducts stable under pullback, and stable
regular-epimorphism--monomorphism factorizations; moreover, every
internal equivalence relation is the kernel pair of its quotient.
Here an internal equivalence relation is an equivalence relation
expressed by a subobject of $X\times X$. A pretopos is
\emph{Boolean} if every subobject has a complement: a monomorphism
$A\to X$ then identifies $X$ with a coproduct $A+D$.

\begin{theorem}\label{concl:pretopos-examples}
The category of pointed objects of a Boolean pretopos satisfies
Axioms~\ref{ax:N}--\ref{ax:U}. The same holds for the full
subcategory of noninitial objects of a nondegenerate two-valued
Boolean pretopos.
\end{theorem}
\begin{proof}
In the second case, the image of $X\to1$ is $1$ whenever
$X$ is noninitial, so these morphisms are regular epimorphisms.
In both cases the intrinsic epimorphisms are the underlying regular
epimorphisms. To detect a proper image, use its complement to
construct two distinct morphisms to $1+1$ agreeing on that image;
in the pointed case both preserve the first summand as basepoint.
Monomorphisms and image factorizations restrict as in
Theorem~\ref{concl:topos-examples}. Stability of regular
epimorphisms under pullback proves Axiom~\ref{ax:Dpull}.

For the pushout axiom, write $X=A+D$ using the complement of
the given monomorphism. The required square is
\[
\xymatrix@C=3.5em@R=2.3em{
A\ar[r]\ar[d]_a&A+D\ar[d]^{a+\id_D}\\
B\ar[r]&B+D.
}
\]
Disjointness and stability of coproducts make it a pullback as
well as a pushout. In the noninitial case, every point of $B+D$
factors through exactly one summand: its inverse images give a
coproduct decomposition of $1$, and two-valuedness makes one
summand initial. The fibres are therefore fibres of $a$ or
terminal objects. This proves Axiom~\ref{ax:pointpushout};
in the pointed case, the uniqueness of categorical points proves it.

Finally, complement the given point $j:1\to B$ and write
$B=1+D$. The terminal pullback factorization has middle object
$A+D$ and right morphism $q_A+\id_D$. This morphism is epic;
in the pointed case it is split epic. A competing factorization
decomposes over the two summands as $C'=A+E$, with its map to
$B$ of the form $q_A+g$ for $g:E\to D$. The unique comparison
is $\id_A+g$. It also preserves the point when points are
specified. This proves Axiom~\ref{ax:U}.
\end{proof}

These results include nonempty sets, finite nonempty sets, pointed
sets, and finite pointed sets. For any group $G$, they also include nonempty $G$-sets
and their finite counterparts, and pointed objects of categories
of presheaves and sheaves. For the finite $G$-set example, finite
limits, disjoint unions, invariant complements, and quotients by
invariant equivalence relations give a Boolean pretopos, and its
terminal singleton has exactly two subobjects.

\begin{remark}\label{concl:inhabited-caution}
The two-valued hypothesis cannot simply be omitted in the
unpointed case. In $\mathbf{Set}\times\mathbf{Set}$, take the
epimorphism $(2,1)\to(1,1)$ between inhabited objects and push it
out along the summand inclusion
$(2,1)\to(2,1)+(1,1)$. The resulting epimorphism has target
$(2,2)$. Its fibres over the two points whose first coordinate is
in the first summand are both $(2,1)$, hence nonterminal.
The original target has just one point. Thus
Axiom~\ref{ax:pointpushout} fails.
\end{remark}

\begin{lemma}\label{concl:poset-examples}
Every poset with a greatest element, regarded as a category,
satisfies Axioms~\ref{ax:N}--\ref{ax:U} and~\ref{ax:O}.
It satisfies Axiom~\ref{ax:finite} exactly when every upper
interval is finite, and Axiom~\ref{ax:Ddetect} exactly when
the elements other than the greatest element form an antichain.
\end{lemma}
\begin{proof}
Every morphism is epic. A strong monomorphism must therefore be
invertible, and hence an identity. The required factorizations
are a morphism followed by an identity; the required pullbacks,
pushouts, and diagonals are consequently those involving
identities. Every object has at most one point, so
Axiom~\ref{ax:pointpushout} holds. A point $1\to B$ exists
only when $B=1$. The filling category then has the single
object $(A,q_A,\id_A)$, proving Axiom~\ref{ax:U}.

The epimorphisms out of $x$ correspond to the elements above
$x$, proving the finiteness assertion. A nonidentity morphism
with codomain below $1$ violates Axiom~\ref{ax:Ddetect}, since
that codomain has no points. For a morphism $x\to1$, its
pullback along the unique point of $1$ is the morphism itself,
so the antecedent of the detection axiom holds only for the
identity. This proves the last assertion. In particular, the
three-element chain satisfies the first eight axioms and fails
the ninth.
\end{proof}

\begin{remark}\label{concl:later-axiom-failures}
Small skeletons of finite nonempty sets and finite pointed sets
satisfy Axiom~\ref{ax:finite}: a fixed finite object admits
only finitely many incoming injections and outgoing surjections
to the chosen representatives. Axiom~\ref{ax:O} fails because
the subsets $\{0,1\}$ and $\{1,2\}$ of $\{0,1,2\}$ have
a nonempty intersection and neither contains the other. The
pointed version uses $\{*,a\}$ and $\{*,b\}$ inside
$\{*,a,b\}$. Nonidentity permutations violate the literal
identity conclusion in Axiom~\ref{ax:Ddetect}. In pointed
sets there is also a noninvertible counterexample: the surjection
$\{*,a,b\}\to\{*,c\}$ sending $a,b$ to $c$ has terminal
fibre over the unique categorical point. Nonempty sets and pointed
sets with infinite objects allowed fail Axiom~\ref{ax:finite} as well.
\end{remark}

\begin{remark}\label{concl:initial-and-abelian-obstructions}
Taking $j=\id_1$ in Axiom~\ref{ax:U} forces $q_A$ to be
epic: the top morphism of the pullback square is an isomorphism,
and its right morphism is epic. Thus a nondegenerate topos with
its initial object included fails this axiom.

Pointedness and exactness do not by themselves supply the sixth
axiom either. In a nonzero abelian category, choose $A\ne0$.
A pullback factorization for $A\to0\to A$ is an extension
$A\xrightarrow{u}C\xrightarrow{d}A$, with $u$ its kernel
inclusion and $d$ epic. The automorphism $\id_C+ud$ fixes
both maps and has inverse $\id_C-ud$, since $du=0$.
It is nonidentity, by monicity of $u$ and epicity of $d$.
Consequently no object of this filling category can be terminal,
although nonzero abelian categories satisfy the first five axioms.
\end{remark}

\section*{Index of notation}\label{sec:notation-index}
\addcontentsline{toc}{section}{Index of notation}
The entries are grouped by their use in the paper. Each reference
links to the introduction of the notation or to the indicated
definition of its abstract counterpart.

\begingroup
\small
\renewcommand{\arraystretch}{1.15}
\setlength{\LTpre}{8pt}
\setlength{\LTpost}{0pt}
\begin{longtable}{@{}>{\raggedright\arraybackslash}p{.23\textwidth}
                    >{\raggedright\arraybackslash}p{.52\textwidth}
                    >{\raggedright\arraybackslash}p{.19\textwidth}@{}}
\textbf{Notation}&\textbf{Meaning}&\textbf{Reference}\\
\hline
\endfirsthead
\textbf{Notation}&\textbf{Meaning}&\textbf{Reference}\\
\hline
\endhead
\multicolumn{3}{r}{\emph{Continued on the next page}}\\
\endfoot
\endlastfoot
\multicolumn{3}{@{}l}{\textbf{The category and its tree representation}}\\[3pt]
$\one$, $q_X$ & The terminal object and the unique morphism
  from $X$ to it. & \hyperref[ax:N]{Section~\ref*{sec:axiom-list}}\\
$\Fill(A,j)$ & The category of pullback factorizations specifying
  an insertion of $A$ at the point $j$. & \hyperref[ax:U]{Axiom~\ref*{ax:U}}\\
$L$, $\operatorname{ar}$, $P_\lambda$ & A ranked set, its arity
  function, and the indexed input places of a label.
  & \hyperref[def:ranked-set]{Definition~\ref*{def:ranked-set}}\\
$\C_L$, $\Dcat_0$, $\Ocat_0$ & The concrete matching category,
  its pruning maps, and its subtree embeddings.
  & \hyperref[sec:model]{Section~\ref*{sec:model}}\\
$\Dcat$, $\Ocat$ & The subcategories of epimorphisms and strong
  monomorphisms in $\C$.
  & \hyperref[sec:strong-positions]{Section~\ref*{sec:strong-positions}}\\
$\Occ(X)$, $\Out(X)$ & All strong monomorphisms into $X$ and all
  epimorphisms out of $X$, respectively.
  & \hyperref[sec:strong-positions]{Section~\ref*{sec:strong-positions}}\\
$\Comp(X)$, $X_i$ & The indecomposable strong monomorphisms into
  $X$, and the domain of a component $i$.
  & \hyperref[def:components]{Definition~\ref*{def:components}}\\
$u\inside v$ & The tree branching order: $u$ factors through $v$.
  & \hyperref[prop:paths]{Lemma~\ref*{prop:paths}}\\
$d|_u$, $u^d$ & The epimorphism and strong monomorphism in the
  normal factorization of $du$.
  & \hyperref[eq:residual]{Equation~\ref*{eq:residual}}\\
$\pi_u$, $X/u$ & The pruning which replaces the subtree selected
  by $u$ by an empty leaf, and its target.
  & \hyperref[sec:single-pruning]{Section~\ref*{sec:single-pruning}}\\
$d^*v$ & The pullback of a strong monomorphism $v$ along an
  epimorphism $d$.
  & \hyperref[lem:strongpullbacks]{Lemma~\ref*{lem:strongpullbacks}}\\
$\Icat$ & The subcategory of epimorphisms with nonterminal target,
  together with $\id_\one$.
  & \hyperref[sec:epi-restrictions]{Section~\ref*{sec:epi-restrictions}}\\
$\sigma_d$, $\Theta_X$ & The bijection of components induced by
  $d$, and the description of an epimorphism by its component
  restrictions or total pruning.
  & \hyperref[thm:epicclassification]{Theorem~\ref*{thm:epicclassification}}\\
$c(X)$ & The complexity of an object, equal to $|\Occ(X)|-1$.
  & \hyperref[def:complexity-normal]{Definition~\ref*{def:complexity-normal}}\\
$\rho_X$, $r(X)$ & The normalization which empties every child
  at the root, and its normal target.
  & \hyperref[thm:root]{Theorem~\ref*{thm:root}}\\
$\Phi$, $P_\phi$ & The derived root labels and the input-place
  set $\Comp(\phi)$ of a label.
  & \hyperref[def:labels]{Definition~\ref*{def:labels}}\\
$i_p^X$, $X_p$, $d_p$ & The component at the root place $p$,
  its domain, and the restriction of $d$ at this component.
  & \hyperref[eq:ports]{Section~\ref*{sec:roots}}\\
$\Pos(X)$, $X|_a$, $\emptyword$ & Positions represented by
  addresses, the subtree at address $a$, and the root address.
  & \hyperref[sec:model]{Section~\ref*{sec:model}};
  \hyperref[def:addresses]{Definition~\ref*{def:addresses}}\\
$\delta_K$, $X_K$ & The epimorphism specified by an effective
  cut $K$, and its target.
  & \hyperref[thm:cuts]{Theorem~\ref*{thm:cuts}}\\
$\operatorname{outer}(K\cup L)$ & The union of two cuts with
  addresses having a shorter selected prefix removed.
  & \hyperref[prop:cutcomposition]{Lemma~\ref*{prop:cutcomposition}}\\
$p_d$, $p_u$ & The functions on positions induced by a pruning
  and an occurrence; composition gives the action of a matching.
  & \hyperref[not:position-maps]{Section~\ref*{sec:cuts}}\\
$\Fill((A_i,j_i)_{i\in I})$ & The category of simultaneous
  fillings of the objects $A_i$ at the distinct points $j_i$.
  & \hyperref[def:finitefilling]{Definition~\ref*{def:finitefilling}}\\
$\phi((A_p)_{p\in P_\phi})$ & The object assembled with root
  label $\phi$ and child $A_p$ at each input place $p$.
  & \hyperref[not:assembly-operation]{Section~\ref*{sec:filling}}\\
$f_*u$, $\le_X$ & Direct image of a position under a matching,
  and the uniform linear order on positions.
  & \hyperref[not:occurrence-images]{Section~\ref*{sec:ordering}}\\
$\mathcal O_\C$, $\gamma$ & The recovered free non-symmetric
  operad and its composition by simultaneous filling.
  & \hyperref[thm:operad-recovery]{Theorem~\ref*{thm:operad-recovery}}\\
$\mathcal B_\Sigma$, $s(E)$, $L_\Sigma$ & Balanced bracketed
  expressions, the root pattern of an expression, and the set of
  nonempty root patterns.
  & \hyperref[def:bracketed]{Definition~\ref*{def:bracketed}}\\
$\mathcal E_\Sigma$, $\tau$ & The category of bracketed
  expressions and the correspondence with its expression trees.
  & \hyperref[not:bracket-category]{Section~\ref*{sec:bracketed}}\\
$\Sigma^*$, $L^*$ & The sets of finite words over the respective
  alphabets, including the empty word.
  & \hyperref[def:bracketed]{Section~\ref*{sec:bracketed}};
  \hyperref[def:wordcategory]{Section~\ref*{sec:words}}\\
$\mathcal W_L$, $f^r_{u,v}$ & The word category on $L$ and a
  matching with overlap length $r$.
  & \hyperref[def:wordcategory]{Definition~\ref*{def:wordcategory}}\\
$\varepsilon$, $|u|$, $\operatorname{pre}_r(u)$, $\operatorname{suf}_r(u)$
  & The empty word, word length, and the prefix and suffix of
  length $r$.
  & \hyperref[def:wordcategory]{Definition~\ref*{def:wordcategory}}\\[6pt]
\multicolumn{3}{@{}l}{\textbf{The form and its subobject and quotient data}}\\[3pt]
$\widehat{\mathcal T}$, $0$ & The tree category with an adjoined
  strict initial object, representing the tree with no root.
  & \hyperref[not:augmented-tree-category]{Section~\ref*{sec:tree-form-construction}}\\
$\mathcal E$, $\mathcal M$ & In the augmented category, prunings
  together with $\id_0$, and occurrences together with all maps
  from $0$.
  & \hyperref[not:augmented-tree-classes]{Section~\ref*{sec:tree-form-construction}}\\
$\mathsf S(X)$, $\mathsf Q(X)$ & The lattices of
  $\mathcal M$-subobjects and $\mathcal E$-quotients.
  & \hyperref[not:tree-subobjects-quotients]{Section~\ref*{sec:tree-form-construction}}\\
$[q,u]$, $\le_f$ & A quotient with an empty marking or a marked
  empty leaf, and a comparison of clusters over $f$.
  & \hyperref[def:tree-clusters]{Definition~\ref*{def:tree-clusters}}\\
$\mathsf G$, $\mathsf G(X)$ & The form of these clusters and its
  fibre over $X$.
  & \hyperref[def:tree-clusters]{Definition~\ref*{def:tree-clusters}}\\
$V_+(X)$ & The labeled vertices of $X$.
  & \hyperref[not:labeled-vertices]{Section~\ref*{sec:tree-form-images}}\\
$f_*^{\mathsf S}$, $f_{\mathsf S}^*$ & Direct and inverse images
  of subobjects, defined by factorization and pullback.
  & \hyperref[not:subobject-quotient-images]{Section~\ref*{sec:tree-form-images}}\\
$f_*^{\mathsf Q}$, $f_{\mathsf Q}^*$ & Direct and inverse images
  of quotients, defined by pushout and factorization.
  & \hyperref[not:subobject-quotient-images]{Section~\ref*{sec:tree-form-images}}\\
$I_f$, $K_f$ & The subobject and quotient represented by the
  two parts of the factorization of $f$.
  & \hyperref[not:matching-factor-parts]{Section~\ref*{sec:tree-form-images}}\\
$\alpha_X(A)$ & The quotient which collapses the subtree $A$;
  the empty subtree gives the bottom quotient.
  & \hyperref[not:subtree-collapse]{Section~\ref*{sec:tree-form-images}}\\
$\sigma_R(A)$ & The inverse image of the image of $A$ under
  the quotient $R$.
  & \hyperref[not:subtree-saturation]{Section~\ref*{sec:tree-form-images}}\\
$f_*L$, $f^*N$ & Direct and inverse images of clusters in a form.
  & \hyperref[not:cluster-images]{Section~\ref*{sec:noetherian-property}}\\
$\operatorname{Ker}f$, $\operatorname{Im}f$ & The kernel and
  image clusters, obtained as $f^*(\bot)$ and $f_*(\top)$.
  & \hyperref[not:form-kernel-image]{Section~\ref*{sec:noetherian-property}}\\
$n(L)$, $c(L)$ & The largest normal and conormal clusters
  contained in $L$.
  & \hyperref[not:normal-conormal-parts]{Section~\ref*{sec:noetherian-property}}\\
$(A,R)$ & A cluster described by a saturated subtree and a
  quotient that erases its labeled vertices.
  & \hyperref[eq:tree-pairs]{Equation~\ref*{eq:tree-pairs}}\\
$\mathfrak n_X(R)$, $\mathfrak c_X(A)$ & The normal cluster
  $(0,R)$ and the conormal cluster $(A,\alpha_X(A))$.
  & \hyperref[prop:cosub-join-parts]{Lemma~\ref*{prop:cosub-join-parts}}\\
$\mathcal C^+$ & The full subcategory of objects not isomorphic
  to the strict initial object.
  & \hyperref[not:nonempty-subcategory]{Section~\ref*{sec:tree-form-characterization}}\\
\end{longtable}
\endgroup

\end{document}